\documentclass[10pt]{article}

\usepackage[top=1in,bottom=1in,left=1.25in,right=1.25in]{geometry}

\usepackage[absolute,overlay]{textpos}

\usepackage{graphicx}
\usepackage{ragged2e}
\usepackage{tabto}

\usepackage{amsmath}
\usepackage{amsthm}
\usepackage{amsfonts}
\usepackage{amssymb}

\usepackage{float}
\usepackage{multirow}
\usepackage{tikz}
\usepackage{graphics}
\usepackage{color}
\usepackage{url}
\usepackage{afterpage}
\usepackage{multicol}
\usepackage[font=footnotesize]{caption}
\usepackage[font=footnotesize]{subcaption}
\usepackage{marginnote}
\usepackage{tcolorbox}

\usepackage{lmodern}

\usepackage[sf,bf,medium]{titlesec}
\usepackage{pgfplots}

\usepackage{bm}
\usepackage{booktabs}
\usepackage{siunitx}
\usepackage{isomath}

\usepackage[plainpages=false, colorlinks=true, citecolor=blue, filecolor=blue,
   linkcolor=blue, urlcolor=blue]{hyperref}
\usepackage{cleveref}

\newcommand{\opnm}{\operatorname}
\newcommand{\opd}{{\rm d}}

\newcommand{\lp}{\left(}
\newcommand{\rp}{\right)}

\newcommand\bbR{\mathbb R}
\newcommand\bbC{\mathbb C}

\newcommand\bx{\boldsymbol{x}}

\newcommand\by{\boldsymbol{y}}
\newcommand\bz{\boldsymbol{z}}

\newcommand\bn{\boldsymbol n}

\newtheorem{theorem}{\sffamily Theorem}
\newtheorem{remark}{\sffamily Remark}
\newtheorem{definition}{\sffamily Definition}
\newtheorem{assumption}{\sffamily Assumption}

\newtheorem{lemma}{\sffamily Lemma}
\newtheorem{proposition}{\sffamily Proposition}
\newtheorem{corollary}{\sffamily Corollary}[theorem]

\newcommand{\cD}{\mathcal D}
\newcommand{\cS}{\mathcal S}

\newcommand{\cI}{\mathcal I}

\newcommand{\cC}{\mathcal C}

\newcommand{\cK}{\mathcal K}

\newcommand{\cG}{\mathcal G}

\newcommand{\GammaC}{\Gamma_{\bbC}}

\numberwithin{equation}{section}
\newcommand{\tGamma}{\widetilde\Gamma}
\newcommand{\thetaeta}{\theta_\eta}

\newcommand{\bs}{\boldsymbol}

\newcommand{\figref}[1]{\figurename~\hyperref[#1]{\ref{#1}}}
\newcommand{\tableref}[1]{\tablename~\hyperref[#1]{\ref{#1}}}

\renewcommand{\phi}{\varphi}

\binoppenalty=\maxdimen
\relpenalty=\maxdimen

\newcommand{\imagdecay}[2]{\cC^{#1,#2}}
\newcommand{\realdecay}[1]{\cD^{#1}}
\newcommand{\Kslope}{K_{\opnm{slope}}}

\newcommand{\Sdiff}{B}
\newcommand{\Spdiff}{D}
\newcommand{\Ddiff}{A}
\newcommand{\Dpdiff}{C}

\usepackage{abstract}

\begin{document}

{\centering {\Large Complex Scaling for the Junction of Semi-infinite Gratings\\}
\vspace{0.1cm}

 \large Fruzsina J. Agocs\footnote{Department of Computer Science, University of Colarado, Boulder, Boulder, CO 80309 USA, fruzsina.agocs@colorado.edu}, Tristan Goodwill\footnote{Department of Statistics and CCAM, University of Chicago,
 Chicago, IL 60637 USA, tgoodwill@uchicago.edu}, and Jeremy Hoskins\footnote{Department of Statistics and CCAM, University of Chicago,
 Chicago, IL 60637 USA, jeremyhoskins@uchicago.edu}\\
}

\begin{abstract}
We present and analyze an integral equation method for the scattering of a non-periodic source from a geometry consisting of two semi-infinite, periodic structures glued together in two dimensions. The two structures may involve a periodic wall, several layers of transmission surfaces with a shared period, or periodic sets of obstacles. This integral equation is posed on the infinite interface between the two periodic structures using kernels built out of the Green's function for each structure. To combat the slow decay of the Green's function, we also show that our integral equation can be analytically continued into the complex plane, where it can be truncated with exponential accuracy. 
A careful analysis of the domain Green's functions far from the periodic structure is then used to prove that the analytically continued equation is Fredholm index zero. Finally, we show that the solution we generate satisfies a radiation condition and demonstrate an efficient and high order solver for this problem.

\end{abstract}
  

\section{Introduction}

Periodic structures are widely used in the construction of acoustic and electromagnetic devices. Among the first of these were diffraction gratings~\cite{bornwolf,diff_grat}, particularly metallic or semiconductor gratings, which were found to be highly absorbing for certain narrow wavelength bands of incident light. This was exploited to optimize gratings for the filtering or propagation of optical signals, such as for diffraction grating spectrometers~\cite{enoch2012plasmonics}.
Small, periodic holes of sub-wavelength size on a metallic surface have been shown to allow significantly enhanced transmission of light than what would be expected for a single sub-wavelength hole~\cite{ebbesen1998extraordinary}. This phenomenon, called extraordinary optical transmission, is used in scanning electron microscopy and to manufacture structures with sub-wavelength features (sub-wavelength photo-lithography).
More recent technology allows photonic crystals with periodic nanostructures to be designed to achieve specific optical properties~\cite{phot_cryst,zayats2005nano}.
Periodic arrays of scatterers (e.g. phononic crystals) have been used in acoustics for the absorption of phononic waves ranging in frequency from seismic to radio~\cite{acoust_absorp}. Human-scale periodic structures are commonly used in architecture to enhance sound conductivity, e.g. in the design of amphitheaters. 

One particularly useful feature of periodic structures is to support trapped modes that are confined near the periodic boundary and propagate along it. Trapped acoustic or elastic modes are localized solutions of the wave equation without sources. They have been studied since at least the 1940s \cite{fano1941anomalous,jones1953eigenvalues} (see \cite{linton2007embedded} for a comprehensive review), and are exploited in applications involving sensing, filtering, and nondestructive measurements. Open, periodic geometries with troughs support the existence of modes that are exponentially decaying perpendicular to, and propagating along, the scattering surface. These structures therefore act as open waveguides for frequencies at which trapped modes exist. 
Trapped modes in general may exist at low or high frequencies \cite{pagneux2013trapped}. The latter are embedded in a continuous spectrum of radiating waves and are thus called bound states in the continuum (BICs) \cite{hsu2016bound}. They were originally proposed in the context of quantum systems, but have since been observed in acoustic, electromagnetic, and water waves.  
The periodic geometries described above have been shown to support low-frequency trapped modes with a maximum associated frequency, which is comparable to $\pi/d$, where $d$ is the period of the boundary \cite{dhia2007resonances}. 
These low-frequency modes that arise from fluid-solid interactions are of practical importance to e.g.\ floating offshore platforms such as bridges and airports \cite{mciver2000water}. 
For this reason, and for ease of exposition, we focus on wavenumbers below $\pi/d$ for the majority of this paper.

Scattering surfaces consisting of multiple periodic segments (which may be copies rotated with respect to each other or have different unit cells altogether) are also of practical importance and give rise to phenomena that single periodic structures do not exhibit.
Examples of these geometries include crystalline surfaces with multiple domain or inclusions, and randomly rough surfaces (modeled as multiple periodic segments, such as in \cite{enoch2012plasmonics}). Randomly rough surfaces are of particular interest because of Anderson (strong) localization: they can absorb a surface plasmon polariton (generated by a periodic metallic surface) propagating towards them and produce ``hot spots", localized electromagnetic surface modes, on their surface. A commonly studied geometry in this context involves a periodic grating interfacing with a randomly rough surface, with a flat surface separating the two. Other interfaces between periodic structures are also worth studying for the interference effects they produce that may help with the control of optical signals, which motivates our choice of geometries to study.

A large number of numerical techniques have been developed to compute scattering from a single periodic domain. These methods can generally be split into two categories. The first collection of methods compute quasi-periodic fields induced by a plane wave incident on the wall. Some do this using quasi-periodic Green's functions, such as~\cite{desanto1998theoretical,arens2006integral,pinto2021fast,meng2023new}, or a windowed approximation of the quasi-periodic Green's function~\cite{bruno2014rapidly,bruno2017rapidly,perez2018domain}. Other methods enforce the quasi-periodicity as a constraint~\cite{barnett2011new,zhang2021fast,zhang2022fast,gillman2013fast,strauszer2023windowed} or build Rayleigh expansions for the solution \cite{petit1980electromagnetic,millar1973rayleigh}. There has also been extensive work on the high-order perturbation of surfaces method, which builds an expansion for the solution as a perturbation from a flat boundary (see \cite{kehoe2023joint,nicholls2025analyticity} and the references therein). 
The second set of methods look for aperiodic solutions and are usually based on the inverse Floquet--Bloch transform (also known as the array scanning method) and include works such as \cite{munk_plane-wave_1979,rana1981current,capolino2005mode,lechleiter2017convergent,agocs2024trapped,zhang2021numerical}.
Note that the latter set of approaches require expressing the solution of the aperiodic problem in terms of a family of quasiperiodic solutions, whose symmetry may then be exploited to reduce the domain to a single unit cell. Naive truncation of the domain is not possible due to the artificial reflections of trapped modes it would cause, which would result in $\mathcal{O}(1)$ error near the boundary.

In parallel to this work, a number of methods have been developed for understanding the effect of defects in the periodic structure. Such defects are known to introduce large localization effects that change the response of periodic systems considerably \cite{ammari}. One example of a computational work studying periodic systems with defects is~\cite{kirsch2025pml}, which used a Floquet--Bloch transform and treated the defect as a coupling between quasi-periodicities. Each quasi-periodic problem was solved using a hybrid spectral-finite difference method and truncated at a finite height using a perfectly matched layer. In~\cite{watanabe2012accurate}, the authors used used a Floquet--Bloch transform and the recursive transition-matrix algorithm to compute the scattering from a periodic line of circular scatterers with some deletions. Other methods for handling defects include the fictitious supercell methods used in \cite{liu2017electromagnetic} and a method based on matching Bessel (also known as cylindrical) expansions \cite{liu2018fast,liu2018electromagnetic,liu2020electromagnetic}. There has also been much work on developing approximations for the effect of defects. See \cite{virk2021fast} and the references therein.

Recently, there has also been considerable interest in the effect of changes in the periodic structure. These methods are usually based on gluing Poincar\'e--Steklov operators for a semi-infinite periodic half space. In \cite{pierre}, the authors consider the junction of two doubly periodic lossy materials and build a Dirichlet-to-Neumann operator for each half-space. They do this by taking a Floquet--Bloch transform and lifting the resulting quasi-periodic problem to a higher dimensional periodic one. They then construct the Dirichlet-to-Neumann operator for a single unit cell and solve a Riccati equation for the half-space Dirichlet-to-Neumann operator. More recently, \cite{turc2025} considered a semi-finite array of compact scatterers. In that work, the author constructed a Robin-to-Robin map for single unit cell and solved a Riccati equation to build a Robin-to-Robin map for the periodic half-space. Other works based on this approach include~\cite{fliss2009exact,fliss2010exact,amenoagbadji2023wave,fliss2020time,fliss2019wave,klindworth2014numerical}. 

There has also been interest in the coupling of closed periodic waveguides. This problem has been considered in various works including \cite{fliss2021dirichlet}, which used the Poincar\'e--Steklov approach described above, and \cite{qiu2023bifurcation}, which used the domain Green's function for each semi-infinite waveguide to build an integral equation at their junction.

In this work, we adapt the method presented in \cite{epstein2023solvinga,epstein2023solvingb,epstein2024solving} to simulate the junction of two parallel semi-infinite periodic gratings. Our approach is simple and computationally efficient, avoiding the need to solve costly Riccati equation for the half-space Poincar\'e--Steklov operators. It is directly applicable to problems with real wavenumbers and easily extendable to the case of gratings connected by a compact transition region. In this method, we express the scattering problem as a transmission problem connecting the left half to the right half of the infinite domain. We then use the domain Green's function for each periodic problem to convert this transmission problem into an integral equation on the fictitious interface between the half-spaces. We compute these Green's functions using the method presented in \cite{agocs2024trapped}. 

A key feature of our method is the use of complex scaling to mitigate the slow decay of the densities and kernels of this integral equation away from the boundaries. In particular, we adapt the complex scaling approach analyzed in~\cite{goodwill2024numerical,epstein2025complex}, and show that both can be analytically continued to functions that decay exponentially in the complex plane. We then show that the analytically continued operator is Fredholm index zero and that the equation can be solved on a single choice of contour. As the kernels and densities decay exponentially, the resulting integral equation can be truncated with controllable accuracy. This analysis is of independent interest, as it demonstrates how to study the matched complex scaling method when the integral equation contains integral operators on the diagonal and when the range of the integral operators contains a mixture of oscillatory and exponentially decaying functions.

The remainder of the paper is structured as follows. In Section~\ref{sec:transm_prob} we describe the process for converting this problem into an integral equation on a subset of the~$x_2$-axis. In Section~\ref{sec:Green_period_def} we define the domain Green's functions for each half. In Section~\ref{sec:glued_IE_anal} we analyze the glued integral equation and show that it can be analytically continued to an integral equation with an operator that is Fredholm index zero. In Section~\ref{sec:recovered_sol} we show that the solution of the integral equation gives a solution of the PDE that satisfies the Sommerfeld radiation condition in any cone that does not include either grating. We also discuss physically meaningful data for our integral equation. In Section~\ref{sec:numerics} we illustrate the approach with several numerical examples. Finally, in Section~\ref{sec:conclusion} we finish with some concluding remarks and directions for future research.

\section{An integral equation formulation}\label{sec:transm_prob}

Let~$\gamma_{L,R}$ be two two-dimensional, periodic boundaries with periodicities $d_{Lp,Rp}$ and unit ``lattice vector'' $\bs e_1$, i.e.
\begin{equation}
    \forall \bx \in \gamma_{L,R} \Rightarrow \bx + d_{L,R}\bs e_1 \in \gamma_{L,R},
\end{equation}
and let the coordinates be $\bx = (x_1, x_2)$ along the boundaries and perpendicular to them, respectively, so that $ \bs e_1 = (1, 0)$. In the following we let~$\Omega_{L,R}$ denote the regions above~$\gamma_{L,R}$. 
For simplicity, we focus on the case that both~$\gamma_L$ and~$\gamma_R$ only touch the~$x_2$-axis at~$(0,X_2)$. We also assume that both are flat in a neighborhood of that point and that their slopes agree and aren't vertical. The case where one or both of $\gamma_{L,R}$ are not flat at the~$x_2$-axis can be studied using similar techniques, but careful analysis would be required to understand the singularities of the solution at~$(0,X_2)$.

We let 
\begin{equation}
    \Theta = \lp\Omega_L \cap \{x_1\leq 0\}\rp\cup \lp\Omega_R \cap \{x_1\geq 0\}\rp
\end{equation}
be the domain above~$\gamma_L$ and~$\gamma_R$ in the left and right half spaces respectively (see~\figref{fig:prob_setup}). Without loss of generality, we shall suppose that both~$\gamma_L$ and~$\gamma_R$ lie in the region~$x_2\leq 0$.

\begin{figure}
    \centering
    \includegraphics[width=0.5\linewidth]{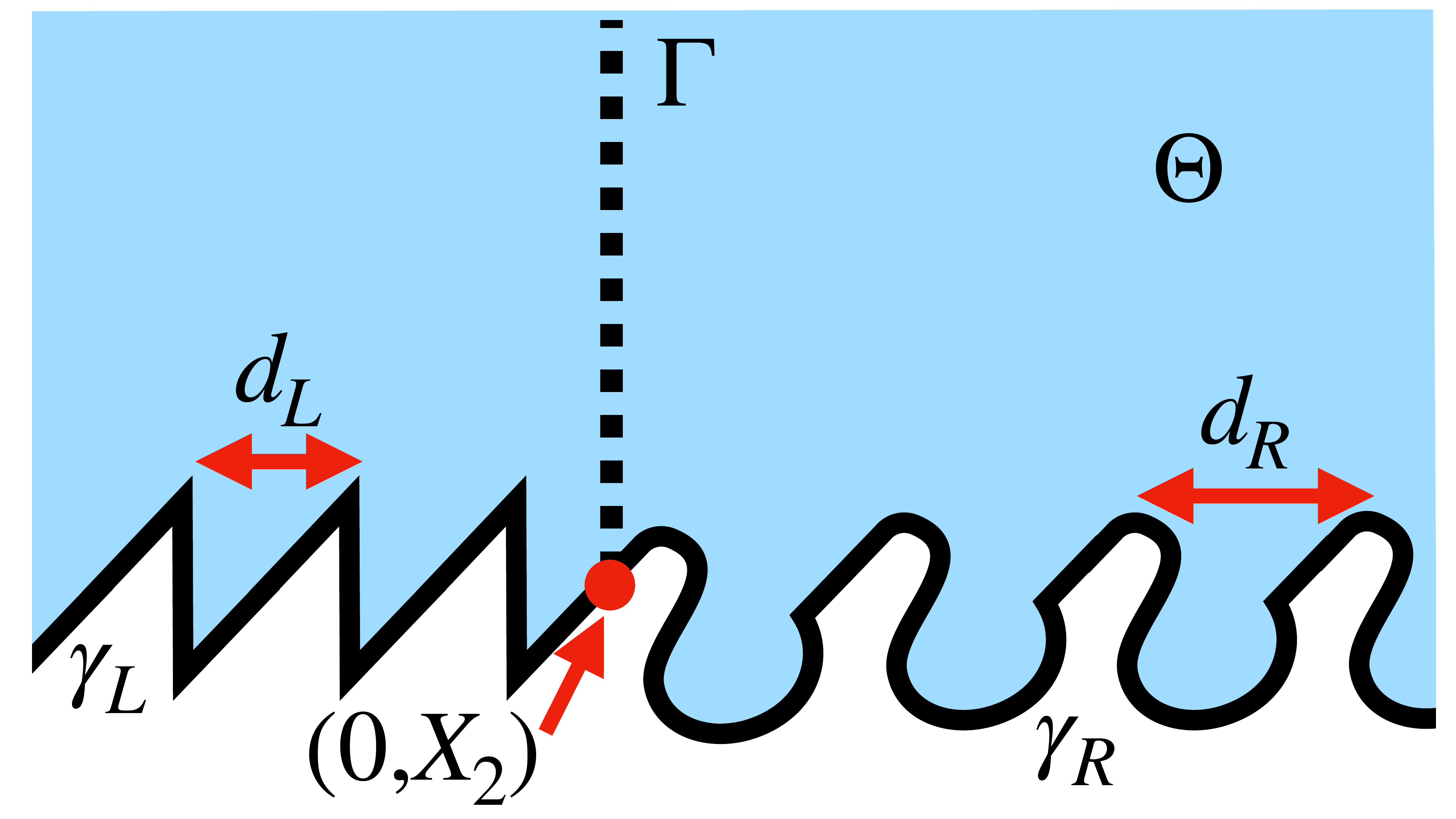}
    \caption{This figure illustrates our problem setup. We have two periodic walls~$\gamma_L$ and~$\gamma_R$ respectively and the fictitious interface~$\Gamma$ that separates the left and right halves of the computational domain. }
    \label{fig:prob_setup}
\end{figure}

We wish to solve the Helmholtz equation subject to Neumann boundary conditions,
\begin{equation}
\label{eq:tot_PDE}
    \begin{cases}
        \Delta u + k^2 u= f & \text{in } \Theta,\\
        \partial_{\bn} u = 0 & \bx \in \gamma_L,\;x_1<0 \\
        \partial_{\bn} u = 0 & \bx \in \gamma_R,\;x_1>0,
    \end{cases} 
\end{equation}
where $ \Delta = \frac{\partial^2}{\partial x_1^2} + \frac{\partial^2}{\partial x_2^2}$, $\partial_{\bn}  = \bn \cdot \nabla $, and~$f$ is a compactly supported source term.

Additionally, to insure solutions are outgoing, in a suitable sense, we require an additional `radiation condition' at infinity. Mathematically, the precise formulation of such conditions which guarantee uniqueness appears to be an open question. In this work, we content ourself with looking for solutions that satisfy the Sommerfeld radiation condition along any ray away from the interface, which will be part of any physically meaningful radiation condition. 

Here, for simplicity and ease of exposition, we focus on the case where the wavenumber $k$ is below $\tfrac{\pi}{d_{L,R}}.$ in principle, our analysis should extend {\it mutatis mutandis} to arbitrary wavenumbers, excluding Wood's anomalies. Extensive numerical evidence suggests that our numerical code extends in a similar way. 

Following the approach developed in~\cite{epstein2023solvinga,epstein2023solvingb,epstein2024solving}, we reformulate \eqref{eq:tot_PDE} as a transmission problem from a left half ($x_1 < 0$) to a right half ($x_1 > 0$) space. If~$\Gamma=\{(0,x_2)\in\bbR^2\;|\; x_2\geq X_2\}$ is the portion of the~$x_2$-axis in $\Theta$, then we wish to find~$u_{L,R}$ such that
\begin{equation}
    \begin{cases}
    \Delta u_{L,R} + k^2 u_{L,R} = 0 & \text{in } \Omega_{L,R}\\
    \partial_{\bn} u_{L,R} = 0 & \text{on } \gamma_{L,R}
    \end{cases} \label{eq:ulr_def}
\end{equation}
and
\begin{equation}
    \begin{cases}
        u_L - u_R = r_D & \text{on }\Gamma\\
        \partial_{x_1}u_L - \partial_{x_1}u_R = r_N & \text{on }\Gamma,
    \end{cases} \label{eq:transmission_prob}
\end{equation}
for some functions~$r_N$ and~$r_D$ that depend on the source~$f$ in a manner discussed in Section~\ref{sec:data}.

The advantage of this formulation is that we can reduce~\eqref{eq:transmission_prob} to an integral equation posed only on~$\Gamma$. To do this, we introduce the domain Green's functions~$G_{\gamma_{L,R}}(\bx, \by)$, which satisfy
\begin{equation}
    \begin{cases} 
    (\Delta+k^2) G_{\gamma_{L,R}}(\bx, \by) = \delta(\bx-\by)  &\text{in } \Omega_{L,R}\\
    \partial_n G_{\gamma_{L,R}}(\bx, \by) =0 & \text{on } \gamma_{L,R}.
    \end{cases}
\end{equation}
Using these Green's functions, we define the layer potentials
\begin{equation}
    \cS_{L,R}[\tau] = \int_{\Gamma} G_{\gamma_{L,R}}(\bx, \by)\tau(\by)\, \opd \by 
    \quad \text{and}\quad \cD_{L,R}[\sigma] = \int_{\Gamma} \partial_{y_1}G_{\gamma_{L,R}}(\bx, \by)\sigma(\by)\,\opd \by.
\end{equation}
We represent the solutions in the left and right half spaces,~$u_{L,R},$ by
\begin{equation}
    u_{L,R} = \cS_{L,R}[\tau]+\cD_{L,R}[\sigma].
\end{equation}
By construction, any~$u_{L,R}$ of this form will satisfy~\eqref{eq:ulr_def}. Since~$G_{\gamma_{L,R}}$ has the same singularity as the free-space Green's functions, we can use the standard jump relations for the Helmholtz layer potentials \cite{colton2013integral} to show that~$u_{L,R}$ solve~\eqref{eq:transmission_prob} if~$\sigma$ and~$\tau$ solve the integral equation
\begin{equation}
    \begin{pmatrix}
        \cI +\Ddiff&  \Sdiff \\ \Dpdiff & \cI + \Spdiff
    \end{pmatrix}\begin{pmatrix}
        \sigma\\\tau
    \end{pmatrix} := \begin{pmatrix}
        \cI +\cD_R-\cD_L& \cS_r-\cS_: \\ -(\cD_R'-\cD_L') & \cI-(\cS_R'-\cS_L') 
    \end{pmatrix}\begin{pmatrix}
        \sigma\\\tau
    \end{pmatrix} = \begin{pmatrix}
        r_D\\ -r_N
    \end{pmatrix}. \label{eq:real_IE}
\end{equation}
We will denote the kernels of~$\Ddiff,\Sdiff,\Dpdiff,$ and~$\Spdiff$ as~$k_\Ddiff, k_\Sdiff, k_\Dpdiff,$ and~$ k_\Spdiff,$ respectively.

As we will see below, the solutions of these integral equations will be oscillatory and decay algebraically, making them impossible to accurately truncate. To address this, we use the complex scaling method, used in
\cite{bonnet2022complex,dhia2024complex,epstein2024coordinate,goodwill2024numerical,epstein2025complex,hoskins2025quadrature}. In short, we analytically deform the contour~$\Gamma$ into the complex plane~$\tilde\Gamma$. If it is chosen correctly, then the kernels and densities will decay exponentially, allowing us to truncate the computational domain. The resulting integral equation can then be solved using standard methods. In the following section, we summarize several useful properties of the domain Green's functions~$G_{L,R}$. Afterwards, we shall show that \eqref{eq:real_IE} can be analytically continued and that the complexified operator is Fredholm index zero.

\begin{remark}\label{rem:other_BC}
    For ease of exposition, we focus our analysis on the case of Neumann boundary conditions. Many of the results easily extend to other boundary conditions, such as Dirichlet or impedance boundary conditions. All that we need is to build an integral representation (like \eqref{eq:w_xi_def}) and integral equation (analogous to \eqref{eq:quasi-IE}). In particular, all that is required is to show that Assumption \ref{ass:boundary} below is satisfied, in which case the proofs follow \textit{mutatis mutandis}. Finally, we remark that our approach also extends to the case of transmission problems, provided that the wavenumbers above and below the interface match. This can also be extended to allow for quasi-periodic dielectric 'leaky' waveguides. The straight waveguide case is analyzed in~\cite{epstein2025complex}.
\end{remark}

\section{Green's functions for periodic domains}\label{sec:Green_period_def}
In this section we summarize relevant properties of the Green's function for a domain~$\Omega$ with periodic boundary~$\gamma,$ which we denote by~$G_{\gamma}$. Though these properties are well-known, we include them here both for completeness and notational consistency.

Following \cite{agocs2024trapped}, we first write~$G_{\gamma}$ as an inverse Floquet--Bloch transform,
\begin{equation}
    G_\gamma(\bx,\by) = \frac{d}{2\pi} \int_{c} G_{\xi,\gamma}(\bx,\by) \opd \xi, \label{eq:green-floq}
\end{equation}
where~$c$ is a contour connecting~$\pm\frac{\pi}d$ lying in the second and fourth quadrants of the complex plane, with the functions~$G_{\xi,\gamma}$ satisfying
\begin{equation}
    \begin{cases}
        (\Delta + k^2)  G_{\xi,\gamma}(\bx,\by) =\delta(\bx-\by)  & \text{in } \Omega,\\
        \partial_n  G_{\xi,\gamma}(\bx,\by) =0 & \text{on } \gamma,\\
         G_{\xi,\gamma}(\bx+d\bs e_1,\by) = e^{i\xi d}G_{\xi,\gamma}(\bx,\by)& \text{in } \Omega,
    \end{cases}\label{eq:quasi_per}
\end{equation}
as well as a standard radiation condition at infinity. This quasi-periodic domain Green's function $G_\gamma$ has a few important properties. First, it is clearly periodic in~$\xi$, i.e.~$G_{\xi+\frac{2\pi}{d},\gamma}=G_{\xi,\gamma}$. Second, it is well-defined for all~$\xi\in \bbC$ away from branch cuts coming from~$\xi=\pm k$ and possibly poles~$\pm \tilde\xi_1,\ldots \pm \tilde \xi_{n_p}$ on the real line with~$\tilde\xi_j\in [k,\frac\pi d]$. These poles correspond to modes~$v_j$ that satisfy \eqref{eq:ulr_def} and propagate along~$\gamma$ (see \cite{agocs2024trapped} and the references therein).

Again following~\cite{agocs2024trapped}, we next split~$G_{\xi,\gamma}(\bx,\by) = G_{\xi}(\bx-\by) +w_{\xi,\gamma}(\bx,\by)$, where
\begin{equation}
    \begin{cases}
        (\Delta + k^2)  G_{\xi}(\bx) =\delta(\bx),  & \text{in } \bbR^2,\\
         G_{\xi}(\bx+d\bs e_1) = e^{i\xi d}G_{\xi}(\bx),& \text{in } \bbR^2,
    \end{cases}
\end{equation}
is the quasi-periodic fundamental solution and
\begin{equation}
    \begin{cases}
        (\Delta + k^2)  w_{\xi,\gamma}(\bx,\by) =0, & \text{in } \Omega,\\
        \partial_n  w_{\xi,\gamma}(\bx,\by) =-\partial_n G_{\xi}(\bx-\by), & \text{on } \gamma,\\
         w_{\xi,\gamma_{L,R}}(\bx+d\bs e_1,\by) = e^{i\xi d}w_{\xi,\gamma}(\bx,\by),& \text{in } \Omega.
    \end{cases}
\end{equation}
In real space, i.e. after performing the integral in $\xi$ over $c,$ this splitting corresponds to writing
\begin{equation}\label{eq:G_dom_split}
G_\gamma(\bx,\by) = G(\bx-\by) + w(\bx,\by),
\end{equation}
where~$G$ is the free-space fundamental solution for the Helmholtz equation
\begin{equation}
    G\lp\bx\rp = \frac{i}{4}H^{(1)}_0(k\|\bx\|)
\end{equation}
and
    \begin{equation}\label{eq:w_def}
        w(\bx,\by):= \frac{d}{2\pi}\int_c w_{\xi,\gamma}(\bx,\by)\, \opd \xi.
    \end{equation}

As noted in the introduction, there are a number of numerical methods for solving problems of the form~\eqref{eq:quasi_per}. In this work, we use the method employed by~\cite{agocs2024trapped}, which is well-suited to our analysis. In this method, we find~$w_{\xi,\gamma_{L,R}}$ by expressing it as
\begin{equation}\label{eq:w_xi_def}
    w_{\xi,\gamma}(\bx,\by) = S_{\xi, \gamma}[\rho_{\xi,\by}](\bx) = \int_{\gamma}  G_{\xi}(\bx-\bz) \rho_{\xi,\by}(\bz)\, \opd \bz,
\end{equation}
where~$\rho_{\by}$ satisfies
\begin{equation}
    \cK_\xi[\rho_{\xi,\by}](\bz):=-\frac{\rho_{\xi,\by}(\bz)}{2}+ S_{\xi, \gamma}'[\rho_{\xi,\by}](\bz)=-\partial_{\bn(\bz)} G_\xi(\bz-\by) \label{eq:quasi-IE}
\end{equation}
with
\begin{equation}
    S_{\xi, \gamma}'[\rho_{\xi,\by}](\bx) = \int_{\gamma}  \partial_{\bn(\bz)}G_{\xi}(\bx-\bz) \rho_{\xi,\by}(\bz)\, \opd \bz.
\end{equation}

This formulation has a few important features. First, it allows us to use the analyticity of~$G_\xi$ to prove that~$G_{\xi,\gamma}$, and so~$G_\gamma$, can be analytically continued. Second, in \cite{agocs2024trapped,aylwin2020properties} it was observed that \eqref{eq:quasi-IE} is well-conditioned when~$\xi$ is away from the branch cuts and poles, making it easy to solve accurately. 

There are a few equivalent formulas for the quasi-periodic fundamental solution. The most obvious formula is the conditionally convergent series (equation 2.3 in~\cite{linton2010lattice})
\begin{equation}\label{eq:cond_form}
    G_\xi(\bx) = \sum_{n=-\infty}^\infty e^{in\xi d} G\lp\bx+nd \bs e_1\rp.
\end{equation}
This formula is difficult to use in practice because it does not converge in any sense for complex~$\xi$. A more convenient formula is the~$x_1$-Fourier series (equation 2.9 in~\cite{linton2010lattice}):
    \begin{equation}
    G_{\xi} (\bx) = \sum_{m=-\infty}^\infty e^{i\xi_m x_1}\frac{e^{\alpha(\xi_m)\sqrt{x_2^2}}}{-2\alpha(\xi_m)} \label{eq:dualsum}
\end{equation} 
where~$\xi_m = \xi + \frac{2\pi}{d}m$ and~$\alpha(\xi) = -\sqrt{i(\xi-k)}\sqrt{- i(\xi+k)}$ with the branch cut of the square root taken along the negative real axis. 
We will see below that~\eqref{eq:dualsum} converges as long as~$x_2\neq 0$. 

To analyze the behavior of~$G_\xi$ when~$x_2\approx 0$, we introduce the following formula, which is equivalent to equation 17 in \cite{yasumoto2002efficient} using \eqref{eq:cond_form} and an integral formula for the Hankel function.
For any $l\in\mathbb{N}$, the quasi-periodic Green's function can be written as a sum of the Bessel functions
\begin{equation}
    G_\xi(\bx) = \sum_{j=-l}^le^{ij\xi d}G(\bx+jd\bs e_1) + \frac12 S_0 J_0(k\|\bx\|) +\sum_{n=1}^\infty S_n J_n(k\|\bx\|) \cos(n\opnm{arg} \bx) \label{eq:bessel_sum}
\end{equation}
for~$\|\bx\|< (l+1/2)d$. The lattice coefficients are given by
\begin{multline}
    S_n = \frac{e^{i\pi/4}}{\sqrt{2}\pi}\left[(-1)^ne^{-(l+1)\xi d i} \int_0^\infty\lp G_n((1-i)t)+G_n(-(1-i)t) \rp F((1-i)t,\xi) \, \opd t\right.\\
    +\left. e^{(l+1)\xi d i}\int_0^\infty\lp G_n((1-i)t)+G_n(-(1-i)t) \rp F((1-i)t,-\xi) \, \opd t\right]
\end{multline}
where
\begin{equation}
    G_n(t) = \lp t-i\sqrt{1-t^2} \rp^n e^{ilkd\sqrt{1-t^2}},\quad F(t,\xi) = \lp \sqrt{1-t^2} \left[1-e^{ikd\lp\sqrt{1-t^2}-\xi/kd\rp} \right]\rp^{-1}.
\end{equation}

An immediate consequence of this formula is that~$G_\xi(\bx) - G(\bx)$ is smooth in the rectangle~$(x_1,x_2)\in (-d,d)\times (-ld,ld)$ for any~$l$. In the remainder of this section, we use the two formulas for $G_\xi$ given in~\eqref{eq:dualsum} and~\eqref{eq:bessel_sum} to study the behavior of~$G_\gamma$ near~$\gamma$ with the following assumptions.
\begin{assumption}\label{ass:boundary}
    The boundary $\gamma$ is piecewise smooth and flat in a neighborhood of~$(0,X_2)$. Further,~$\gamma$ is such that the operator~$\cK_\xi$ \eqref{eq:quasi-IE} is bounded on~$L^2(\gamma)$, the space of continuous functions on~$\gamma$, and invertible for all~$\xi$ except for the branch cuts of~$\alpha$ and modes~$\pm \tilde \xi_1,\ldots, \pm \tilde \xi_{n_p}$ that lie on the real axis with~$k<\tilde \xi_j<\frac\pi d$.
\end{assumption}
\begin{remark}
    It was shown in \cite{agocs2024trapped} that if~$\gamma$ is the graph of a piecewise smooth function, flat near $(0,X_2),$ then Assumption~\ref{ass:boundary} is satisfied.
\end{remark}

\begin{lemma}\label{lem:ops_analytic}
    Let~$h=\frac\pi d+k+1$,~$\epsilon$ be a positive constant, and~$B_k$ be the branch cuts of~$\alpha(\xi)$. 
    Further, let
    \begin{equation}\label{eq:vdef}
        V_{\gamma,\epsilon} = \left\{\xi\in\bbC \,\middle| \, |\Re \xi| < \frac{\pi}{d} +\epsilon,\,|\Im \xi| < h,\, |\xi\pm \tilde\xi_j|>\epsilon, \opnm{dist}(\xi,B_k)>\epsilon \right\},
    \end{equation}
    where~$\epsilon>0$ is small enough that~$V_{\gamma,\epsilon}$ contains the origin and~$\pm \pi/d$.
    The operators $\cK^{-1}_{\xi,\gamma}$ and~$S_{\xi,\gamma}$ are analytic operators for~$\xi\in V_{\gamma,\epsilon}$.
\end{lemma}
\begin{proof}
We first observe that equations~\eqref{eq:dualsum} and~\eqref{eq:bessel_sum} imply that~$G_\xi$ is an analytic function of~$\xi\in V_{\gamma,\epsilon}$ for all~$\bx,\by$. Thus the kernels of $S_{\xi, \gamma}$ and $S_{\xi, \gamma}'$ are analytic. To prove that they are bounded operators, we write
\begin{equation}
    S_{\xi, \gamma} = \lp S_{\xi, \gamma}-S_{\gamma}\rp+S_{\gamma} \quad S_{\xi, \gamma} = \lp S_{\xi, \gamma}'-S_{\gamma}'\rp+S_{\gamma}',
\end{equation}
where~$S_\gamma$ and~$S_\gamma'$ are the usual Helmholtz layer potentials. These are well-known to be bounded on~$L^2(\gamma)$ (see \cite{kress1999linear}).

By \eqref{eq:bessel_sum}, the operators~$S_{\xi, \gamma}-S_{\gamma}$ and $S_{\xi, \gamma}'-S_{\gamma}'$ have smooth analytic kernels and so are bounded and analytic operators on the same spaces. Putting this together with the results for $S_{\gamma}$ and $S_{\gamma}'$, we have that $S_{\xi,\gamma},S_{\xi,\gamma}'$ and~$\cK_{\xi,\gamma}$ are analytic and bounded operators.

    Since~$\cK_{\xi,\gamma}$ is an analytic and invertible at each~$\xi\in V_{\gamma,\epsilon}$, we have that~$\cK^{-1}_{\xi,\gamma}$ is also analytic in the same region (see e.g. \cite{Dunford1958} Lemma VII.6.4). 
\end{proof}
We now use this result to establish the existence of the domain Green's function.
\begin{lemma}\label{lem:G_exist}
    The function $G_\gamma(\bx,\by)$ and all of its derivatives exist for~$\bx,\by\in \Omega$ with~$\bx\neq \by$.
\end{lemma}
\begin{proof}
    By equations~\eqref{eq:dualsum} and~\eqref{eq:bessel_sum}, we have that $-\partial_{\bn(\bz)} G_\xi(\bz-\by)$ is well defined for all~$\by$ not on~$\gamma$ with $|y_1|<2d$ and all~$\xi\in V_{\gamma,\epsilon}$. The periodicity of~$G_\xi$ then implies that it exists and is a smooth function of~$\bz$ for all~$\by\in\Omega$. By the previous lemma and~\eqref{eq:w_xi_def}, we thus have that 
    \begin{equation}
        w_{\xi}(\bx,\by)=S_{\xi,\gamma} \left[\cK_{\xi,\gamma}^{-1}[-\partial_{\bn(\cdot )} G_\xi(\cdot-\by)]\right](\bx)\label{eq:w_xi_comp}
    \end{equation}
    is well defined and analytic in~$V_{\gamma,\epsilon}$. Since the contour~$c\subset V_{\gamma,\epsilon}$ in \eqref{eq:w_def} is compact, we thus have that $w(\bx,\by)$ 
    is well defined for each~$\bx,\by\in\Omega$. Since $G_\gamma(\bx,\by) = G(\bx-\by)+w_\gamma(\bx,\by)$, we have proved the result.

    To see that that~$w_\gamma$ is a smooth function of~$\by$, we note that, since the operators~$S_{\xi,\gamma} $ and $\cK_{\xi,\gamma}^{-1}$ are bounded, we can pull any $\by$ derivatives inside the righthand side of $\eqref{eq:w_xi_comp}$ and repeat the argument. The smoothness as a function of~$\bx$ follows from the smoothness of the kernel of~$S_{\xi,\gamma}$. The smoothness of~$G_\gamma$ then follows.
\end{proof}

\section{Analyzing the glued integral equation}\label{sec:glued_IE_anal}
Given the system of integral equations~\eqref{eq:real_IE} on the real contour $\Gamma$, we next show that it can be analytically continued to an integral equation on a suitable family of complex contours with a Fredholm index zero operator. We do this in three steps. First, in Section~\ref{sec:green_ana} we show that the domain Green's functions (and so the kernels of the integral operators in~\eqref{eq:real_IE}) can be analytically continued. We then show in Section~\ref{sec:bdd_opps} that the operators are bounded in a suitable Banach space. Finally, in Section~\ref{sec:Fredholm} we show that the resulting operator is Fredholm index zero. We refer the reader to Section \ref{sec:transm_prob} and Figure \ref{fig:prob_setup} for an illustration of the geometry considered and relevant notation.

\subsection{Analyticity of the domain Green's function}\label{sec:green_ana}
In order to handle the slow decay of the kernels and densities, we analytically continue the integral equation into the complex plane. Specifically, the relevant set of the complex plane will be
\begin{equation}
    \Gamma_U = \{ x_2 \in \bbC \;|\; \Re x_2\geq \max(d_L,d_R)/2,\; 0\leq \Im x_2\leq K_{slope} \Re x_2\},
\end{equation}
for some constant~$\Kslope>0$. The domain~$\Gamma_U$ is chosen so that outgoing oscillatory functions, such as~$G((0,x_2))$, will decay exponentially as~$\Im x_2$ grows. We give conditions on the size of~$\Kslope$ in \eqref{eq:slope_def} that will guarantee that the kernels of our integral equation decay sufficiently rapidly on the complexified contour. We choose~$\Gamma_U$ to start at the height~$\Re x_2=\max(d_L,d_R)/2$ to ensure that the the contour deformation starts a finite distance away from the boundaries,~$\gamma_{L}$ and~$\gamma_R$ which were assumed to live in the region~$x_2<0$.

We denote the remaining piece of the interface as 
\begin{equation}
    \Gamma_D = ({X_2},\max(d_L,d_R)/2],
\end{equation}
and the union of these regions as
\begin{equation}
    \Gamma_\bbC = \Gamma_D\cup \Gamma_U.
\end{equation}
We also use the following extension of~$\Omega$:
\begin{equation}
    \Omega_\bbC = \{\by \in \bbR\times \lp (-\infty,\max(d_L,d_R)/2]\cup \Gamma_U\rp \;|\; (x_1,\Re x_2)\in \Omega\}.
\end{equation}
It will sometimes be convenient to exclude a neighborhood of any corners of~$\gamma$. For any~$\delta>0$, we let $ \Omega_{\bbC,\delta}$ be the set of points in~$\Omega_{\bbC}$ that are at least a distance~$\delta$ from every corner of~$\gamma$.
The set of allowable contours will be defined as follows.
\begin{definition}\label{def:G}
Let~$\Kslope$ be a real number greater than zero and~$\cG$ be the collection of curves~$\tGamma\subset \Gamma_\bbC$ with a parameterization
    \begin{equation}
        x_2(t) = t + i f(t)
    \end{equation}
    where~$f$ is a smooth function satisfying $0\leq f'(t) \leq \Kslope$ and~$f(t)=0$ if~$t<\max(d_L,d_R)/2$.
\end{definition}
In order to show that~$w$ can be analytically continued to~$\GammaC$, we need the following properties of~$\alpha$.
\begin{lemma}\label{lem:alphaprop}
    The function~$\alpha$ satisfies
    \begin{equation}
    \alpha(\xi) = \begin{cases}
        i\sqrt{k^2- \xi^2} & \text{if} \;|\xi|<k\\
        -\sqrt{\xi^2-k^2} & \text{if} \;|\Re \xi| >k
    \end{cases}\label{eq:alpha_expl}.
\end{equation}
    Further
    \begin{equation}
        \alpha(\xi) = -\sqrt{\xi^2} + \frac{k^2}{2\sqrt{\xi^2}}+ O(\xi^{-3})\label{eq:alpha_asym}
    \end{equation}
    as~$\Re\xi\to\pm \infty$.
\end{lemma}
\begin{proof}
    When~$\Re\xi>k$, we have~$\opnm{Arg}(i(\xi-k))\in (0,\pi)$, where~$\opnm{Arg}$ is the principle argument chosen to lie in~$(-\pi,\pi]$. Similarly~$\opnm{Arg}(-i(\xi+k))\in (-\pi,0)$. We thus have that
    \begin{equation}
        \opnm{Arg}(\xi^2-k^2) = \opnm{Arg}(i(\xi-k))+\opnm{Arg}(-i(\xi+k)),
    \end{equation}
    which implies that~$\alpha(\xi) = -\sqrt{\xi^2-k^2}$ when~$\Re \xi>k$. A similar argument can be applied to the case~$\Re \xi<-k$. The expression for~$|\xi|<k$ can be proved by a similar, but more tedious calculation.

    To derive the asymptotic formula, we simply note that if $\Re\xi$ is large, then~$\xi^2-k^2$ is far from the branch cut of the square root for large. We can therefore apply the binomial formula to derive the result.
\end{proof}

\begin{lemma} \label{lem:ana_cont}
    If~$\gamma$ satisfies Assumption \ref{ass:boundary} and~$\xi\in V_{\gamma,\epsilon}$, then~$w_{\xi,\gamma}(\bx,\by)$ and its $x_1$ and~$y_1$ derivatives can be analytically continued to any~$\bx,\by\in \Omega_\bbC$.
\end{lemma}
\begin{proof}
We first show that~$G_\xi(\bx)$ is analytic in the region
\begin{equation}
    D_\delta=\{(x_1,x_2)\in \bbR\times \bbC\;|\Re x_2\geq \delta, \; \Im x_2\geq 0,\; |x_2|<\delta^{-1}\}
\end{equation} 
for any~$\delta>0$. Each term in the representation of $G_\xi$ given by ~\eqref{eq:dualsum} can be bounded as follows
\begin{equation}
    \left|e^{i\xi_mx_1 }\frac{e^{\alpha(\xi_m)\sqrt{x_2^2} } }{-2\alpha(\xi_m)}\right| =\left|e^{i(\xi + \frac{2|m|\pi}{d})x_1 }\frac{e^{\alpha(\xi_m)x_2} }{2\alpha(\xi_m)} \right|.
\end{equation}
For~$m\neq 0$,~$\Re \alpha(\xi_m)<0$, so we can in turn bound the above expression by
\begin{multline}
    |e^{i\xi_mx_1 }|\frac{e^{\Re\lp\alpha(\xi_m)x_2 \rp} }{2|\alpha(\xi_m)|}=e^{-x_1\Im \xi_m } \left|\frac{e^{\Re\lp -\frac{2|m|\pi}{d}x_2 - \operatorname{sign}(m) \xi x_2  + O(m^{-1})\rp } }{\frac{4|m|\pi}{d} + O(1)} \right|\\
    \leq  e^{-x_1\Im \xi } \left|\frac{e^{-\frac{2|m|\pi}{d}\delta + O(m^{-1})} }{\frac{4|m|\pi}{d} + O(1)} \right|e^{|\xi| \delta^{-1}}.
\end{multline}
This decays exponentially in~$|m|$ for any fixed~$\xi\in V_{\gamma,\epsilon}$, so the series in~\eqref{eq:dualsum} converges and $G_\xi(\bx)$ exists. In fact, we have shown that~$G_\xi(\bx)$ is a uniform limit of analytic functions on~$D_\delta$ and so is analytic there. Repeating this for all~$\delta$ shows that~$G_\xi(\bx)$ is analytic on~$\Re x_2>0$ and~$\Im x_2\geq 0$. An equivalent argument shows that~$G_\xi$ is an analytic on the region~$\Re x_2 < 0$ and~$\Im x_2\leq 0$, so~$G_\xi(\bx)$ is analytic when~$\Re x_2\neq 0$ and~$\operatorname{sign}(\Im x_2) = \operatorname{sign}(\Re x_2)$.

By differentiating~\eqref{eq:dualsum} with respect to~$x_1$ term by term and repeating the above argument, we can also show that the $x_1$ derivatives of~$G_\xi$ exist and are analytic when~$\Re x_2 \neq 0$. We therefore have that 
\begin{equation}
    \partial_{\bn(\bz)} G_\xi (\bz- \by) 
\end{equation}
can be analytically continued from~$\by\in\Omega$ to~$y\in\Omega_\bbC$ and is a smooth function of~$\bz$. By assumption, the boundary integral equation~\eqref{eq:quasi-IE} is invertible and so~$\rho_{\by}(\bz)\in L^2(\gamma)$ is well-defined for any~$\by\in \Omega_\bbC$. We may therefore write
\begin{equation}
    w_{\xi,\gamma}(\bx,\by) = S_{\xi,\gamma}[\rho_{\xi,\by}](\bx) = \int_\gamma G_\xi(\bx-\bz) \rho_{\xi,\by}(\bz) \,\opd \bz. \label{eq:wxi_analytic}
\end{equation}
For any fixed~$\bx \in \Omega_\bbC,$ $G_\xi(\bx-\bz)$ will be a smooth function of~$\bz\in\gamma$ and so the integral must converge. A simple application of Morera's and Fubini's theorems implies that~$w_{\xi,\gamma}(\bx,\by)$ is an analytic function of~$\bx$ in~$\Omega_\bbC$.

To see that the~$x_1$ derivatives of~$w_{\xi,\gamma}$ are analytic, we note that the integrand in~\eqref{eq:wxi_analytic} is a smooth function of~$\bx$ in the interior of $\Omega_\bbC$ and so we can pull the derivatives inside the integral and repeat the above argument. For the $y_1$ derivatives, we note that~$\rho_{\by}(\bz)$ is a bounded linear operator applied to a smooth function of~$\by$ in the interior of~$\Omega_\bbC$, and so a smooth function of $\by$ for each $\bz$. We can therefore pull the derivatives of~$w_{\xi,\gamma}$ with respect to~$y_1$ inside the integral and apply similar proofs.
\end{proof}

We now prove that the real-space function~$w_\gamma$ can be analytically continued.
\begin{theorem}\label{thm:w_analytic}
    The function~$w_\gamma$ and its $x_1$ and~$y_1$ derivatives can be analytically continued from~$\Omega^2$ to~$\Omega_\bbC^2$.
\end{theorem}
\begin{proof}
    Since~$w_{\xi,\gamma}$ is well defined on~$\Omega_\bbC^2$, we can formally write
\begin{equation}
    w_\gamma(\bx,\by) = \int_c w_{\xi,\gamma}(\bx,\by)\,\opd \xi
\end{equation}
for all~$\bx,\by\in \Omega_\bbC$. By Lemma~\ref{lem:ops_analytic},~$w_{\xi,\gamma}$ is an analytic function of~$\xi$ in a neighborhood of~$c$, so the integral is well defined.

By Lemma~\ref{lem:ana_cont}, we know that~$w_{\xi,\gamma}(\bx,\by)$ is an analytic function of~$x_2$ and~$y_2$. A simple application of Morera's and Fubini's theorems therefore gives that~$w_\gamma$ is analytic function of~$x_2$ and~$y_2$. The $x_1$ and~$y_1$ derivatives of~$w_\gamma$ can similarly be shown to be analytic.
\end{proof}
To find the appropriate domains for our integral operators, we must understand the behavior of their kernels for large~$x_2,y_2$, which is characterized in following theorem.
\begin{theorem}\label{thm:split_w}
    Let~$\thetaeta>0$ be such that
    \begin{equation}
        \thetaeta =  \frac12 \opnm{Arg}\left[\lp \frac{\pi}d +ih\rp^2-k^2\right]+\arctan\Kslope\label{eq:slope_def},
    \end{equation}
    where $h$ is the same $h$ as was used in the definition of $V_{\gamma,\epsilon}$ in \eqref{eq:vdef}. Also let~$\Kslope$ be small enough that~$\thetaeta<\pi/2$ and let~$\eta = |\cos(\thetaeta) \alpha\lp\pi/d\rp+\epsilon|$ for some small~$\epsilon>0$. 
    
    For every~$l\geq 0$, there exist a constant~$K_l$ and continuous functions~$a_l$ and~$A_l$ such that~$w_\gamma$ can be split
    \begin{equation}
        w_{\gamma} = w_{\gamma,rr} + w_{\gamma,ri}+ w_{\gamma,ir}+ w_{\gamma,ii}
    \end{equation}
    with
    \begin{align}
        &|\partial_{x_1}^l w_{\gamma,rr}(\bx,\by)| \leq K e^{-\eta\Re(x_2+y_2)}e^{h|x_1-y_1|},\quad \\ 
        &\left|\partial_{x_1}^l w_{\gamma,ri}(\bx,\by) - \frac{a_l(\bx,y_1)e^{iky_2}}{y_2^{\opnm{ceil}(l/2)+1/2}}\right| \leq \frac{K  e^{-\eta\Re x_2-k\Im y_2}e^{h|x_1-y_1|}}{(1+|y_2|)^{\opnm{ceil}(l/2)+1}},\\
        &w_{\gamma,ir}(\bx,\by)=w_{\gamma,ri}(\by,\bx), \quad \\ \text{and}&\quad \left|\partial_{x_1}^l  w_{\gamma,ii}(\bx,\by) - \frac{A_l(x_1,y_1)e^{ik(x_2+y_2)}}{(x_2+y_2)^{\opnm{ceil}(l/2)+1/2}}\right| \leq K \frac{e^{-k\Im(x_2+y_2)}e^{h|x_1-y_1|} }{(1+|x_2+y_2|)^{\opnm{ceil}(l/2)+1}},
    \end{align}
    for all~$\bx,\by\in\Omega_{\bbC}$ with~$\Re x_2$ and~$\Re y_2> d/2$. The functions~$a_l$ are analytic and satisfy $|a_l(\bx,y_1)|\leq D_l e^{-\eta\Re x_2 + h|x_1-y_1|}$ for some constant~$D_l$.

    Further, for every~$\delta>0$ and~$l\geq 0$, there exists a constant~$K_{l,\delta}$ and continuous function~$b_l$ such that~$w_\gamma$ can be split
    \begin{equation}
        \left|w_{\gamma,r}(\bx,\by) - \frac{b_l(\bx,y_1)e^{iky_2}}{y_2^{^{\opnm{ceil}(l/2)+1/2}}}\right| \leq K_{l,\delta} e^{-\eta \Re y_2 + h|y_1|} + K_{l,\delta} \frac{e^{-k\Im y_2+ h|y_1|}}{|y_2|^{^{\opnm{ceil}(l/2)+1}}}\label{eq:wnearfar_bd}
    \end{equation}
    for all~$\bx,\by\in\Omega_{\bbC,\delta}$ with~$\Re x_2 \leq d/2$ and~$\Re y_2> d/2$. 

    Finally, equivalent expressions hold for~$\Re x_2 > d/2$ and~$\Re y_2\leq d/2$ and for the $y_1$ derivatives of~$w_\gamma$.
\end{theorem}
\begin{proof}
We study the behavior for large $x_2$ or~$y_2$ in Appendices~\ref{app:far_as} and~\ref{app:far2near_as}. The desired results can be obtained by applying~\eqref{eq:wnm_def} and Lemmas~\ref{lem:nonzero_nm}, \ref{lem:00_bound}, and \ref{lem:far2near} and Proposition \ref{lem:0n_estimate}, the observation that~$\partial_{x_1} w_\gamma(\bx,\by) = - \partial_{y_1} w_\gamma(\bx,\by)$, the symmetry of~$w_\gamma$ in~$\bx,\by$, and the fact that~$\eta_1 = \alpha(\pi/d)\cos \thetaeta$, so $\eta = -(\eta_1+\epsilon)$.
\end{proof}
The behavior of $w_\gamma$ when both~$\bx$ and~$\by$ are close to~$\gamma$ is captured in the following lemma.
\begin{lemma}\label{lem:w_image}
Suppose~$\gamma$ coincides with the line~$\gamma_{X_2}$ in the~$\delta$ ball centered at~$(0,{X_2})$ and~$\tilde\by$ is the reflection of~$\by$ through $\gamma_{X_2}$. We can split
    \begin{equation}
        w_\gamma(\bx,\by) = G(\bx-\tilde\by) + \tilde w_\gamma(\bx,\by)
    \end{equation}
    where~$\tilde w_\gamma$ and its derivatives are bounded functions of~$\bx,\by\in \Omega \cap \{\|\bz-(0,{X_2})\|<\delta/2\}$.
\end{lemma}
\begin{proof}
    By repeating the above argument, we can see that if $\by$ is within~$\delta/2$ of~$(0,{X_2})$ then
\begin{equation}
    \tilde w_{\xi,\gamma}(\bx,\by)=- S_{\xi,\gamma} \left[\cK_{\xi,\gamma}^{-1}[\tilde h_{\xi,\by}]\right](\bx),
\end{equation}
  where
\begin{equation}
    \tilde h_{\xi,\by}(\bz)=\partial_{\bn(\bz )} G_\xi(\bz-\by)+\partial_{\bn(\bz )} G_\xi(\bz-\tilde\by).
\end{equation}
By symmetry, we have that $\tilde h_{\xi,\by}(\bz)$ is zero for~$\bz\in \gamma\cap\gamma_{X_2}$. Thus $\tilde h_{\xi,\by}(\bz)$ is bounded for all $\by\in \{\|\bz-(0,{X_2})\|<\delta/2\}$ and~$\bz\in\gamma$. We can then repeat the proof of Lemma~\ref{lem:G_exist} to prove the result.
\end{proof}

\subsection{Properties of the continued operators}\label{sec:bdd_opps}
Using the above analytic continuation, we can also define analytic continuation of the kernels~$k_\Ddiff, k_\Sdiff, k_\Dpdiff,$ and~$ k_\Spdiff$. The Helmholtz Green's function can also be analytically continued to the same domain (see~\cite{epstein2024coordinate}). We are therefore able to define the analytic continuation of the layer potentials to curves in $\cG$ (see Definition \ref{def:G}).
\begin{definition}
    For any~$\tGamma\in \cG$ and functions~$\sigma$ and~$\tau$, we define the following integral operators:
    \begin{equation}        
        \Ddiff_{\tGamma}[\sigma](x_2) = \int_{\tGamma} k_{\Ddiff}(x_2,y_2) \sigma(y_2)\,\opd y_2,\quad\Sdiff_{\tGamma}[\sigma](x_2) = \int_{\tGamma} k_{\Sdiff}(x_2,y_2) \sigma(y_2)\,\opd y_2,
    \end{equation}
    and
    \begin{equation}        
        \Dpdiff_{\tGamma}[\sigma](x_2) = \int_{\tGamma} k_{\Dpdiff}(x_2,y_2) \sigma(y_2)\,\opd y_2,\quad\Spdiff_{\tGamma}[\sigma](x_2) = \int_{\tGamma} k_{\Spdiff}(x_2,y_2) \sigma(y_2)\,\opd y_2,
    \end{equation}
    for any~$x_2\in \Gamma_\bbC$. We also define
\begin{multline}
    \cD_{\tGamma,L,R}[\sigma](\bx)=\int_{\tGamma} \partial_{y_1}G_{\gamma_{L,R}}(\bx; 0,y_2)\sigma(y_2)\, \opd y_2 \\\quad\text{and} \quad 
    \cS_{\tGamma,L,R}[\tau](\bx)=\int_{\tGamma} G_{\gamma_{L,R}}(\bx;0,y_2) \tau(y_2)\, \opd y_2
\end{multline}
for any~$\bx\in \Omega_\bbC$.
\end{definition}

\begin{theorem}
    The kernel~$k_\Sdiff$ can be split
    \begin{equation}
        k_\Sdiff = k_{\Sdiff{rr}} + k_{\Sdiff{ri}}+ k_{\Sdiff{ir}}+ k_{\Sdiff{ii}}
    \end{equation}
    with
    \begin{multline}
        |k_{\Sdiff{rr}}(x_2,y_2)| \leq K e^{-\eta\Re(x_2+y_2)},\quad |k_{\Sdiff{ri}}(x_2,y_2)| \leq K  (1+|y_2-{X_2}|)^{-1/2}e^{-\eta\Re x_2-k\Im y_2},\\
        |k_{\Sdiff{ir}}(x_2,y_2)| \leq K (1+|x_2-{X_2}|)^{-1/2}e^{-k\Im x_2-\eta\Re y_2}, \\\quad \text{and}\quad |k_{\Sdiff{ii}}(x_2,y_2)| \leq K \frac{e^{-k\Im(x_2+y_2)} }{(1+|x_2+y_2-2{X_2}|)^{1/2}}
    \end{multline}
    for all~$x_2,y_2\in \Gamma_\bbC$. The kernels~$k_\Ddiff,k_\Dpdiff,$ and~$k_\Spdiff$ can be split similarly with algebraic power~$3/2$. 
\end{theorem}
\begin{proof}
By Lemma~\ref{lem:w_image} if~${X_2}<x_2,y_2<{X_2}+\delta/2 $, then
\begin{equation}
    w_{\gamma_L}((0,x_2),(0,y_2) - w_{\gamma_R}((0,x_2),(0,y_2)=\tilde w_{\gamma_L}((0,x_2),(0,y_2) - \tilde w_{\gamma_R}((0,x_2),(0,y_2).
\end{equation}
The boundedness of~$\tilde w_{\gamma_{L,R}}$ thus give that~$k_\Ddiff,k_\Sdiff,k_\Dpdiff,$ and~$k_\Spdiff$ are bounded on~$\Gamma_D\times\Gamma_D$. The above estimates can then be derived by applying Theorem~\ref{thm:split_w} to~$w_{\gamma_{L,R}}$ and noting that~$(1+|x_2|)^{\alpha}$ can be bounded by a multiple of $(1+|x_2-X_2|)^{\alpha}$ for any~$\alpha$.
\end{proof}
The split kernels define operators, which we denote by~$\Ddiff_{\tGamma,rr}$, etc. We introduce the following Banach spaces, which are naturally suited to kernels with these decay properties.
\begin{definition}
    Let~$\imagdecay{\alpha}{\beta}$ be the space of functions that are continuous in~$\GammaC$, are analytic in its interior, and satisfy
    \begin{equation}
        \|f\|_{\alpha,\beta} = \sup_{z\in\GammaC} (1+|z|)^\alpha e^{\beta \Im z} |f(z)| <\infty.
    \end{equation}
    Let~$\realdecay{\rho}$  be the space of functions that are continuous in~$\GammaC$, are analytic in its interior, and satisfy
    \begin{equation}
        \|f\|_{\rho} = \sup_{z\in\GammaC} e^{\rho \Re z} |f(z)| <\infty.
    \end{equation}

\end{definition}
The~$\imagdecay{\alpha}{\beta}$ space was observed to be a Banach space in~\cite{epstein2025complex}. The space~$\realdecay{\rho}$ can be seen to be a Banach space by similar arguments. While we introduce~$\realdecay\rho$ for convenience and clarity, it is important to note that the choice of~$\GammaC$ implies that it is contained in~$\imagdecay\alpha\beta$ for an appropriate choice of~$\beta$.

\begin{lemma}\label{lem:embed}
    If~$\alpha>0$ and~$0<\epsilon<\rho$, then
    \begin{equation}
        \realdecay\rho\subset\imagdecay\alpha{(\rho-\epsilon)/\Kslope} .
    \end{equation}
\end{lemma}
\begin{proof}
    Since~$0\leq \Im z\leq \Kslope \Re z$ for all~$z\in \GammaC$, we know
    \begin{equation}
        e^{\lp (\rho-\epsilon)/\Kslope\rp \Im z} \leq e^{(\rho-\epsilon) \Re z}.
    \end{equation}
    Similarly, we can bound
    $
        (1+|z|)^\alpha\leq C_{\epsilon,\alpha} e^{\lp \epsilon/\sqrt{1+\Kslope^2}\rp|z|} \leq C_{\epsilon,\alpha} e^{\epsilon\Re z}
    $
    for all~$z\in\GammaC$. Putting these together, we have that for all~$f\in \realdecay{\rho}$,
    \begin{equation}
        (1+|z|)^\alpha e^{(\rho -\epsilon)/\Kslope\Im z} |f(z)| \leq C_{\epsilon,\alpha}  e^{\epsilon \Re z} e^{\lp(\rho -\epsilon)/\Kslope\rp\Im z}  |f(z)| \leq C_{\epsilon,\alpha} e^{\rho \Re z} |f(z)| \leq C_{\epsilon,\alpha}\|f\|_{\rho},
    \end{equation}
    which proves the result.
\end{proof}

It is also clear that our spaces are nested in the sense that
\begin{equation}
    \imagdecay{\alpha'}{\beta'}\subset\imagdecay{\alpha}{\beta} \quad\text{and}\quad \realdecay{\rho'}\subset\realdecay{\rho}
\end{equation}
whenever~$\alpha\leq \alpha'$, $\beta\leq \beta'$, and~$\rho\leq \rho'$.

\begin{assumption}\label{ass:powers}
    Unless otherwise stated, we assume that the parameters~$\alpha,\beta,\rho$ satisfy
    \begin{equation}
        0<\alpha<\frac 12,\quad 0\leq \beta \leq  \min(k,(\eta-\epsilon)/\Kslope), \quad\text{and} \quad 0<\rho< \eta.
    \end{equation}
    where~$k$ is the wavenumber and~$\eta$ is the positive constant defined in Theorem~\ref{thm:split_w}. Further, we also assume that~$0<\epsilon < \eta$
\end{assumption}

We now check that the integral operators are defined for~$\sigma,\tau$ in these spaces.

\begin{lemma}\label{lem:op_def}
    If~$\sigma\in \imagdecay\alpha\beta$, then for any~$x_2 \in \GammaC$ and any~$\tGamma\in \cG$, the integrals defining~$\Ddiff_{\tGamma}[\sigma](x_2)$ and~$\Dpdiff_{\tGamma}[\sigma](x_2)$ converge. If $\tau\in \imagdecay{\alpha+\frac12}\beta$ then the integrals defining~$\Sdiff_{\tGamma}[\tau](x_2)$ and~$\Spdiff_{\tGamma}[\tau](x_2)$ converge.
\end{lemma}
\begin{proof}
    As all of the integrands are locally bounded, it is sufficient to check that the integrands decay fast enough at infinity. 
    Since~$\sigma\in \imagdecay\alpha\beta$, we have
    \begin{multline}
        |k_{\Ddiff_{ii}}(x_2,y_2) \sigma(y_2)|\leq K\frac{e^{ -k\Im (x_2+y_2)}}{(1+|x_2+y_2-2{X_2}|)^{\frac{3}{2}}}\|\sigma\|_{\alpha,\beta} \frac{e^{-\beta \Im y_2}}{(1+|y_2-{X_2}|)^\alpha}\\
        \leq \frac{K\|\sigma\|_{\alpha,\beta} }{(1+|x_2+y_2-2{X_2}|)^{\frac{3}{2}}(1+|y_2-{X_2}|)^\alpha}.
    \end{multline}
    Since~$\tGamma$ is monotonic and~$\frac32+\alpha>1$, this is integrable so $\Ddiff_{ii}[\sigma](x_2)$ exists. A similar calculation shows that~$\Ddiff_{ri}[\sigma](x_2), \Ddiff_{ir}[\sigma](x_2),$ and~$\Ddiff_{rr}[\sigma](x_2)$ exist. Since~$\Ddiff = \Ddiff_{rr}+\Ddiff_{ri}+\Ddiff_{ir}+\Ddiff_{ii}$, we also have that~$\Ddiff[\sigma](x_2)$ exists. 

The kernels~$k_\Dpdiff$ and~$k_\Spdiff$ satisfy the same estimates as~$k_\Ddiff$ and so~$\Dpdiff_{\tGamma}[\sigma](x_2)$ and $\Spdiff_{\tGamma}[\tau](x_2)$ also exist for any~$\sigma\in \imagdecay\alpha\beta$ and~$\tau\in \imagdecay{\alpha+\frac12}\beta$.
To see that $\Sdiff_{\tGamma}[\tau](x_2)$ exists, the only different term is~$\Sdiff_{\tGamma,ii}[\tau](x_2)$.
For~$\tau\in \imagdecay{\alpha+\frac12}\beta$ we have
    \begin{equation}
        |k_{\Sdiff_{ii}}(x_2,y_2) \tau(y_2)|\leq \frac{K\|\tau\|_{\alpha+\frac12,\beta} }{(1+|x_2+y_2-2{X_2}|)^{\frac{1}{2}}(1+|y_2-{X_2}|)^{\alpha+\frac12}},
    \end{equation}
    which is integrable because~$\alpha>0$. Thus~$\Sdiff_{\tGamma,ii}[\tau](x_2)$ exists. The rest of the proof that~$\Sdiff_{\tGamma}[\tau](x_2)$ exists for~$\tau\in\imagdecay{\alpha+\frac12}\beta$ is identical to the the other cases.
\end{proof}
We also verify that our operators map onto continuous functions that are analytic in the required region.
\begin{lemma}    \label{lem:layers_ana}
If $\sigma\in \imagdecay\alpha\beta$ and~$\tau\in \imagdecay{\alpha+\frac12}\beta$, then for any~$\tGamma\in \cG$, $\Ddiff_{\tGamma}[\sigma]$ and~$\Dpdiff_{\tGamma}[\sigma]$, $\Sdiff_{\tGamma}[\tau]$, and~$\Spdiff_{\tGamma}[\tau]$ are continuous on~$\GammaC$ and analytic in the interior of~$\GammaC$.
\end{lemma}
\begin{proof}
    The continuity follows from the continuity of the kernels. The analyticity can be derived from the analyticity of the kernels and an application of Morera's theorem.
\end{proof}

We now show that the operators~$\Ddiff_{\tGamma},\Sdiff_{\tGamma},\Dpdiff_{\tGamma},$ and~$\Spdiff_{\tGamma}$ are independent of~$\tGamma$.
\begin{theorem}\label{thm:path_indep}
    Let~$\sigma\in \imagdecay\alpha\beta$. If~$\tGamma,\overline \Gamma\in \cG$, then~$\Ddiff_{\tGamma}[\sigma](x_2)=\Ddiff_{\overline \Gamma}[\sigma](x_2)$ for all~$x_2\in \Gamma_\bbC$. Similar results hold for~$\Sdiff,\Dpdiff$ and~$\Spdiff$.
\end{theorem}
\begin{proof}
    Let~$\tGamma_M$ and~$\overline\Gamma_M$ be the truncation of~$\tGamma$ and~$\overline\Gamma$ to the region~$\Re y_2\leq M$. Let~$\Gamma_M$ be the straight line in~$\Gamma_\bbC$ connecting their endpoints and orientated to start at~$\overline\Gamma_M$. Since the kernel and density are analytic, we have that
    \begin{equation}
        \Ddiff_{\tGamma_M}[\sigma](x_2) - \Ddiff_{\overline\Gamma_M}[\sigma](x_2) =\Ddiff_{\Gamma_M}[\sigma](x_2) .
    \end{equation}
    By Lemma~\ref{lem:op_def} we have~$\Ddiff_{\tGamma}[\sigma](x_2) = \lim_{M\to\infty} \Ddiff_{\tGamma_M}[\sigma](x_2)$ and~$\Ddiff_{\overline\Gamma}[\sigma](x_2) = \lim_{M\to\infty} \Ddiff_{\overline\Gamma_M}[\sigma](x_2)$ so it is enough to show that~$\Ddiff_{\Gamma_M}[\sigma](x_2)\to 0$ for all~$x_2\in\Gamma_{\bbC}$. 
    
    If we  parameterize~$\Gamma_M$ by~$M + it$ and note~$|M + it-X_2|\geq |M-X_2|$, then
we have that
    \begin{equation}
        \max_{x_2\in \Gamma_\bbC}|k_{\Ddiff}(x_2,M + it)\sigma(M + i t)| 
        \leq K\|\sigma\|_{\alpha,\beta} \lp \frac{e^{ -k t}}{(1+M-{X_2})^{1/2}} +e^{ -\eta M}\rp e^{-\beta t}.
    \end{equation}
    We can thus bound
        \begin{equation}
        \left| \Ddiff_{\Gamma_M}[\sigma](x_2)\right|\leq \int_0^\infty K\|\sigma\|_{\alpha,\beta} \lp \frac{1}{(1+M-{X_2})^{1/2}} +e^{ -\eta M}\rp e^{-\beta t}\,\opd t,
    \end{equation}
    which is finite because~$\beta>0$ and also goes to zero as~$M\to \infty$.  Identical proofs show that that the other operators are independent of~$\tGamma\in \cG$.
\end{proof}
In light of the previous theorem, we can drop the subscripts and unambiguously define~$\Ddiff,\Sdiff,\Dpdiff$ and~$\Spdiff$ as an integral over any~$\tGamma\in\cG$. We now show that the operators are bounded.
\begin{theorem}\label{thm:bdd_ops}
    We have the following mapping properties
    \begin{multline}
\Ddiff:\imagdecay\alpha\beta\to \imagdecay{\alpha+\frac12}{\min(k,(\eta-\epsilon)/\Kslope)}, \quad \Sdiff:\imagdecay{\alpha+\frac12}\beta\to \imagdecay\alpha{ \min(k,(\eta-\epsilon)/\Kslope)}, \\\Dpdiff:\imagdecay\alpha k\to \imagdecay{\alpha+\frac12} {\min(k,(\eta-\epsilon)/\Kslope)}, \quad\text{and}\quad 
\Spdiff:\imagdecay{\alpha+\frac12}k\to \imagdecay{\alpha+1} {\min(k,(\eta-\epsilon)/\Kslope)}.
    \end{multline}
\end{theorem}
\begin{proof}
    Due to the path independence, we are free to integrate over~$\Gamma$, i.e. a subset of the real line. We begin by studying~$\Ddiff$. Since the bounds on~$k_{\Ddiff rr}, k_{\Ddiff ir},$ and~$k_{\Ddiff ri}$ are separable in~$x_2$ and $y_2$, the operators~$\Ddiff_{rr},\Ddiff_{ir},$ and~$\Ddiff_{ri}$ are clearly bounded from~$\imagdecay\alpha\beta$ to~$\realdecay\eta, \realdecay\eta,$ and~$\imagdecay{3/2}k$ respectively.  For example, if~$\sigma\in \imagdecay\alpha\beta$, then
    \begin{multline}
        \left|\Ddiff_{ir}[\sigma](x_2) \right| \leq \int_{X_2}^\infty K\frac{e^{-k\Im x_2-\eta t}}{(1+|x_2-{X_2}|)^{3/2}} \frac{\|\sigma\|_{\alpha,\beta} }{(1+t-{X_2})^\alpha}\, \opd t \\\leq  K\|\sigma\|_{\alpha,\beta}\frac{e^{-k\Im x_2}}{(1+|x_2-{X_2}|)^{3/2}} \frac{e^{-\eta {X_2}}}{\eta}.
    \end{multline}
    The remaining piece of~$\Ddiff$ is~$\Ddiff_{ii}$.
   If~$\sigma\in \imagdecay\alpha\beta$, then
    \begin{multline}
        |\Ddiff_{ii}[\sigma](x_2)| \leq \int_{X_2}^\infty \frac{K e^{-k\Im x_2  }}{(1+|x_2+t-2{X_2}|)^{3/2}}   \|\sigma\|_{\alpha,\beta} (t-{X_2})^{-\alpha}\, \opd t \\\leq  \int_0^\infty \frac{K \|\sigma\|_{\alpha,\beta} e^{-k \Im x_2}}{(1+|x_2-{X_2}|+s)^{3/2}s^\alpha} \,\opd s. \label{eq:A_ii_bd}
    \end{multline}
    Substituting~$s=(1+|x_2-{X_2}|)u$ gives
    \begin{equation}
        |\Ddiff_{ii}[\sigma](x_2)| \leq \frac{K \|\sigma\|_{\alpha,\beta} e^{-k \Im x_2}}{(1+|x_2-{X_2}|)^{-1+\alpha+3/2}} \int_0^\infty \frac{1}{\lp 1+u\rp^{3/2}u^\alpha} \,\opd u \leq  \frac{\tilde K \|\sigma\|_{\alpha,\beta} e^{-k \Im x_2}}{(1+|x_2-{X_2}|)^{\alpha+1/2}}.
    \end{equation}

    Thus~$\Ddiff_{ii}:\imagdecay\alpha\beta\to\imagdecay{\alpha+1/2}k$. Since~$\Ddiff =\Ddiff_{rr}+\Ddiff_{ir}+\Ddiff_{ri}+\Ddiff_{ii}$, Lemma~\ref{lem:embed} tells us that~$\Ddiff$ maps $\imagdecay\alpha\beta\to\imagdecay{\alpha+1/2}{\min(k,(\eta-\epsilon)/\Kslope)}$. The kernel~$k_{\Dpdiff}$ has identical bounds, so~$\Dpdiff$ is also bounded.

    We next turn to~$\Spdiff$. For~$\tau\in \imagdecay{\alpha+\frac12}\beta$~the equivalent expression to~\eqref{eq:A_ii_bd} is
    \begin{multline}
        |\Spdiff_{ii}[\sigma](x_2)| \leq \int_{X_2}^\infty \frac{K e^{-k\Im x_2  }}{(1+|x_2+t-2{X_2}|)^{3/2}}   \|\sigma\|_{\alpha+\frac12,\beta} (t-{X_2})^{-\alpha-\frac12}\, \opd t\\
        \leq   \frac{\tilde K \|\sigma\|_{\alpha+\frac12,\beta} e^{-k \Im x_2}}{(1+|x_2-{X_2}|)^{-1+\alpha+1/2+3/2}} .
    \end{multline}
    Thus~$\Spdiff_{ii}:\imagdecay{\alpha+\frac12}\beta\to \imagdecay{\alpha+1}{k}$ and the fact that~$\Spdiff:\imagdecay{\alpha+\frac12}\beta\to \imagdecay{\alpha+1}{\min(k,(\eta-\epsilon)/\Kslope)}$ follows.

    For~$\Sdiff$, the proof is nearly identical, except that~$k_{\Sdiff,ii}$ decays slower at infinity. Thus, if~$\tau\in \imagdecay{\alpha+\frac12}\beta$, then the equivalent expression to~\eqref{eq:A_ii_bd} is
    \begin{multline}
        |\Sdiff_{ii}[\sigma](x_2)| \leq \int_{X_2}^\infty \frac{K e^{-k\Im x_2  }}{(1+|x_2+t-2{X_2}|)^{1/2}}   \|\sigma\|_{\alpha+\frac12,\beta} (t-{X_2})^{-\alpha-\frac12}\, \opd t\\
        \leq   \frac{\tilde K \|\sigma\|_{\alpha+\frac12,\beta} e^{-k \Im x_2}}{(1+|x_2-{X_2}|)^{-1+\alpha+1/2+1/2}} .
    \end{multline}
    Thus~$\Sdiff_{ii}:\imagdecay{\alpha+\frac12}\beta\to \imagdecay{\alpha}k$. The fact that~$\Sdiff:\imagdecay{\alpha+\frac12}\beta\to \imagdecay{\alpha}{\min(k,(\eta-\epsilon)/\Kslope)}$ follows.
\end{proof}
This theorem tells us that the operator on the left hand side of
\begin{equation}
       \begin{pmatrix}
        \cI +\Ddiff&  \Sdiff \\ \Dpdiff & \cI + \Spdiff
    \end{pmatrix}\begin{pmatrix}
        \sigma\\\tau
    \end{pmatrix} = \begin{pmatrix}
        r_D\\ -r_N
    \end{pmatrix}. \label{eq:comp_IE} 
\end{equation}
is a bounded operator on~$\imagdecay\alpha\beta\,\oplus\,\imagdecay{\alpha+\frac12}\beta$. In the next section, we show that is in fact Fredholm index zero on the same space.

\subsection{Fredholm structure}\label{sec:Fredholm}
In this section, we work to show that the operator on the left hand side of~\eqref{eq:comp_IE} is Fredholm index zero. To show that an operator is compact on $\imagdecay\alpha\beta$, we need the following proposition, which was proved in~\cite{epstein2025complex}.
\begin{proposition}[Proposition 2 of \cite{epstein2025complex}]\label{prop:compact}
    Suppose~$0 <\alpha$,~$0<\tilde\alpha < \alpha'$, and $\beta \in \mathbb{R}$. If~$T:\imagdecay\alpha\beta\to \imagdecay{\alpha'}\beta$ be a bounded linear operator such that~$\partial_{x_2}T:\imagdecay\alpha\beta
    \to \imagdecay0{\beta'}$ boundedly, then $T:\imagdecay\alpha\beta\to \imagdecay{\tilde\alpha}\beta$ is a compact operator.
\end{proposition}

This proposition and the proof of Theorem~\ref{thm:bdd_ops} immediately gives that many of our operators are compact.
\begin{lemma}\label{eq:easy_comp}
    The operators
    \begin{multline}
        \Ddiff: \imagdecay\alpha\beta \to \imagdecay\alpha\beta, \quad \Sdiff_{rr}, \Sdiff_{ri},\Sdiff_{ir}: \imagdecay\alpha\beta \to \imagdecay{\alpha+\frac12}\beta\\
        \Dpdiff_{ir}, \Dpdiff_{ri},\Dpdiff_{ir}: \imagdecay{\alpha+\frac12}\beta \to \imagdecay{\alpha}\beta
        \quad\text{and}\quad \Spdiff: \imagdecay{\alpha+\frac12}\beta \to \imagdecay{\alpha+\frac12}\beta
    \end{multline}
    are compact.
\end{lemma}
Based on this lemma, we separate the left-hand side of~\eqref{eq:comp_IE} to expose the remaining operators:
\begin{equation}
    \cI + \begin{pmatrix}
        \Ddiff & \Sdiff\\ \Dpdiff & \Spdiff
    \end{pmatrix} = \cI +\cK_{\opnm{comp}} + \begin{pmatrix}
        0 & \Sdiff_{ii} \\ \Dpdiff_{ii} &0
    \end{pmatrix},
\end{equation}
where
\begin{equation}
     \cK_{\opnm{comp}}=\begin{pmatrix}
        \Ddiff & \Sdiff_{rr} + \Sdiff_{ri}+\Sdiff_{ir}\\ \Dpdiff_{rr}+\Dpdiff_{ri}+\Dpdiff_{ir} & \Spdiff
    \end{pmatrix}.
\end{equation}
The operator~$\cK_{\opnm{comp}}$ is a compact map from~$\imagdecay\alpha\beta\,\oplus\, \imagdecay{\alpha+\frac12}\beta$ to itself by the the previous lemma. Proposition~\ref{prop:compact} cannot be used to show that the remaining off diagonal operators~$\Sdiff_{ii}$ and~$\Dpdiff_{ii}$ are compact, as they swap the algebraic rates of decay~$\alpha$ and~$\alpha+\frac12$. Instead, we follow~\cite{epstein2023solvinga} and write:
\begin{multline}
    \begin{pmatrix}
        \cI & -\Sdiff_{ii} \\ 0 & \cI
    \end{pmatrix} \lp   \cI + \begin{pmatrix}
        \Ddiff & \Sdiff\\ \Dpdiff & \Spdiff
    \end{pmatrix}\rp \begin{pmatrix}
        \cI & 0\\ -\Dpdiff_{ii} & \cI
    \end{pmatrix} \\
    = \begin{pmatrix}
        \cI & -\Sdiff_{ii} \\ 0 & \cI
    \end{pmatrix}    \cK_{\opnm{comp}} \begin{pmatrix}
        \cI & 0\\ -\Dpdiff_{ii} & \cI
    \end{pmatrix}  +     \begin{pmatrix}
        \cI -\Sdiff_{ii}\Dpdiff_{ii} & 0 \\ 0 & \cI
    \end{pmatrix}
\end{multline}
Since~$\begin{pmatrix}
        \cI & -\Sdiff_{ii} \\ 0 & \cI
    \end{pmatrix} $ and~$\begin{pmatrix}
        \cI & 0\\ -\Dpdiff_{ii}  & \cI
    \end{pmatrix} $ are invertible and $\cK_{\opnm{comp}}$ is compact on~$\imagdecay\alpha\beta\,\oplus\, \imagdecay{\alpha+\frac12}\beta$, it is enough to show that $\begin{pmatrix}
        \cI -\Sdiff_{ii}\Dpdiff_{ii} & 0 \\ 0 & \cI
    \end{pmatrix} = \cI -\begin{pmatrix}
         \Sdiff_{ii}\Dpdiff_{ii} & 0 \\ 0 & 0
    \end{pmatrix} $ is Fredholm second kind. This fact will follow immediately from the following lemma.
\begin{lemma}
    The product
    $\Sdiff_{ii}\Dpdiff_{ii}$ is a compact map from~$\imagdecay\alpha\beta$ to~$\imagdecay{\alpha}\beta$.
\end{lemma}
\begin{proof}
The asymptotic form of~$k_{\Dpdiff,ii}$ and~$k_{\Sdiff,ii}$ are the same as those considered in Lemma 9 of \cite{epstein2025complex}. The same proof therefore shows that $\Sdiff_{ii}\Dpdiff_{ii}$ is a compact map from~$\imagdecay\alpha\beta$ to~$\imagdecay{\alpha}\beta$.
\end{proof}
Assembling the above arguments yields the following result.
\begin{theorem}
The operator
    $ \cI + \begin{pmatrix}
        \Ddiff & \Sdiff\\ \Dpdiff & \Spdiff
    \end{pmatrix}$ 
    is a Fredholm index zero operator on~$\imagdecay\alpha\beta\,\oplus\, \imagdecay{\alpha+\frac12}\beta$.
\end{theorem}
We now show that it is enough to enforce~\eqref{eq:comp_IE} on a single contour.
\begin{theorem}\label{thm:solve_complex}
    Suppose~$(r_D,r_N)\in \imagdecay\alpha\beta\,\oplus\, \imagdecay{\alpha+\frac12}\beta$. If~$(\sigma,\tau)\in \imagdecay\alpha\beta\,\oplus\, \imagdecay{\alpha+\frac12}\beta$  satisfies
    \begin{equation}\label{eq:complex_solve}
        \left.\lp \cI + \begin{pmatrix}
        \Ddiff & \Sdiff\\ \Dpdiff & \Spdiff
    \end{pmatrix}\rp \begin{pmatrix}
        \sigma\\ \tau
    \end{pmatrix} \right|_{\tGamma} =\left.\begin{pmatrix}
        r_D\\ -r_N
    \end{pmatrix} \right|_{\tGamma}
    \end{equation}
    on some~$\tGamma\in \cG$ then
    \begin{equation}\label{eq:comp_IE_2}
        \lp \cI + \begin{pmatrix}
        \Ddiff & \Sdiff\\ \Dpdiff & \Spdiff
    \end{pmatrix}\rp \begin{pmatrix}
        \sigma\\ \tau
    \end{pmatrix}=\begin{pmatrix}
        r_D\\ -r_N
    \end{pmatrix}
    \end{equation}
    on all of~$\Gamma_\bbC$.
\end{theorem}
\begin{proof}
    By Lemma~\ref{lem:layers_ana}, $\begin{pmatrix}
        \Ddiff & \Sdiff\\ \Dpdiff & \Spdiff
    \end{pmatrix} \begin{pmatrix}
        \sigma\\ \tau
    \end{pmatrix}$ is analytic on the interior of~$\Gamma_\bbC$. The sum~$\begin{pmatrix}
        \sigma\\ \tau
    \end{pmatrix}+ \begin{pmatrix}
        \Ddiff & \Sdiff\\ \Dpdiff & \Spdiff
    \end{pmatrix} \begin{pmatrix}
        \sigma\\ \tau
    \end{pmatrix}$ is thus analytic there. The result then follows from the identity theorem.
\end{proof}
The two previous theorems have the following important corollary.

\begin{corollary}\label{cor:unique}
    Let~$\cC^\alpha(\Gamma)$ be the set of continuous functions on~$\Gamma$ that are bounded by a multiple of~$(1+|x_2|)^{-\alpha}$.
    If
    \begin{equation}
        \left.\lp \cI + \begin{pmatrix}
        \Ddiff_\Gamma & \Sdiff_\Gamma\\ \Dpdiff_\Gamma & \Spdiff_\Gamma
    \end{pmatrix}\rp \begin{pmatrix}
        \sigma\\ \tau
    \end{pmatrix} \right|_{\Gamma} =0\label{eq:uni_real}
    \end{equation}
    has only the trivial solution in~$\cC^{\alpha}(\Gamma)\,\oplus\,\cC^{\alpha+\frac12}(\Gamma)$, then for all $(r_D,r_N)\in \imagdecay\alpha\beta\,\oplus\, \imagdecay{\alpha+\frac12}\beta$ there exists a unique $(\sigma,\tau)\in \imagdecay\alpha\beta\,\oplus\, \imagdecay{\alpha+\frac12}\beta$ that solves~\eqref{eq:complex_solve} for any~$\tGamma$. 
\end{corollary}
\begin{proof}
    We begin by proving uniqueness. By linearity, it is enough to show that if~$(\sigma,\tau)\in \imagdecay\alpha\beta\oplus\, \imagdecay{\alpha+\frac12}\beta$ satisfy
    \begin{equation}
        \lp \cI + \begin{pmatrix}
        \Ddiff & \Sdiff\\ \Dpdiff & \Spdiff
    \end{pmatrix}\rp \begin{pmatrix}
        \sigma\\ \tau
    \end{pmatrix} = 0
    \end{equation}
    on $\tGamma$, then~$\sigma\equiv\tau\equiv0$.
    By the previous theorem and the path independence of the operators we have that such $(\sigma,\tau)$ satisfy~\eqref{eq:uni_real}. Since $(\sigma|_\Gamma, \tau|_\Gamma)\in\cC^{\alpha}(\Gamma)\,\oplus\,\cC^{\alpha+\frac12}(\Gamma)$, the assumption gives that $(\sigma|_\Gamma, \tau|_\Gamma)=0$. Their analyticity thus gives that~$\sigma=\tau=0$ on all of~$\Gamma_\bbC$.

    As the operator in \eqref{eq:comp_IE_2} is Fredholm index zero on~$\Gamma_\bbC$, the assumptions of the corollary imply that~\eqref{eq:comp_IE_2} has a unique solution for all~$(r_D,r_N)\in \imagdecay\alpha\beta\,\oplus\, \imagdecay{\alpha+\frac12}\beta$. Theorem~\ref{thm:solve_complex} then gives that~\eqref{eq:complex_solve} has a unique solution.
\end{proof}
\begin{remark}
    The previous corollary asserts that if we could show uniqueness for the integral equation~\eqref{eq:real_IE}, then we would know that the complexified integral equation~\eqref{eq:comp_IE_2} has a unique solution for every appropriate right hand side. As with many transmission integral equations, the uniqueness on the real line would be an easy consequence of a uniqueness result for the original PDE~\eqref{eq:tot_PDE} (see~\cite{epstein2023solvinga}). 
    
    The uniqueness of this and related problems has been the subject of much research. For example, the uniqueness for the transmission version of this problem with straight interfaces was established in~\cite{epstein2024solving}. The recent work~\cite{humaikani2025rellich} studies, among other things, the junction of several straight open waveguides and one waveguide with periodic walls. That work establishes that there can be no solutions of the homogeneous problem that are trapped in the vicinity of a transmission junction. As the authors are not aware of a uniqueness theorem for the glued grating problem, we leave this as an open question.
\end{remark}

If its assumptions are satisfied, the previous corollary implies that the integral equation~\eqref{eq:comp_IE_2} has a unique solution in~$\imagdecay\alpha\beta\oplus\, \imagdecay{\alpha+\frac12}\beta$. For physically meaningful choices of~$(r_D,r_N)$, we will actually be able to have tighter control on the behavior of~$\sigma$ and~$\tau$. These results are discussed in Theorem \ref{thm:dens_asymp} below.

\section{Recovered solution}\label{sec:recovered_sol}
So far we have shown that the integral equation~\eqref{eq:real_IE} can be analytically continued to the Fredholm integral equation~\eqref{eq:complex_solve} and that the kernels and densities will be exponentially decaying along the complexified contour~$\tGamma$. In this section we show that the solutions of~\eqref{eq:complex_solve} can be used to recover the solution~$u_{L,R}$. We also show that the solutions satisfy the Sommerfeld radiation condition away from the boundaries~$\gamma_{L,R}$. Finally, we discuss physically meaningful data~$(r_D,r_N)$ that lies in the spaces~$\imagdecay\alpha\beta\oplus\imagdecay{\alpha+\frac12}\beta$. 

\begin{theorem}\label{thm:solve_pde}
Suppose~$(r_D,r_N)\in \imagdecay\alpha\beta\,\oplus\, \imagdecay{\alpha+\frac12}\beta$.
If~$(\sigma,\tau)\in \imagdecay\alpha\beta\,\oplus\, \imagdecay{\alpha+\frac12}\beta$ solves \eqref{eq:complex_solve}, then
\begin{equation}\label{eq:u_comp_rep}
    u_{L,R}(\bx) = \cD_{\tGamma,LR}[\sigma](\bx) + \cS_{\tGamma,LR}[\tau](\bx)
\end{equation}
exists for all~$\bx \in \Theta$ with~$x_1\neq 0$ and is independent of~$\tGamma$, provided~$\tGamma$ is such that~$y_2\in\tGamma$ is real when~$\Re y_2\leq x_2$. Further~$u_{L,R}$ and satisfies~\eqref{eq:ulr_def} in~$\Theta\setminus\Gamma$ and \eqref{eq:transmission_prob}. 
\end{theorem}
\begin{proof}
    To show that~$\cS_{\tGamma,LR}[\tau](\bx)$ exists, we split the operator~$\cS_{\tGamma,L,R}= \cS_{\tGamma,0}+\cS_{\tGamma,w_{L,R}}$, where~$\cS_0$ is standard Helmholtz single layer operator and~$\cS_{\tGamma,w_{L,R}}$ is the operator with kernel~$w_{\gamma_L,\gamma_R}$. The function~$\cS_{\tGamma,w_{L,R}}[\tau](\bx)$ is well-defined and finite for all~$\bx\in \Theta,$ which can be shown using an almost identical proof to that of Lemma~\ref{lem:op_def}. The kernel~$G(\bx-(0,y_2))$ is analytic in the region of interest and has the same decay rate as~$w_{\gamma_{L,R}ii}$ as~$y_2\to \infty$, so~$\cS_{\tGamma,0}[\tau](\bx)$ exists. We thus have that~$\cS_{\tGamma,LR}[\tau](\bx)$ exists. 

    The proof of Theorem~\ref{thm:path_indep} can be repeated to show that~$\cS_{\tGamma,w_{L,R}}[\tau](\bx)$ is independent of~$\tGamma$. For~$\cS_{\tGamma,0}[\tau](\bx)$, we must understand the analyticity of~$G(\bx-(0,y_2))$. The analytic continuation of the free-space kernel is given by
    \begin{equation}
        G(\bx-(0,y_2)) = \frac{i}4 H_0^{(1)}\lp k\sqrt{x_1^2 + (x_2-y_2)^2}\rp,
    \end{equation}
    which will be analytic as long as~$\Re y_2> x_2$. The argument can thus be repeated for $\cS_{\tGamma,0}[\tau](\bx)$ as long as~$\tGamma$ is real when~$\Re y_2\leq  x_2$. A similar argument can be applied to~$\cD_{\tGamma,LR}[\sigma](\bx)$.

    To check that $u_{L,R}$ satisfies the PDE, we note that the kernels of both integral operators satisfy~\eqref{eq:ulr_def}. As the integral converges uniformly for $\bx$ in any closed subset of $\overline{\Theta}\setminus\Gamma$ that doesn't include a corner of~$\gamma_{L,R}$, we have that $u_{L,R}$ also satisfies~\eqref{eq:ulr_def}. Finally, we have already noted that~$u_{L,R}$ will satisfy \eqref{eq:transmission_prob} because~$w_{L,R}(\bx,\by)$ is smooth for all~$\bx,\by\in \Theta$ away from any corners and the usual jump relations for the Helmholtz layer potentials.
\end{proof}

\begin{remark}
    The advantage of introducing the complexified contour~$\tGamma$ is that the kernels and densities will decay exponentially along~$\tGamma$. Indeed if~$\tGamma$ is a line of slope~$\Kslope$ outside some compact region then as~$y_2\to\infty$ along~$\tGamma$ the kernels and densities will decay as
    \begin{equation}
        O\lp e^{-k \Im y_2} + e^{-\eta\Re y_2}\rp \quad\text{and}\quad O\lp e^{-\beta\Im y_2 } \rp,
    \end{equation}
    respectively. Thus, for numerical purposes, truncation of the contour will produce easily controllable errors. To that end, let~$\tGamma_\epsilon$ be the truncation of~$\tGamma$ to the region~$e^{-\min (k \Im x_2, \eta\Re y_2)}<\epsilon$. It is not hard to show (see~\cite{epstein2025complex}) that under the same assumptions as Corollary~\ref{cor:unique} and for~$\epsilon$ sufficiently small, that there exists a unique solution of
\begin{equation}\label{eq:complex_solve_trunc}
        \left.\lp \cI + \begin{pmatrix}
        \Ddiff & \Sdiff\\ \Dpdiff & \Spdiff
    \end{pmatrix}\rp \begin{pmatrix}
        \sigma_\epsilon\\ \tau_\epsilon
    \end{pmatrix} \right|_{\tGamma_\epsilon} =\left.\begin{pmatrix}
        r_D\\ -r_N
    \end{pmatrix} \right|_{\tGamma_\epsilon}.
    \end{equation}
    Further, the arguments in \cite{epstein2025complex} can be used to show that, if~$\tGamma$ is real in the region~$\Re y_2< L$, then for every compact subset~$V$ of~$\Theta$ that is contained in the region~$\{x_2\leq L\}$ and does not contain a corner of~$\gamma_L$, there exists a~$C$ such that
    \begin{equation}
        \left|   u_{L,R}(\bx) - \cD_{\tGamma_\epsilon,LR}[\sigma_\epsilon](\bx) + \cS_{\tGamma_\epsilon,LR}[\tau_\epsilon](\bx) \right| < C \epsilon
    \end{equation}
    for all~$\bx\in V$.
\end{remark}

\subsection{Outgoing solutions}
In order for the solution~$u_{L,R}$ in \eqref{eq:u_comp_rep} to be physically meaningful, it must be outgoing. For scattering problems involving compact obstacles, the appropriate radiation condition is the Sommerfeld radiation condition. We recall that a field~$u$ is said to satisfy the Sommerfeld radiation condition if
\begin{equation}
    (\partial_r - ik) u(r\cos\theta,r\sin\theta) = o(r^{-1/2})
\end{equation}
as~$r\to \infty$, where the implicit constant is independent of angle. It is well-known result by Rellich that if the obstacle is compact and the Sommerfeld radiation condition holds uniformly in angle, then the solution will be unique, and will be the limiting absorption solution (see e.g. \cite{colton2013integral}). 
 For problems involving unbounded interfaces, such as~\eqref{eq:tot_PDE}, the Sommerfeld radiation condition are insufficient because trapped modes must be considered outgoing, even though they oscillate at frequencies other than $k$. The radiation condition for problems involving periodic interfaces is not well understood and so we simply show that the field~$u_{L,R}$ satisfies the Sommerfeld radiation condition in directions that point away from the boundary~$\gamma$. 

This proof is similar to the arguments presented in~\cite{epstein2023solvingb} for the junction of two leaky waveguides. In short, there are two steps. First, we work to show that the densities are outgoing. Following that, we show that this implies that the layer potentials satisfy the Sommerfeld radiation condition. We begin by showing that the functions in the range of the system matrix of~\eqref{eq:comp_IE} are outgoing in the following two lemmas.
\begin{lemma}\label{lem:apply_asym}
    Suppose $\phi$ is a continuous function that is identically zero when~$\Re x_2\leq 0$ and identically one when~$\Re x_2\geq \max(d_L,d_R)/2$. If
    $(\sigma,\tau)\in \imagdecay{\alpha}\beta \,\oplus\, \imagdecay{\alpha+\frac12}\beta$ for some~$\alpha,\beta>0$ and
\begin{equation}
        \begin{pmatrix}
        \tilde\sigma\\ \tilde \tau
    \end{pmatrix}=\begin{pmatrix}
        A & B\\ C &D
    \end{pmatrix} \begin{pmatrix}
        \phi\sigma\\ \phi\tau
    \end{pmatrix}
    \end{equation}
    then $(\tilde\sigma,\tilde\tau)\in \imagdecay{\frac12}{\min(k,(\eta-\epsilon)/\Kslope)} \,\oplus\, \imagdecay{\frac32}{\min(k,(\eta-\epsilon)/\Kslope)}$ and there exist constants~$a,b,$ and~$K$ such that
        \begin{equation}
        \left|\tilde\sigma(x_2) - \frac{ae^{ik x_2}}{\sqrt{x_2}} \right| \leq \frac{K e^{-k \Im x_2}}{|x_2|}+Ke^{-\eta\Re x_2} \quad \text{and}\quad         \left|\tilde\tau(x_2) - \frac{b e^{ik x_2}}{x_2^{3/2}} \right| \leq \frac{Ke^{-k \Im x_2}}{|x_2|^2}+Ke^{-\eta\Re x_2} \label{eq:apply_asym}
    \end{equation}
    for all~$x_2\in\Gamma_U$.
\end{lemma}
\begin{proof}
    Let $\tGamma\in\cG$ be a contour with slope~$\Kslope$ at infinity. We begin by studying $\tilde\sigma_\Ddiff(x_2) := \Ddiff[\phi\sigma]$
    By definition, we can split
    \begin{equation}
        \tilde\sigma_\Ddiff(x_2)= \tilde\sigma_{A,R}(x_2) - \tilde\sigma_{A,L}(x_2),
    \end{equation}
    where
    \begin{equation}
        \tilde\sigma_{\Ddiff,L,R}(x_2) =\int_{\tGamma} \partial_{x_1} w_{\gamma_{L,R}}(0,x_2;0,y_2) \phi(y_2) \sigma(y_2) \, \opd y_2.
    \end{equation}
    We can similarly define
    \begin{multline}
        \tilde\sigma_{\Ddiff,L,R,\xi}(x_2) =\int_{\tGamma} \partial_{x_1} w_{\xi,\gamma_{L,R}}(0,x_2;0,y_2) \phi(y_2) \sigma(y_2) \, \opd y_2\\
        =\int_{\tGamma}\lp \int_{\gamma_{L}}\partial_{x_1} G_\xi((0,x_2)-\bz)  \rho_{\xi,y_2}(\bz)   \,  \opd \bz \rp \phi(y_2)\sigma(y_2)\,  \opd y_2,
    \end{multline}
    where
    \begin{equation}
        \rho_{\xi,y_2} = -\cK^{-1}_{\xi,\gamma_L}[ \partial_{\bn(\cdot)} G_\xi(\cdot-(0,y_2))].
    \end{equation}
    As $\sigma$ decays exponentially along~$\tGamma $ and~$\phi$ is zero in the vicinity of~$\gamma_L$, the function
    \begin{equation}
        \tilde\rho_{\xi}(\bz) := \int_{\tGamma} \sigma(y_2) \rho_{\xi,y_2}(\bz) \phi(y_2)  \, \opd y_2
    \end{equation}
    is a bounded analytic function of~$\xi$ in~$V_{\gamma_L,\epsilon}$. We can thus use Fubini's theorem to see that
   \begin{equation}
        \tilde\sigma_{\Ddiff,L,\xi}(x_2) =\int_{\gamma_{L}}\partial_{x_1} G_\xi((0,x_2)-\bz)  \tilde\rho_{\xi}(\bz)   \,  \opd \bz.
    \end{equation}
    Since the integrals in~$\xi$ and~$\bz$ converge absolutely for finite~$x_2$, we can use Fubini's theorem again to see that
    \begin{equation}
        \tilde\sigma_{\Ddiff,L}(x_2) = \int_c \tilde\sigma_{\Ddiff,L,\xi}(x_2)\,\opd \xi = \int_c \int_{\gamma_{L}}\partial_{x_1} G_\xi((0,x_2)-\bz)  \tilde\rho_{\xi}(\bz)   \,  \opd \bz\,\opd \xi.
    \end{equation}
    This integral is of the same form as is considered in Appendix~\ref{app:far2near_as}. An estimate of the form~\eqref{eq:apply_asym} then follows from the same argument. Repeating the same argument for $\tilde\sigma_{A,R}$ and the other operators gives the desired result.
\end{proof}

\begin{lemma}\label{lem:close_apply_asym}
    Under the same assumptions as the previous lemma, if
\begin{equation}
        \begin{pmatrix}
        \tilde\sigma\\ \tilde \tau
    \end{pmatrix}=\begin{pmatrix}
        A & B\\ C &D
    \end{pmatrix} \begin{pmatrix}
       (1-\phi)\sigma\\ (1-\phi)\tau
    \end{pmatrix},
    \end{equation}
    then $(\tilde\sigma,\tilde\tau)\in \imagdecay{\frac12}{\min(k,(\eta-\epsilon)/\Kslope)} \,\oplus\, \imagdecay{\frac32}{\min(k,(\eta-\epsilon)/\Kslope)}$ and there exist constants~$a,b,$ and~$K$ such that
        \begin{equation}
        \left|\tilde\sigma(x_2) - \frac{a e^{ik x_2}}{\sqrt{x_2}} \right| \leq \frac{Ke^{-k \Im x_2}}{|x_2|}+Ke^{-\eta\Re x_2} \quad \text{and}\quad         \left|\tilde\sigma(x_2) - \frac{be^{ik x_2}}{x_2^{3/2}} \right| \leq \frac{Ke^{-k \Im x_2}}{|x_2|^2}+Ke^{-\eta\Re x_2} \label{eq:apply_asym_loc}
    \end{equation}
    for all~$x_2\in\Gamma_U$.
\end{lemma}
\begin{proof}
    This result follows directly from~\eqref{eq:wnearfar_bd} because~$1-\phi$ is compactly supported and~$b_l$ is continuous.
\end{proof}
We are now ready to prove that the solutions of~\eqref{eq:complex_solve} are outgoing in the sense of~\eqref{eq:apply_asym}.
\begin{theorem}\label{thm:dens_asymp}
    Suppose $(\sigma,\tau)\in \imagdecay{\alpha}{\beta} \,\oplus\, \imagdecay{\alpha+\frac12}{\beta}$ for some~$\alpha,\beta>0$. Further suppose~$(\sigma,\tau)$ solves \eqref{eq:complex_solve} with a right hand side $(r_D,r_N) \in \imagdecay{\frac12}{\min(k,(\eta-\epsilon)/\Kslope)} \,\oplus\, \imagdecay{\frac32}{\min(k,(\eta-\epsilon)/\Kslope)}$. If there are constants~$a,b,$ and~$K$ such that
        \begin{equation}
        \left|r_D(x_2) - \frac{ae^{ik x_2}}{\sqrt{x_2}} \right| \leq \frac{Ke^{-k \Im x_2}}{|x_2|}+Ke^{-\eta\Re x_2} \quad \text{and}\quad         \left|r_N(x_2) - \frac{be^{ik x_2}}{x_2^{3/2}} \right| \leq \frac{Ke^{-k \Im x_2}}{|x_2|^2}+Ke^{-\eta\Re x_2}\label{eq:rhs_asym}
    \end{equation}
    for all~$x_2\in\Gamma_U$, then the solution $(\sigma,\tau)$ is in $\imagdecay{\frac12}{\min(k,(\eta-\epsilon)/\Kslope)} \,\oplus\, \imagdecay{\frac32}{\min(k,(\eta-\epsilon)/\Kslope)}$ and there are constants~$\tilde a,\tilde b,$ and~$\tilde K$ such that
        \begin{equation}
        \left|\sigma(x_2) - \frac{\tilde ae^{ik x_2}}{\sqrt{x_2}} \right| \leq \frac{\tilde Ke^{-k \Im x_2}}{|x_2|}+\tilde Ke^{-\eta\Re x_2} \quad \text{and}\quad         \left|\tau(x_2) - \frac{\tilde be^{ik x_2}}{x_2^{3/2}} \right| \leq \frac{\tilde Ke^{-k \Im x_2}}{|x_2|^2}+\tilde K e^{-\eta\Re x_2} \label{eq:dens_asym}
    \end{equation}
    for all~$x_2\in\Gamma_U$.
\end{theorem}
\begin{proof}
    By Theorem~\ref{thm:solve_complex}, we can write
    \begin{equation}
    \begin{pmatrix}
        \sigma\\ \tau
    \end{pmatrix} 
       =\begin{pmatrix}
        r_D\\ -r_N
    \end{pmatrix}-\begin{pmatrix}
        \Ddiff & \Sdiff\\ \Dpdiff & \Spdiff
    \end{pmatrix}\begin{pmatrix}
        \sigma\\ \tau
    \end{pmatrix}.
    \end{equation}
    The result then follows from the previous two lemmas.
\end{proof}

Having shown that the densities~$\sigma$ and~$\tau$ are outgoing, we now work to show that the solution~$u_{L,R}$ satisfies the Sommerfeld radiation condition. We begin with the following lemma.

\begin{proposition}\label{prop:asym_rho}
    Let~$\sqrt{2}k<\frac\pi d$ and~$\rho_\xi(\bz)\in L^2(\gamma_{L,R})$ be an analytic function of~$\xi$ on~$V_{\gamma_L,\epsilon}$ with~$\|\rho_\xi(\bz)\|_{L^2(\gamma_L)}\leq K$ for some~$K>0$. If
    \begin{equation}
        v_{L,R}(\bx) := \int_{c} S_{\xi,\gamma_{L,R}}[\tilde \rho_\xi](\bx)\, \opd \xi,
    \end{equation}
    then there are functions~$b_{L,R}(\theta)$ and $c_{L,R}(\theta)$ such that if~$0<\theta<\pi $, then along any ray $\bx = r(\cos(\theta),\sin(\theta))=r\hat\theta$ we have
    \begin{equation}\label{eq:v_sommerfeld}
        \left|(\partial_r -ik)v_{L,R}(r\hat\theta) \right| \leq \frac{b_{L,R}(\theta)}{r^{3/2}}+c_{L,R}(\theta)e^{-\sin(\theta)\tilde \eta_{L,R}(\theta) r}
    \end{equation}
    for any $r>0$,
    where~$\tilde \eta_{L,R}(\theta)>0 $ and $\tilde \eta_{L,R}(\theta)=|\alpha(\pi/d_{L,R})|$ for~$|\cos\theta|< k\pi/d_{L,R}$.
     Further the functions~$b_{L,R},c_{L,R},$ and $\tilde\eta_{L,R}(\theta)$ are continuous on $(0,\pi)$.
\end{proposition}
We prove this in Appendix~\ref{app:som_ang}.
\begin{remark}
    The requirement~$\sqrt{2} k < \frac\pi {d_{L,R}}$ is a technical assumption required by our proof. Extensive numerical evidence suggests that it can be relaxed, and so Theorem~\ref{thm:u_som} should hold for all~$k<\frac\pi {d_{L,R}}$.
\end{remark}

\begin{lemma}\label{lem:close_u_asym}
    If~$\sqrt{2}k<\frac\pi{d_{L,R}}$, there are functions~$b_{L,R}(\theta,\by)$ and~$c_{L,R}(\theta,\by)$ such that if~$0<\theta<\pi $, then along any ray $r\hat\theta: = r(\cos(\theta),\sin(\theta))=$ we have
    \begin{equation}
        \left|(\partial_r -ik)w_{L,R}(r\hat\theta,\by)\right| \leq \frac{b_{L,R}(\theta,\by)}{r^{3/2}}+c_{L,R}(\theta,\by)e^{-\sin(\theta)\tilde \eta_{L,R}(\theta) r}
    \end{equation}
    for any $r>0$. Further the functions~$b_{L,R}$ and $c_{L,R}$ are continuous and an equivalent expression can be found for the~$y_1$ derivative of~$w_{L,R}$.
\end{lemma}
\begin{proof}
    As in Appendix~\ref{app:far2near_as}, we split
    \begin{equation}
        w_{L,R,\xi}(\bx,\by) =  \sum_n e^{i\xi_n x_1 +\alpha(\xi_n) x_2} S[\rho_{\xi,n}](\by) = \sum_n w_{L,R,\xi,n}(\bx,\by).
    \end{equation}
    We begin by studying the $n\neq 0$ terms. By the same proof as Lemma~\ref{lem:near2far_non} and~\ref{lem:vn_asym}, we can show that there is a constant~$C$ such that
\begin{equation}
    \left|\int_c\sum_{n\neq 0}  (\partial_r -ik) w_{L,R,\xi,n}(\bx,\by)\, \opd \xi \right| \leq C e^{-\tilde \eta_{L,R}(\theta) \sin \theta x_2}.
\end{equation}
The remaining piece can be bounded in the same way as the integral in Proposition~\ref{prop:asym_rho} because~$S_{\xi,\gamma_{L,R}}[\rho_{\xi,0}](\by)$ is a continuous function of~$\by\in\Omega_{L,R}$. An equivalent theorem holds for the~$y_1$ derivative of~$w_{L,R}$ because the derivative of~$S_{\xi,\gamma_{L,R}}[\rho_{\xi,0}](\by)$ is continuous up to the boundary~$\gamma_L$.
\end{proof}

These two results allow us to find asymptotics for our solution in directions pointing away from the gratings.
\begin{theorem}\label{thm:u_som}
    Suppose~$\sqrt{2}k<\frac\pi d$ and $(\sigma,\tau)\in \imagdecay{\frac12}{\min(k,(\eta-\epsilon)/\Kslope)} \,\oplus\, \imagdecay{\frac32}{\min(k,(\eta-\epsilon)/\Kslope)}$ satisfy \eqref{eq:dens_asym}. If $u$ is defined by \eqref{eq:u_comp_rep} then there are functions~$b_{L,R}$ and~$c_{L,R}$ such that
    \begin{equation}
       \left|(\partial_r - ik)u_{L,R}(r\hat\theta)\right| \leq \frac{b(\theta)}{r^{3/2}}+c(\theta)e^{-\sin(\theta)\tilde \eta_{L,R}(\theta) r}
    \end{equation}
    for all~$r>0$. Further, these functions are all smooth for~$\theta$ in the interior of~$(0,\pi)$, except for a jumpt discontinuity at~$\theta=\pi/2$.
\end{theorem}
\begin{proof}
As in the proof of Theorem~\ref{thm:solve_pde}, we split $u_{L,R}$ into the free space part $u_{0}$ and the scattered part~$u_{w_{L,R}}$. The Sommerfeld condition for the free-space part $u_0$ is given by Theorem 2 of~\cite{epstein2023solvingb}. For the remaining piece, we let $\tGamma \in \cG$ be a contour with slope $\Kslope$ at infinity. We then write
\begin{multline}
    u_{w_{L,R}}(\bx) = \int_{\tGamma} \phi(\by)\lp w_{L,R}(\bx,\by)\sigma(\by) + \partial_{y_1} w_{L,R}(\bx,\by)\tau(\by)\rp \, \opd y\\ + \int_{\tGamma} (1-\phi(\by))\lp w_{L,R}(\bx,\by)\sigma(\by) + \partial_{y_1} w_{L,R}(\bx,\by)\tau(\by) \rp\, \opd y =u_{w_{L,R},0}(\bx)+u_{w_{L,R},1}(\bx)  .  
\end{multline}
Using the same ideas as the proof of Lemma~\ref{lem:apply_asym}, we can write
\begin{equation}
    u_{w_{L,R},0}(\bx)=\int_c S_{\xi,\gamma_{L,R}}[\rho_\xi](\bx)\,\opd \xi,
\end{equation}
where~$\rho_\xi$ satisfies the assumptions of Proposition~\ref{prop:asym_rho}. The estimate for~$(\partial_r - ik)u_{w_{L,R},0}$ then follows from that proposition.

The remaining piece $(\partial_r - ik)u_{w_{L,R},1}$ can be bounded using Lemma~\ref{lem:close_u_asym} and fact that~$1-\phi$ is compactly supported.
\end{proof}

\subsection{Allowable data}\label{sec:data}
In this section, we illustrate a few examples of physically interesting examples of~$(r_D,r_N)$ and prove that the data lives in the required spaces.

First, we wish to solve \eqref{eq:tot_PDE} with right hand side~$f=\delta(\bx-\bz)$ for some~$\bz\in \Theta$ with~$z_{1}<0$. We can write~$u$ with
\begin{equation}
    u(\bx) = \begin{cases}
        u_L(\bx) + G_{L}(\bx,\bz) & x_1<0\\
        u_R(\bx) & x_1>0.
    \end{cases}
\end{equation}
This will solve~\eqref{eq:tot_PDE} provided~$u_{L,R}$ solves~\eqref{eq:transmission_prob} with $r_D =-G_{L}(\bx,\bz)$ and~$r_N = -\partial_{x_1}G_{L}(\bx,\bz)$. We therefore verify that this choice of~$(r_D,r_N)$ lies in the required spaces.

\begin{proposition}\label{prop:pt_allowable}
    If~$\bz\in \Omega\setminus\Gamma$, then~$r_D=-G_{L}(\bx,\bz)|_{\GammaC} \in \imagdecay{\frac12}k$ and $r_N=-\partial_{x_1}G_{L}(\bx,\bz)|_{\GammaC} \in \imagdecay{\frac32}k$ and satisfy \eqref{eq:rhs_asym}.
\end{proposition}
This proposition follows from Theorem \ref{thm:split_w} and the properties of~$G$ discussed in the proof of Theorem~\ref{thm:solve_pde}. In light of this theorem, we can apply the method described above to find~$u_{L,R}$ for this problem.

We next consider the case that the right hand side of \eqref{eq:tot_PDE} is a continuous and compactly supported function with support in the region~$x_1<0$. This case we let \begin{equation}\label{eq:u_in_out}
    u(\bx) = \begin{cases}
        u_L(\bx) +u_{\opnm{in}}(\bx) & x_1<0\\
        u_R(\bx) & x_1>0,
    \end{cases}
\end{equation}
where~$u_{\rm{in}}(\bx) = \int_{\opnm{supp}(f)}G_L(\bx,\bz)f(\bz)\,{\rm d}\bz$. It is easy to use Proposition~\ref{prop:pt_allowable} to see that
\begin{equation}
    (r_D,r_N)=(-u_{\opnm{in}}|_{\GammaC},-\partial_{x_1}u_{\opnm{in}}|_{\GammaC})\in  \imagdecay{\frac12}k \oplus \imagdecay{\frac32}k.
\end{equation} We can therefore also find~$u_{L,R}$ for this problem.

The final case that we consider is the case where~$u_{\opnm{in}}=v_{\tilde \xi_j}$ in~\eqref{eq:u_in_out} is a right-moving trapped mode for the left geometry.
\begin{proposition}\label{prop:mode_data}
    Let~$v_{\tilde \xi_j}$ be a trapped mode for the left geometry. Then~$r_D=-v_{\tilde\xi_j}|_{\GammaC}\in \realdecay\rho$ and $r_N=-\partial_{x_1}v_{\tilde\xi_j}|_{\GammaC} \in \realdecay\rho$ for~$\rho=|\alpha(\tilde \xi_1)|$. Further, $\realdecay\rho\subset \imagdecay{\frac12}{\beta}$ for some~$0<\beta< \min(k,(\eta-\epsilon)/\Kslope)$.
    
    Finally,~$r_D$ and $r_N$ satisfy~\eqref{eq:rhs_asym} with~$\eta$ replaced by~$\min(\eta,\rho)$ and if~$(\sigma,\tau)\in \imagdecay{\alpha}{\beta} \,\oplus\, \imagdecay{\alpha+\frac12}{\beta} $ solve~\eqref{eq:comp_IE_2}, then the results of Theorem~\ref{thm:dens_asymp} hold with~$\eta$ replaced by~$\min(\eta,\rho)$.
\end{proposition}
\begin{proof}
    It was observed~\cite{agocs2024trapped} that~$v_{\tilde \xi_j}$ is given by
    \begin{equation}
        v_{\tilde \xi_j}(\bx) = S_{\gamma,\tilde \xi_j}[\rho_j](\bx),
    \end{equation}
    where~$\rho_j$ is in the null space of~$\cK_{\gamma,\tilde \xi_j}$, the left hand side of \eqref{eq:quasi-IE}. Since~$\tilde \xi_j>k$, the formula \eqref{eq:dualsum} implies that the quasi-periodic Green's function and its~$x_1$ derivative decays exponentially with rate~$|\alpha(\tilde \xi_j)|$. The fact that~$r_D,r_N\in\realdecay\rho$ then follows.

    To see that~$(r_D,r_N)$ live in spaces compatible with Theorem~\ref{thm:solve_complex}, we note that Lemma~\ref{lem:embed} implies that
    \begin{equation}
        \realdecay\rho \subset \imagdecay{\frac12}{\frac{\rho-\epsilon}{\Kslope}}.
    \end{equation}
    We are thus free to choose any~$0<\beta< \min\lp k,(\min(\eta,\rho)-\epsilon)/\Kslope\rp$. Finally, the equation~\eqref{eq:rhs_asym} follows directly from the fact that~$r_D,r_N\in\realdecay\rho$. The asymptotics of~$\sigma$ and~$\tau$ follow from an equivalent proof.
\end{proof}

\section{Numerical Experiments}\label{sec:numerics}
In this section we illustrate the effectiveness of our integral equation formulation as a computational method for solving~\eqref{eq:tot_PDE}. We begin by discussing a method for discretizing~\eqref{eq:comp_IE_2} and~\eqref{eq:u_comp_rep}. We then present some accuracy tests for our method and show the solution of~\eqref{eq:tot_PDE} for a few different incoming fields. Finally, we finish the section by demonstrating that the solver can be extended to other setups such as two semi-infinite gratings separated by a compact transition region and a periodic layered media problem.
\subsection{Discretization}
Discretization proceeds in the following steps.
\begin{itemize}
    \item We discretize the boundary integral equation at each $\xi$ \eqref{eq:quasi-IE} using the modified Nystr\"om method implemented in the ChunkIE package \cite{chunkIE}. This package splits the boundaries~$\gamma_{L,R}$ into 16th order Gauss-Legendre panels and handles the kernel and corner singularities efficiently. To compute the layer potentials in~$w_{\xi,\gamma_{L,R}}$ \eqref{eq:w_xi_def}, we split~$G_\xi$ into a $\xi$-independent singular part~$G$ and a~$\xi$-dependent smooth part~$G_\xi - G$. The advantage of this splitting is that we can we reuse our adaptive integration of the singular part at each~$\xi$ and use the faster Gauss-Legendre quadrature on the~$\xi$-dependent part of the kernel. We also compress the far-field interaction using an analogue of the skeletonization approach described in~\cite{goodwill2024numerical}. If the source~$\by$ is within~$\max(d_L,d_R)/2$ of the interface, we add the image source discussed in Lemma~\ref{lem:w_image} to prevent the blow up of~$\rho_{\xi,\by}$.
    \item To integrate in~$\xi$, we use the contour satisfying~$\Im \xi = -0.3i \sin(d\,\Re \xi)$. We discretize this contour using a 60 point periodic trapezoid rule.
    \item We choose the contour~$\tilde\Gamma$ to be parameterized by
    \begin{equation}
        x_2 = t + i \psi(x_2),
    \end{equation}
    where~$\psi(x_2) = 20\opnm{erfc}((L-t)/5)$. The parameter~$L$ is chosen so that~$\psi$ is less than~$10^{-16}$ in the vicinity of~$X_2$ and~$\tGamma$ is truncated when~$\psi(x_2)=39$ (see \figref{fig:mode_dens}). 
    \item We discretize~$\tGamma$ using 16th order Gauss-Legendre panels and discretize \eqref{eq:comp_IE} using the corresponding smooth quadrature rule, which will be accurate since all involved kernels are smooth. To discretize the layer potentials in~\eqref{eq:u_comp_rep}, we split the Green's functions into free space part and the scattered part~\eqref{eq:G_dom_split}. We then use the same smooth quadrature rule to integrate the smooth scattered part and use adaptive integration for the singular free-space part. 
\end{itemize}
Further details of the discretization will be discussed in an upcoming manuscript.
\begin{figure}
    \centering
    \includegraphics[width=0.9\linewidth]{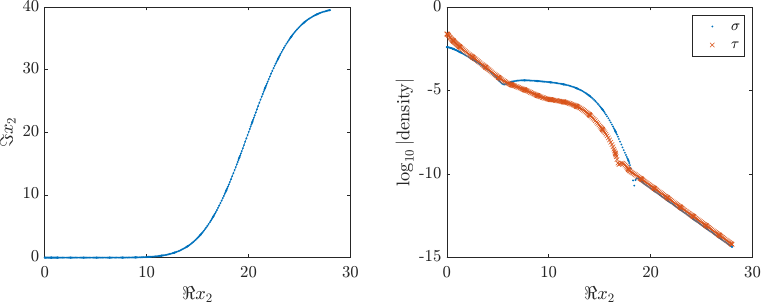}
    \caption{The left figure shows the real and imaginary parts of~$\tilde \Gamma$. The right shows the decay of the densities~$\sigma,\tau$ the solve~\eqref{eq:comp_IE_2} when the right hand side is associated to an incoming trapped mode on the left side.}
    \label{fig:mode_dens}
\end{figure}

\subsection{Single domain}
We demonstrate our solver for a single domain with a periodic wall ($d=1.3$), and wavenumber~$k=1$. A few examples of the domain Green's function are shown in \figref{fig:one-field}. 

To test the accuracy of our solver, we use an analytic solution test. Specifically, we observe that if~$\by$ is outside of~$\Omega$, then~$G_{\gamma}$ is zero, i.e.~$w_\gamma(\bx,\by) = -G_0(\bx-\by)$. To measure the error in our solver, we compute
\begin{equation}\label{eq:one_error}
\opnm{error}=|w_\gamma(\bx,\by) +G_0(\bx-\by)|/\max_{\bz} |G_0(\bz-\by)|,    
\end{equation}
where the maximum is taken over the plotting domain. The resulting errors are shown in \figref{fig:one_err} for~$\by=(0,-0.2)$. It is clear from this figure that the solution we have computed is accurate to at least 11 digits at every point that is inside the unit cell ($|x_1|<d/2$) and away from the corners of~$\gamma$.

\begin{figure}
    \centering
    \includegraphics[width=0.8\linewidth]{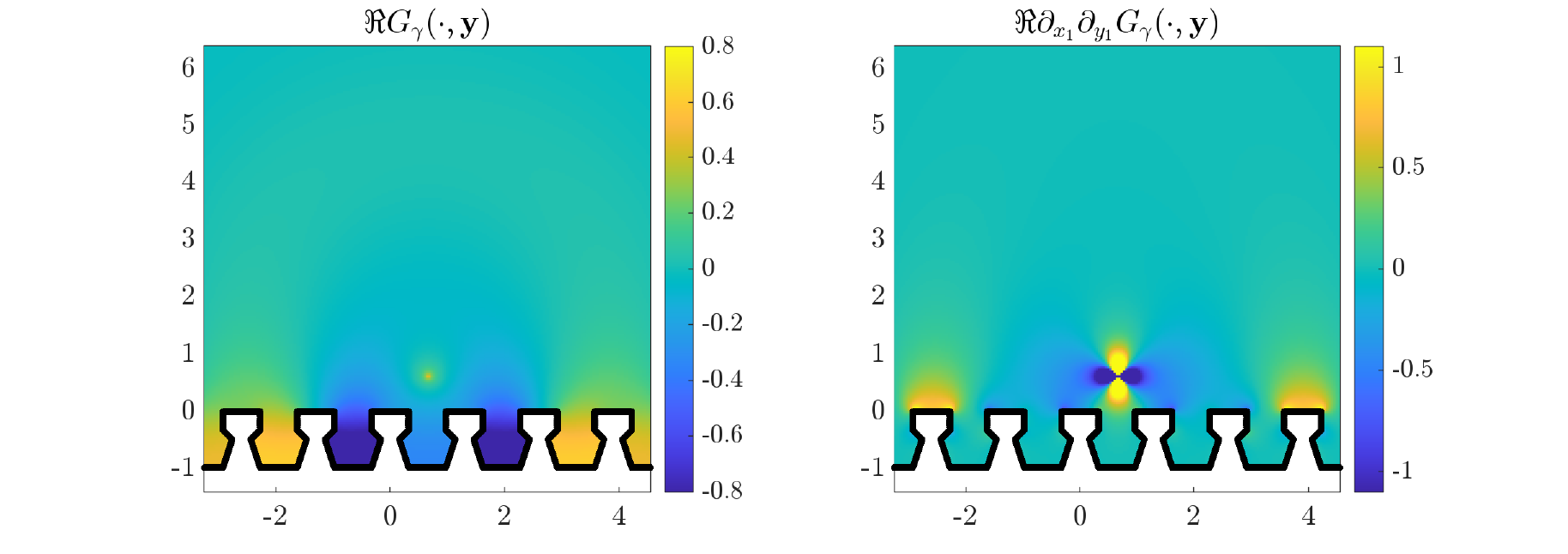}
    \caption{A few examples of the domain Green's function for a choice of~$\gamma$ with~$d=1.3$ and~$k=1$.}
    \label{fig:one-field}
\end{figure}

\begin{figure}
    \centering
    \includegraphics[width=0.4\linewidth]{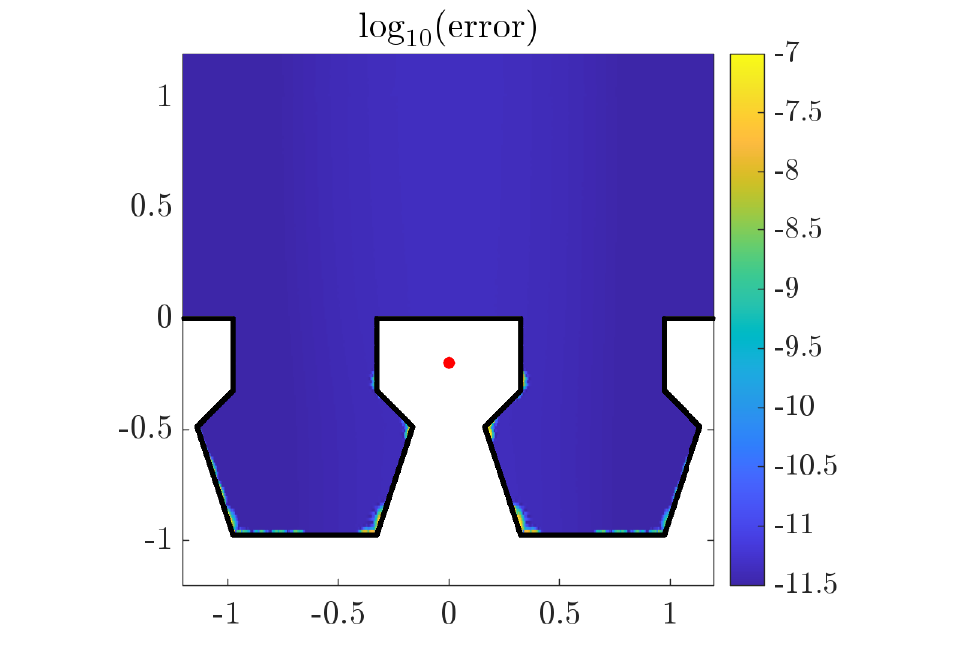}
    \caption{The error in the~\eqref{eq:one_error} for the point~$\by=(0,-0.2)$. We see that the solver is accurate to at least 11 digits everywhere away from the corners. The errors in the vicinity of the corners is due to the implementation of the RCIP method used by ChunkIE. The solver also makes errors outside the unit cell because of the~$G_{\xi}-G_0$ is nearly singular as~$\bx$ approaches a periodic copy of~$\gamma$.}
    \label{fig:one_err}
\end{figure}

\subsection{Glued staircases}
In this section we demonstrate our solver for the glued waveguide problem. Specifically, we consider two periodic boundaries~$\gamma_L$ and~$\gamma_R$ with periods~$d_L=1.6$ and~$d_R=1.3$. We begin looking for a~$u$ of the form~\eqref{eq:u_in_out} where~$u_{\opnm{in}}$ is given by a right-going trapped mode for the left domain. 

We find the trapped mode using the method described in~\cite{agocs2024trapped}. In short, we find the~$\tilde \xi_1$ such that~$\opnm{det} (2\cK_{\tilde \xi_1,\gamma})=0$. For this choice of~$\gamma_L$, the trapped mode occurs at~$\tilde \xi_1\approx 1.422265877314$ and is displayed in \figref{fig:trap_modes}. We then use the solver with the corresponding data in Proposition~\ref{prop:mode_data}. The resulting total field is displayed in~\figref{fig:mode_match}. The corresponding densities~$\sigma$ and~$\tau$ are shown in \figref{fig:mode_dens}.

To test the solver, we solve~\eqref{eq:comp_IE_2} using data from a left-going mode. In this case, it is easy to see that the true solution of~\eqref{eq:tot_PDE} will be~$u=0$, which corresponds to~$u_L=-u_{\opnm{in}}$ and~$u_R=0$. In~\figref{fig:mode_err} we plot the error in the sense~$\opnm{error}=|u|/\max_{\bz} |u_{\opnm{in}}|$, where the maximum is taken over the plotting domain. We see that the solver is accurate to at least 9 digits everywhere that the Green's function was observed to be in the previous section.

These codes were run on an Apple MacBook Pro with an M2 Max chip. Our solver took a total of 92 seconds to generate~\figref{fig:mode_match}, not including the time to find the trapped modes. Of this time, building the solvers for each side took 13 and 26 seconds respectively, building the system matrix~$\cI+\begin{pmatrix}
    \Ddiff&\Sdiff\\ \Dpdiff& \Spdiff
\end{pmatrix}$ took 11 seconds, and plotting the field~$u_{L,R}$ at 16 948 targets took 32 seconds. 

To conclude this section, we compute the field due to a points source at~$\by = (-5.6,0.6)$. We achieve this by picking the data~$(r_D,r_N)$ using the formulas given in Proposition~\ref{prop:pt_allowable}. The resulting field is shown in~\figref{fig:pt_match}.

\begin{figure}
    \centering
    \includegraphics[width=0.8\linewidth]{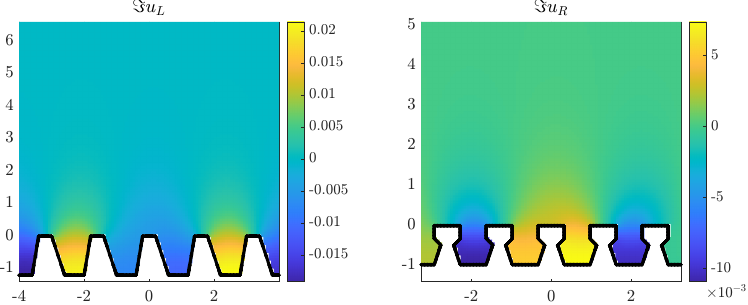}
    \caption{The trapped modes for our two geometries. The left geometry has a mode at $\tilde \xi_1=1.422265877314$. The right geometry has a mode at $\tilde \xi_1=1.47762000473$.}
    \label{fig:trap_modes}
\end{figure}

\begin{figure}
    \centering
    \includegraphics[width=0.7\linewidth]{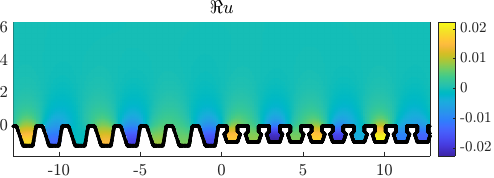}
    \caption{The solution of~\eqref{eq:tot_PDE} where~$u$ is of the form~\eqref{eq:u_in_out} where~$u_{\opnm{in}}$ is an incoming trapped mode on the left side.}
    \label{fig:mode_match}
\end{figure}

\begin{figure}
    \centering
    \includegraphics[width=0.7\linewidth]{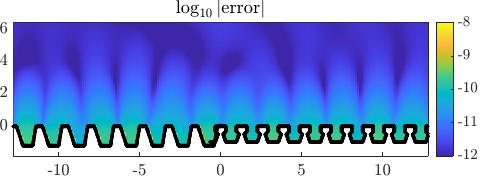}
    \caption{The results of our analytic solution test. In this test we pick~$u_{\opnm{in}}$ to be an outgoing mode on the left side. For this problem the exact solution is~$u\equiv0$ and we plot the error~$\opnm{error}=|u|/\max_{\bz} |u_{\opnm{in}}|$.}
    \label{fig:mode_err}
\end{figure}

\begin{figure}
    \centering
    \includegraphics[width=0.7\linewidth]{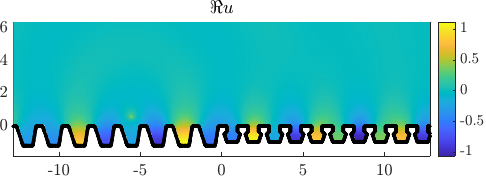}
    \caption{The field due to a point source at~$\by=(-5.6,0.6)$.}
    \label{fig:pt_match}
\end{figure}

\subsection{Compact transition region}\label{sec:egg}
Rather than abruptly changing from one periodic grating to another, it is more physically meaningful for there to be a compact transition region between the waveguides. In order to simulate this problem, we split our computational domain into three pieces: a compact transition region and the left and right halves. We then build an integral equation that forces continuity conditions one the interfaces connecting each region. Details on how this integral equation is constructed in the case of leaky waveguide are given in~\cite{goodwill2024numerical}.

In~\figref{fig:egg_match} we use this method to simulate the junction of the boundaries~$\gamma_L$ and~$\gamma_R$ from the previous example separated by a compact transition region. An equivalent analytic solution test to~\figref{fig:mode_err} indicates that our solver for this problem was accurate to 9 digits.

\begin{figure}
    \centering
    \includegraphics[width=0.7\linewidth]{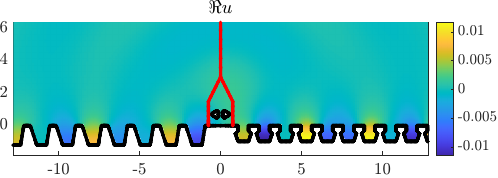}
    \caption{This figure shows our simulation of two semi-infinite gratings meeting with a compact transition region. The field is generated by a incoming trapped mode from the left on the left side. The red lines indicate the boundaries between the left, right, and center regions where we enforce the continuity conditions.}
    \label{fig:egg_match}
\end{figure}

\subsection{Transmission problems}\label{sec:transm}
We observed in Remark~\ref{rem:other_BC} that the method easily extends to other boundary conditions. In this section, we demonstrate the solver on a multilayer problem with period~$d_{L,R}=1.2$. We choose the layers to have wavenumbers 1, 7, 2, 7, and 1 in each layer. While some of these wavenumbers are larger than~$\pi/d_{L,R}$, all satisfy the requirement that the branch cuts of~$\alpha_i(\xi) = -\sqrt{i(\xi-k_i)}\sqrt{-i(\xi+k)}$ lie in the correct quadrants. We can therefore use our discrete inverse Floquet--Bloch transform without modification, though we do need to use~$80$ equispaced nodes to resolve~$w_{\xi,\gamma_{L,R}}$. We discretize each quasi-periodic problem using standard integral equation representations built out of the quasi-periodic Green's functions on each interface. We also impose homogeneous Neumann boundary conditions on some compact obstacles in the right half-space. 

The field due to a point source at (-11.4,1.4) is shown in \figref{fig:pt_layer_match}. An analytic solution test analogous to the test in~\figref{fig:mode_err} indicates that our solver for this problem was accurate to 6 digits.

\begin{figure}
    \centering
    \includegraphics[width=0.7\linewidth]{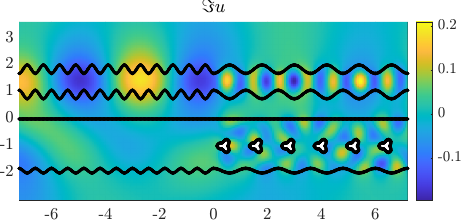}
    \caption{The field in a matched layered media problem due to a point source at~$\by=(-11.4,1.4)$.}
    \label{fig:pt_layer_match}
\end{figure}

\section{Concluding remarks} \label{sec:conclusion}
In this work, we showed how to use the domain Green's functions to reduce the scattering from two semi-infinite periodic gratings to an integral equation on the interface between the two halves of the computational domain. We then derived the asymptotic and analytic properties of the domain Green's functions. These properties allows us to build on the analysis in~\cite{epstein2025complex} and show that the integral equation could be analytically continued to a complex contour with a a Fredholm index zero operator. We also showed that the solution recovered through this method satisfies the Sommerfeld radiation condition in any cone away from the gratings. 

In order to complete the proof that our integral equation is well-posed it will be necessary to show that the solutions of~\eqref{eq:tot_PDE} are unique. As described in~\cite{epstein2024solving}, this requires the development of radiation condition that incorporate both radiated fields and the quasi-periodic trapped modes. Once uniqueness of the PDE solutions is established, we the proof of the uniqueness of solutions of~\eqref{eq:comp_IE_2} by proving that the representation~\eqref{eq:u_comp_rep} satisfies those outgoing conditions.

In this work, we assumed that the boundaries~$\gamma_{L,R}$ were flat in the vicinity of the~$x_2$-axis. In order to remove this assumption, it will be necessary to use a more detailed understanding of the domain Green's function for sources and targets near the boundary to show that the the kernels of our integral operators are not too singular. By using the method introduced in Section~\ref{sec:egg}, we can assume that the boundary is~$C^\infty$ in a neighborhood of the interface. It should also be straightforward to extend our analysis to this setup and to the case of non-parallel gratings.

There also exist a number of extensions of the results of Section~\ref{sec:transm}. One straightforward extension is to use more efficient evaluators for the quasi-periodic layered medium problem such as the method described in~\cite{zhang2022fast}. It would also be straightforward to use an extension of the adjoint Lippmann equation to simulate the junction of media with smoothly varying periodic wavenumber. The resulting method would be an analogue of the piecewise smooth waveguides considered in~\cite{epstein2025complex}.

\section*{Acknowledgements}
The authors would like to thank Alex Barnett, Charles Epstein, Manas Rachh, and Shidong Jiang for many useful discussions. J. Hoskins and T. Goodwill would like to thank the American Institute of Mathematics and, in particular, John Fry for hosting them on Bock Cay during the SQuaREs program, where parts of this work were completed.

\bibliography{references.bib}

@article{linton2010lattice,
  title={Lattice sums for the {H}elmholtz equation},
  author={Linton, Chris M},
  journal={SIAM review},
  volume={52},
  number={4},
  pages={630--674},
  year={2010},
  publisher={SIAM}
}

@article{agocs2024trapped,
  title={Trapped acoustic waves and raindrops: High-order accurate integral equation method for localized excitation of a periodic staircase},
  author={Agocs, Fruzsina J and Barnett, Alex H},
  journal={Journal of Computational Physics},
  volume={519},
  pages={113383},
  year={2024},
  publisher={Elsevier}
}

@book{Dunford1958,
	address = {New York},
	author = {Dunford, Nelson. and Schwartz, Jacob T.},
	isbn = {0471226394 (v. 3)},
	publisher = {Interscience Publishers},
	title = {Linear operators},
	year = {1958}
}

@article{epstein2023solvinga,
  title={Solving the Transmission Problem for Open Wave-Guides, {I} {F}undamental Solutions and Integral Equations},
  author={Epstein, Charles L},
  journal={arXiv preprint arXiv:2302.04353},
  year={2023}
}

@article{epstein2023solvingb,
  title={Solving the Transmission Problem for Open Wave-Guides, {II} {O}utgoing Estimates},
  author={Epstein, Charles L},
  journal={arXiv preprint arXiv:2310.05816},
  year={2023}
}

@article{epstein2024solving,
  title={Solving the Scattering Problem for Open Wave-Guides, {III}: Radiation Conditions and Uniqueness},
  author={Epstein, Charles L and Mazzeo, Rafe},
  journal={arXiv preprint arXiv:2401.04674},
  year={2024}
}

@article{epstein2025complex,
  title={Complex scaling for open waveguides},
  author={Epstein, Charles L and Goodwill, Tristan and Hoskins, Jeremy and Quinn, Solomon and Rachh, Manas},
  journal={arXiv preprint arXiv:2506.10263},
  year={2025}
}

@article{goodwill2024numerical,
  title={A numerical method for scattering problems with unbounded interfaces},
  author={Goodwill, Tristan and Epstein, Charles L},
  journal={arXiv preprint arXiv:2411.11204},
  year={2024}
}

@book{colton2013integral,
  title={Integral equation methods in scattering theory},
  author={Colton, David and Kress, Rainer},
  year={2013},
  publisher={SIAM},address = {Philadelphia}
}

@article{epstein2024coordinate,
  title={Coordinate complexification for the Helmholtz equation with {D}irichlet boundary conditions in a perturbed half-space},
  author={Epstein, Charles L and Greengard, Leslie and Hoskins, Jeremy and Jiang, Shidong and Rachh, Manas},
  journal={arXiv preprint arXiv:2409.06988},
  year={2024}
}

@book{kress1999linear,
  title={Linear integral equations},
  author={Kress, Rainer},
  volume={82},
  year={1999},
  publisher={Springer},address = {New York}
}

@article{yasumoto2002efficient,
  title={Efficient calculation of lattice sums for free-space periodic Green's function},
  author={Yasumoto, Kiyotoshi and Yoshitomi, Kuniaki},
  journal={IEEE Transactions on Antennas and Propagation},
  volume={47},
  number={6},
  pages={1050--1055},
  year={2002},
  publisher={IEEE}
}

@misc{chunkIE,
title= {chunk{IE}: a {MATLAB} integral equation toolbox},
author = {Askham, T and Rachh, M and O'Neil, M and Hoskins, J and Fortunato, D and Jiang, S and Fryklund, F and Goodwill, T and Yang, Hai and Zhu, H},
  publisher = {GitHub},
  journal = {GitHub repository},
howpublished = {\url{https://github.com/fastalgorithms/chunkie}},
year = {2024}
}

@article{bruno2014rapidly,
  title={Rapidly convergent two-dimensional quasi-periodic Green function throughout the spectrum—including {W}ood anomalies},
  author={Bruno, Oscar P and Delourme, B{\'e}rang{\`e}re},
  journal={Journal of Computational Physics},
  volume={262},
  pages={262--290},
  year={2014},
  publisher={Elsevier}
}

@article{desanto1998theoretical,
  title={Theoretical and computational aspects of scattering from rough surfaces: one-dimensional perfectly reflecting surfaces},
  author={DeSanto, J and Erdmann, G and Hereman, W and Misra, M},
  journal={Waves in Random Media},
  volume={8},
  number={4},
  pages={385},
  year={1998},
  publisher={IOP Publishing}
}

@article{arens2006integral,
  title={On integral equation and least squares methods for scattering by diffraction gratings},
  author={Arens, Tilo and Chandler-Wilde, Simon N and DeSanto, John A},
  journal={Communications in Computational Physics},
  volume={1},
  number={6},
  pages={1010--1042},
  year={2006}
}

@article{barnett2011new,
  title={A new integral representation for quasi-periodic scattering problems in two dimensions},
  author={Barnett, Alex and Greengard, Leslie},
  journal={BIT Numerical mathematics},
  volume={51},
  number={1},
  pages={67--90},
  year={2011},
  publisher={Springer}
}

@article{zhang2022fast,
  title={A fast direct solver for two dimensional quasi-periodic multilayered media scattering problems, Part {II}},
  author={Zhang, Yabin and Gillman, Adrianna},
  journal={arXiv preprint arXiv:2204.06629},
  year={2022}
}

@article{bruno2017rapidly,
  title={Rapidly convergent quasi-periodic Green functions for scattering by arrays of cylinders—including {W}ood anomalies},
  author={Bruno, Oscar P and Fernandez-Lado, Agustin G},
  journal={Proceedings of the Royal Society A: Mathematical, Physical and Engineering Sciences},
  volume={473},
  number={2199},
  pages={20160802},
  year={2017},
  publisher={The Royal Society Publishing}
}

@article{aylwin2020properties,
  title={On the properties of quasi-periodic boundary integral operators for the {H}elmholtz equation},
  author={Aylwin, Rub{\'e}n and Jerez-Hanckes, Carlos and Pinto, Jos{\'e}},
  journal={Integral Equations and Operator Theory},
  volume={92},
  number={2},
  pages={17},
  year={2020},
  publisher={Springer}
}

@article{pinto2021fast,
  title={Fast solver for quasi-periodic 2D-{H}elmholtz scattering in layered media},
  author={Pinto, Jos{\'e} and Aylwin, Ruben and Jerez-Hanckes, Carlos},
  journal={ESAIM: Mathematical Modelling and Numerical Analysis},
  volume={55},
  number={5},
  pages={2445--2472},
  year={2021},
  publisher={EDP Sciences}
}

@article{strauszer2023windowed,
  title={Windowed {G}reen function method for wave scattering by periodic arrays of 2{D} obstacles},
  author={Strauszer-Caussade, Thomas and Faria, Luiz M and Fernandez-Lado, Agust{\'\i}n and P{\'e}rez-Arancibia, Carlos},
  journal={Studies in Applied Mathematics},
  volume={150},
  number={1},
  pages={277--315},
  year={2023},
  publisher={Wiley Online Library}
}

@article{meng2023new,
  title={A new periodic {FM-BEM} for solving the acoustic transmission problems in periodic media},
  author={Meng, Wenhui},
  journal={Engineering Analysis with Boundary Elements},
  volume={154},
  pages={54--63},
  year={2023},
  publisher={Elsevier}
}

@article{munk_plane-wave_1979,
	title = {Plane-wave expansion for arrays of arbitrarily oriented piecewise linear elements and its application in determining the impedance of a single linear antenna in a lossy half-space},
	volume = {27},
	copyright = {https://ieeexplore.ieee.org/Xplorehelp/downloads/license-information/IEEE.html},
	issn = {0096-1973},
	url = {http://ieeexplore.ieee.org/document/1142089/},
	doi = {10.1109/TAP.1979.1142089},
	language = {en},
	number = {3},
	urldate = {2025-08-09},
	journal = {IEEE Transactions on Antennas and Propagation},
	author = {Munk, B. and Burrell, G.},
	month = may,
	year = {1979},
	pages = {331--343},
}

@article{rana1981current,
  title={Current distribution and input impedance of printed dipoles},
  author={Rana, I and Alexopoulos, N},
  journal={IEEE Transactions on Antennas and Propagation},
  volume={29},
  number={1},
  pages={99--105},
  year={1981},
  publisher={IEEE}
}

@article{capolino2005mode,
  title={Mode excitation from sources in two-dimensional {EBG} waveguides using the array scanning method},
  author={Capolino, Filippo and Jackson, David R and Wilton, Donald R},
  journal={IEEE Microwave and Wireless Components Letters},
  volume={15},
  number={2},
  pages={49--51},
  year={2005},
  publisher={IEEE}
}

@article{lechleiter2017convergent,
  title={A convergent numerical scheme for scattering of aperiodic waves from periodic surfaces based on the {F}loquet-{B}loch transform},
  author={Lechleiter, Armin and Zhang, Ruming},
  journal={SIAM Journal on Numerical Analysis},
  volume={55},
  number={2},
  pages={713--736},
  year={2017},
  publisher={SIAM}
}

@article{zhang2021numerical,
  title={Numerical methods for scattering problems in periodic waveguides},
  author={Zhang, Ruming},
  journal={Numerische Mathematik},
  volume={148},
  number={4},
  pages={959--996},
  year={2021},
  publisher={Springer}
}

@article{petit1980electromagnetic,
  title={Electromagnetic Theory of Gratings},
  author={Petit, Roger},
  journal={Topics in Current Physics},
  volume={22},
  year={1980}
}

@article{millar1973rayleigh,
  title={The {R}ayleigh hypothesis and a related least-squares solution to scattering problems for periodic surfaces and other scatterers},
  author={Millar, RF},
  journal={Radio Science},
  volume={8},
  number={8-9},
  pages={785--796},
  year={1973},
  publisher={Wiley Online Library}
}

@article{pierre,
    title = {Time-harmonic wave propagation in junctions of two periodic half-spaces},
    author = {P. Amenoagbadji and S. Fliss and P. Joly},
    journal = {Pure and Applied Analysis},
    year = {2025},
    volume = {7},
    number = {2},
    pages = {299–357}
}

@article{fliss2021Dirichlet,
  title={A {D}irichlet-to-{N}eumann approach to the mathematical and numerical analysis in waveguides with periodic outlets at infinity},
  author={Fliss, Sonia and Joly, Patrick and Lescarret, Vincent},
  journal={Pure and Applied Analysis},
  volume={3},
  number={3},
  pages={487--526},
  year={2021},
  publisher={Mathematical Sciences Publishers}
}

@article{amenoagbadji2023wave,
  title={Wave propagation in one-dimensional quasiperiodic media},
  author={Amenoagbadji, Pierre and Fliss, Sonia and Joly, Patrick},
  journal={arXiv preprint arXiv:2301.01159},
  year={2023}
}

@article{qiu2023bifurcation,
  title={On the bifurcation of a {D}irac point in a photonic waveguide without band gap opening},
  author={Qiu, Jiayu and Zhang, Hai},
  journal={arXiv preprint arXiv:2310.17964},
  year={2023}
}

@article{fliss2020time,
  title={Time harmonic wave propagation in one dimensional weakly randomly perturbed periodic media},
  author={Fliss, Sonia and Giovangigli, Laure},
  journal={SN Partial Differential Equations and Applications},
  volume={1},
  number={6},
  pages={40},
  year={2020},
  publisher={Springer}
}

@phdthesis{fliss2019wave,
  title={Wave propagation in periodic media: mathematical analysis and numerical simulation},
  author={Fliss, Sonia},
  year={2019},
  school={Universit{\'e} Paris Sud (Paris 11)}
}

@article{turc2025,
  title={Domain decomposition multiple scattering solvers by semi-infinite and infinite arrays of discrete identical scatterers in two dimensions},
  author={Petropoulos, Peter G and Turc, Catalin},
  journal={Philosophical Transactions A},
  volume={383},
  number={2303},
  pages={20240355},
  year={2025},
  publisher={The Royal Society}
}

@article{klindworth2014numerical,
  title={Numerical realization of {D}irichlet-to-{N}eumann transparent boundary conditions for photonic crystal wave-guides},
  author={Klindworth, Dirk and Schmidt, Kersten and Fliss, Sonia},
  journal={Computers \& Mathematics with Applications},
  volume={67},
  number={4},
  pages={918--943},
  year={2014},
  publisher={Elsevier}
}

@article{watanabe2012accurate,
  title={Accurate analysis of electromagnetic scattering from periodic circular cylinder array with defects},
  author={Watanabe, Koki and Nakatake, Yoshimasa and Pi{\v{s}}tora, Jarom{\'\i}r},
  journal={Optics Express},
  volume={20},
  number={10},
  pages={10646--10657},
  year={2012},
  publisher={Optical Society of America}
}

@article{liu2017electromagnetic,
  title={Electromagnetic modeling of damaged single-layer fiber-reinforced laminates},
  author={Liu, Zicheng and Li, Changyou and Lesselier, Dominique and Zhong, Yu},
  journal={IEEE Transactions on Antennas and Propagation},
  volume={65},
  number={4},
  pages={1855--1866},
  year={2017},
  publisher={IEEE}
}

@article{liu2020electromagnetic,
  title={Electromagnetic modeling of damaged fiber-reinforced laminates},
  author={Liu, Zicheng and Li, Changyou and Zhong, Yu and Lesselier, Dominique},
  journal={Journal of Computational Physics},
  volume={409},
  pages={109318},
  year={2020},
  publisher={Elsevier}
}

@article{liu2018fast,
  title={Fast full-wave analysis of damaged periodic fiber-reinforced laminates},
  author={Liu, Zicheng and Li, Changyou and Lesselier, Dominique and Zhong, Yu},
  journal={IEEE Transactions on Antennas and Propagation},
  volume={66},
  number={7},
  pages={3540--3547},
  year={2018},
  publisher={IEEE}
}

@article{liu2018electromagnetic,
  title={Electromagnetic imaging of damages in fibered layered laminates via equivalence theory},
  author={Liu, Zicheng and Lesselier, Dominique and Zhong, Yu},
  journal={IEEE Transactions on Computational Imaging},
  volume={4},
  number={2},
  pages={219--227},
  year={2018},
  publisher={IEEE}
}

@article{virk2021fast,
  title={Fast computation of scattering by isolated defects in periodic dielectric media},
  author={Virk, Kuljit S},
  journal={Journal of the Optical Society of America B},
  volume={38},
  number={6},
  pages={1763--1775},
  year={2021},
  publisher={OSA}
}

@article{perez2018domain,
  title={Domain decomposition for quasi-periodic scattering by layered media via robust boundary-integral equations at all frequencies},
  author={P{\'e}rez-Arancibia, Carlos and Shipman, Stephen and Turc, Catalin and Venakides, Stephanos},
  journal={arXiv preprint arXiv:1801.09094},
  year={2018}
}

@article{fliss2009exact,
  title={Exact boundary conditions for time-harmonic wave propagation in locally perturbed periodic media},
  author={Fliss, Sonia and Joly, Patrick},
  journal={Applied Numerical Mathematics},
  volume={59},
  number={9},
  pages={2155--2178},
  year={2009},
  publisher={Elsevier}
}

@article{fliss2010exact,
  title={Exact boundary conditions for wave propagation in periodic media containing a local perturbation},
  author={Fliss, Sonia and Joly, Patrick and Li, Jing-Rebecca},
  journal={Wave Propagation in Periodic Media-Analysis, Numerical Techniques and practical Applications, Progress in Computational Physics},
  volume={1},
  pages={108--134},
  year={2010}
}

@article{jones1953eigenvalues,
    title={The eigenvalues of $\nabla^2 u + \lambda u = 0$ when the boundary conditions are given on semi-infinite domains},
    volume={49},
    DOI={10.1017/S0305004100028875},
    number={4},
    journal={Mathematical Proceedings of the Cambridge Philosophical Society},
    author={Jones, D. S.},
    year={1953}, pages={668–684}
}

@article{linton2007embedded,
title = {Embedded trapped modes in water waves and acoustics},
journal = {Wave Motion},
volume = {45},
number = {1},
pages = {16-29},
year = {2007},
note = {Special Issue on Localization of Wave Motion},
issn = {0165-2125},
doi = {https://doi.org/10.1016/j.wavemoti.2007.04.009},
url = {https://www.sciencedirect.com/science/article/pii/S0165212507000376},
author = {C.M. Linton and P. McIver},
}

@book{bornwolf, place={Cambridge}, edition={7}, title={Principles of Optics: Electromagnetic Theory of Propagation, Interference and Diffraction of Light}, publisher={Cambridge University Press}, author={Born, Max and Wolf, Emil and Bhatia, A. B. and Clemmow, P. C. and Gabor, D. and Stokes, A. R. and Taylor, A. M. and Wayman, P. A. and Wilcock, W. L.}, year={1999}, address={Oxford}}

@book{phot_cryst,
author = {J. D. Joannopoulos and S. G. Johnson and R. D. Meade and J. N. Winn}, 
title = {Photonic Crystals: Molding the Flow of Light},
edition = {2nd ed},
publisher = {Princeton Univ. Press},
year ={2007},
address={Princeton, NJ}}

@book{diff_grat,
author = {Calvin H. Wilcox}, 
title = {Scattering Theory for Diffraction Gratings},
publisher = {Springer-Verlag},
year ={1984},
address={New York, New York}}

@Article{acoust_absorp,
AUTHOR = {Herrero-Durá, Iván and Cebrecos, Alejandro and Picó, Rubén and Romero-García, Vicente and García-Raffi, Luis Miguel and Sánchez-Morcillo, Víctor José},
TITLE = {Sound Absorption and Diffusion by 2{D} Arrays of {H}elmholtz Resonators},
JOURNAL = {Applied Sciences},
VOLUME = {10},
YEAR = {2020},
NUMBER = {5},
ARTICLE-NUMBER = {1690},
URL = {https://www.mdpi.com/2076-3417/10/5/1690},
ISSN = {2076-3417},
DOI = {10.3390/app10051690}
}

@article{ammari,
year = {2004},
journal = {SIAM J. APPL. MATH.},
author = {Habib Ammari and Fadil Santosa},
title = {GUIDED WAVES IN A PHOTONIC BANDGAP STRUCTURE},
volume = {64},
number = {6},
pages = {2018–2033}}

@article{zayats2005nano,
  title={Nano-optics of surface plasmon polaritons},
  author={Zayats, Anatoly V and Smolyaninov, Igor I and Maradudin, Alexei A},
  journal={Physics reports},
  volume={408},
  number={3-4},
  pages={131--314},
  year={2005},
  publisher={Elsevier}
}

@article{fano1941anomalous,
author = {U. Fano},
journal = {J. Opt. Soc. Am.},
number = {3},
pages = {213--222},
publisher = {Optica Publishing Group},
title = {The Theory of Anomalous Diffraction Gratings and of Quasi-Stationary Waves on Metallic Surfaces (Sommerfeld's Waves)},
volume = {31},
month = {Mar},
year = {1941},
url = {https://opg.optica.org/abstract.cfm?URI=josa-31-3-213},
doi = {10.1364/JOSA.31.000213}
}

@book{enoch2012plasmonics,
  title={Plasmonics: from basics to advanced topics},
  author={Enoch, Stefan and Bonod, Nicolas},
  volume={167},
  year={2012},
  publisher={Springer}, address = {Heidelberg}
}

@article{ebbesen1998extraordinary,
  title={Extraordinary optical transmission through sub-wavelength hole arrays},
  author={Ebbesen, Thomas W and Lezec, Henri J and Ghaemi, HF and Thio, Tineke and Wolff, Peter A},
  journal={nature},
  volume={391},
  number={6668},
  pages={667--669},
  year={1998},
  publisher={Nature Publishing Group UK London}
}

@article{bonnet2022complex,
  title={The complex-scaled half-space matching method},
  author={Bonnet-Ben Dhia, Anne-Sophie and Chandler-Wilde, Simon N and Fliss, Sonia and Hazard, Christophe and Perfekt, Karl-Mikael and Tjandrawidjaja, Yohanes},
  journal={SIAM Journal on Mathematical Analysis},
  volume={54},
  number={1},
  pages={512--557},
  year={2022},
  publisher={SIAM}
}

@article{dhia2024complex,
  title={A complex-scaled boundary integral equation for time-harmonic water waves},
  author={Dhia, Anne-Sophie Bonnet-Ben and Faria, Luiz M and P{\'e}rez-Arancibia, Carlos},
  journal={SIAM Journal on Applied Mathematics},
  volume={84},
  number={4},
  pages={1532--1556},
  year={2024},
  publisher={SIAM}
}

@article{hoskins2025quadrature,
  title={On quadrature for singular integral operators with complex symmetric quadratic forms},
  author={Hoskins, Jeremy and Rachh, Manas and Wu, Bowei},
  journal={Applied and Computational Harmonic Analysis},
  volume={74},
  pages={101721},
  year={2025},
  publisher={Elsevier}
}

@article{kirsch2025pml,
  title={The {PML}-method for a scattering problem for a local perturbation of an open periodic waveguide},
  author={Kirsch, Andreas and Zhang, Ruming},
  journal={Numerische Mathematik},
  pages={1--32},
  year={2025},
  publisher={Springer}
}

@article{kehoe2023joint,
  title={Joint geometry/frequency analyticity of fields scattered by periodic layered media},
  author={Kehoe, Matthew and Nicholls, David P},
  journal={SIAM Journal on Mathematical Analysis},
  volume={55},
  number={3},
  pages={1737--1765},
  year={2023},
  publisher={SIAM}
}

@article{nicholls2025analyticity,
  title={Analyticity and Stable Computation of Dirichlet--Neumann Operators for Laplace's Equation Under Quasiperiodic Boundary Conditions in Two and Three Dimensions},
  author={Nicholls, David P and Wilkening, Jon and Zhao, Xinyu},
  journal={Studies in Applied Mathematics},
  volume={154},
  number={5},
  pages={e70059},
  year={2025},
  publisher={Wiley Online Library}
}

@article{humaikani2025rellich,
  title={A Rellich-type theorem for the Helmholtz equation in a junction of stratified media},
  author={Humaikani, Sarah Al and Dhia, Anne-Sophie Bonnet-Ben and Fliss, Sonia and Hazard, Christophe},
  journal={arXiv preprint arXiv:2504.17345},
  year={2025}
}

@article{hsu2016bound,
  title={Bound states in the continuum},
  author={Hsu, Chia Wei and Zhen, Bo and Stone, A Douglas and Joannopoulos, John D and Solja{\v{c}}i{\'c}, Marin},
  journal={Nature Reviews Materials},
  volume={1},
  number={9},
  pages={1--13},
  year={2016},
  publisher={Nature Publishing Group}
}

@article{dhia2007resonances,
    author = {Bonnet-Ben Dhia, Anne-Sophie and Mercier, Jean-François},
    title = {Resonances of an elastic plate in a compressible confined fluid},
    journal = {The Quarterly Journal of Mechanics and Applied Mathematics},
    volume = {60},
    number = {4},
    pages = {397-421},
    year = {2007},
    month = {10},
    issn = {0033-5614},
    doi = {10.1093/qjmam/hbm015},
    url = {https://doi.org/10.1093/qjmam/hbm015},
    eprint = {https://academic.oup.com/qjmam/article-pdf/60/4/397/5460395/hbm015.pdf},
}

@Inbook{pagneux2013trapped,
author="Pagneux, Vincent",
editor="Craster, Richard V.
and Kaplunov, Julius",
title="Trapped Modes and Edge Resonances in Acoustics and Elasticity",
bookTitle="Dynamic Localization Phenomena in Elasticity, Acoustics and Electromagnetism",
year="2013",
publisher="Springer Vienna",
address="Vienna",
pages="181--223",
isbn="978-3-7091-1619-7",
doi="10.1007/978-3-7091-1619-7_5"
}

@article{mciver2000water,
title={Water-wave propagation through an infinite array of cylindrical structures},
volume={424},
DOI={10.1017/S0022112000001774},
journal={Journal of Fluid Mechanics},
author={McIVER, P.},
year={2000},
pages={101–125}
}

@article{zhang2021fast,
  title={A fast direct solver for two dimensional quasi-periodic multilayered media scattering problems},
  author={Zhang, Yabin and Gillman, Adrianna},
  journal={BIT Numerical Mathematics},
  volume={61},
  number={1},
  pages={141--171},
  year={2021},
  publisher={Springer}
}

@article{gillman2013fast,
  title={A fast direct solver for quasi-periodic scattering problems},
  author={Gillman, Adrianna and Barnett, Alex},
  journal={Journal of Computational Physics},
  volume={248},
  pages={309--322},
  year={2013},
  publisher={Elsevier}
}

\appendix

\section{Quasi-periodic asymptotics}\label{app:far_as}

In order to find the asymptotics of~$w_{\xi,\gamma}$, we treat the terms in~\eqref{eq:dualsum} separately. As we noted in the proof of Lemma~\ref{lem:ana_cont}, we can pull the derivative~$\partial_{\bn(\bz)}$ inside and expand the right-hand side of \eqref{eq:quasi-IE} as
\begin{equation}
    \partial_{\bn(\bz)}G_\xi(\bz-\by) =\sum_{m=-\infty}^\infty \partial_{\bn(\bz)}\lp\frac{e^{i\xi_m(z_1-y_1)}e^{\alpha(\xi_m)(y_2-z_2)}}{-2\alpha(\xi_m)}\rp= \sum_{m=-\infty}^\infty e^{-i\xi_m y_1+\alpha(\xi_m)y_2} h_{\xi,m}(\bz),
\end{equation}
where
\begin{equation}
    h_{\xi,m}(\bz) = \partial_{\bn(\bz)}\lp\frac{e^{i\xi_mz_1-\alpha(\xi_m)z_2}}{-2\alpha(\xi_m)}\rp .
\end{equation}
The solution~$\rho_{\xi,\by}(\bz)$ of~\eqref{eq:quasi-IE} can therefore be expanded as
\begin{equation}\label{eq:rho_exp}
    \rho_{\xi,\by}(\bz) = \sum_{m=-\infty}^\infty e^{-i\xi_my_1+\alpha(\xi_m)y_2} \lp \cK_\xi^{-1}h_{\xi,m}\rp(\bz)= \sum_{m=-\infty}^\infty e^{-i\xi_my_1+\alpha(\xi_m)y_2} \rho_{\xi,m}(\bz),
\end{equation}
where
\begin{equation}
    \rho_{\xi,m} := \cK_{\xi}^{-1} h_{\xi,m}.
\end{equation}
Plugging this formula for~$\rho_{\xi,\by}$ into~$w_{\xi,\gamma}$ gives
\begin{multline}
     w_{\xi,\gamma}(\bx,\by)=\cS_\xi [\rho_{\xi,\by}](\bx) = \int_\gamma G_\xi (\bx-\bz)\rho_{\xi,\by}(\bz)\, \opd \bz \\
    = \sum_{n,m}  e^{i\xi_m x_1+\alpha(\xi_m)y_2-i\xi_n y_1+\alpha(\xi_n)x_2}\int_\gamma \frac{e^{i\xi_mz_1-\alpha(\xi_m)z_2}}{-2\alpha(\xi_m)}\rho_{\xi,n}(\bz)\,  \opd \bz .
\end{multline}
If we let, 
\begin{equation}
    f_{nm}(\xi)= \int_\gamma \rho_{\xi,n}(\bz)\tilde h_{\xi,m}(\bz)\, \opd \bz,
\quad\text{where}\quad
    \tilde h_{\xi,m}(\bz)=\frac{e^{i\xi_mz_1-\alpha(\xi_m)z_2}}{-2\alpha(\xi_m)},
\end{equation}
then we can separate
\begin{equation}
     w_{\xi,\gamma}(\bx,\by)= \sum_{n,m}  e^{i\xi_mx_1+\alpha(\xi_m)y_2-i\xi_n y_1+\alpha(\xi_n)x_2}f_{nm}(\xi).
\end{equation}
This expression is often referred to as the Rayleigh expansion of~$w_{\xi,\gamma}$. If we introduce the inverse Floquet--Bloch transform of each term, 
\begin{equation}\label{eq:wnm_def}
    w_{nm}(\bx,\by) = \int_c e^{i(\xi_n x_1- \xi_m y_1)+\alpha(\xi_m)y_2+\alpha(\xi_n)x_2}f_{nm}(\xi) \opd \xi,
\end{equation}
then we can write
\begin{equation}
    w_\gamma(\bx,\by) = \sum_{nm} w_{nm}(\bx,\by).
\end{equation}
To bound~$w_\gamma$, it is therefore enough to bound each~$w_{nm}$ and sum those bounds. The first step is to bound the~$f_{nm}$'s. 
\begin{lemma}\label{lem:fnm_bdd}
    Let~$V_{\gamma,\epsilon}$ be as in Lemma~\ref{lem:ops_analytic}.
    There exists a constant~$F>0$ such that
    \begin{equation}
        |f_{nm}(\xi)|\leq F \label{eq:fbdd}
    \end{equation}
    for all~$\xi\in V_{\gamma,\epsilon}$ and all~$n,m$.
\end{lemma}
\begin{proof}
    From the proof of Lemma~\ref{lem:G_exist} it is clear that~$w_{\xi,\gamma}(x_1,0;y_1;0)$ is a smooth function of~$x_1,y_1$. The coefficients~$f_{nm}(\xi)$ are then the Fourier series coefficients of~$e^{-i\xi(x_1-y_1)}w_{\xi,\gamma}(x_1,0;y_1,0)$. We therefore have that
    \begin{equation}
        |f_{nm}(\xi)|\leq C \|e^{-i\xi(x_1-y_1)}w_{\xi,\gamma}(x_1,0;y_1,0)\|_{L^1([-d/2,d/2]^2)}
    \end{equation}
    for all~$n,m$, where~$C$ is a constant depending on~$d$. Since~$w_{\xi,\gamma}$ is an analytic function of~$\xi\in V_{\epsilon,\gamma}$, the right hand side can be bounded independent of~$\xi$ and the result holds.
\end{proof}
\begin{figure}
    \centering
    \includegraphics[width=0.6\linewidth]{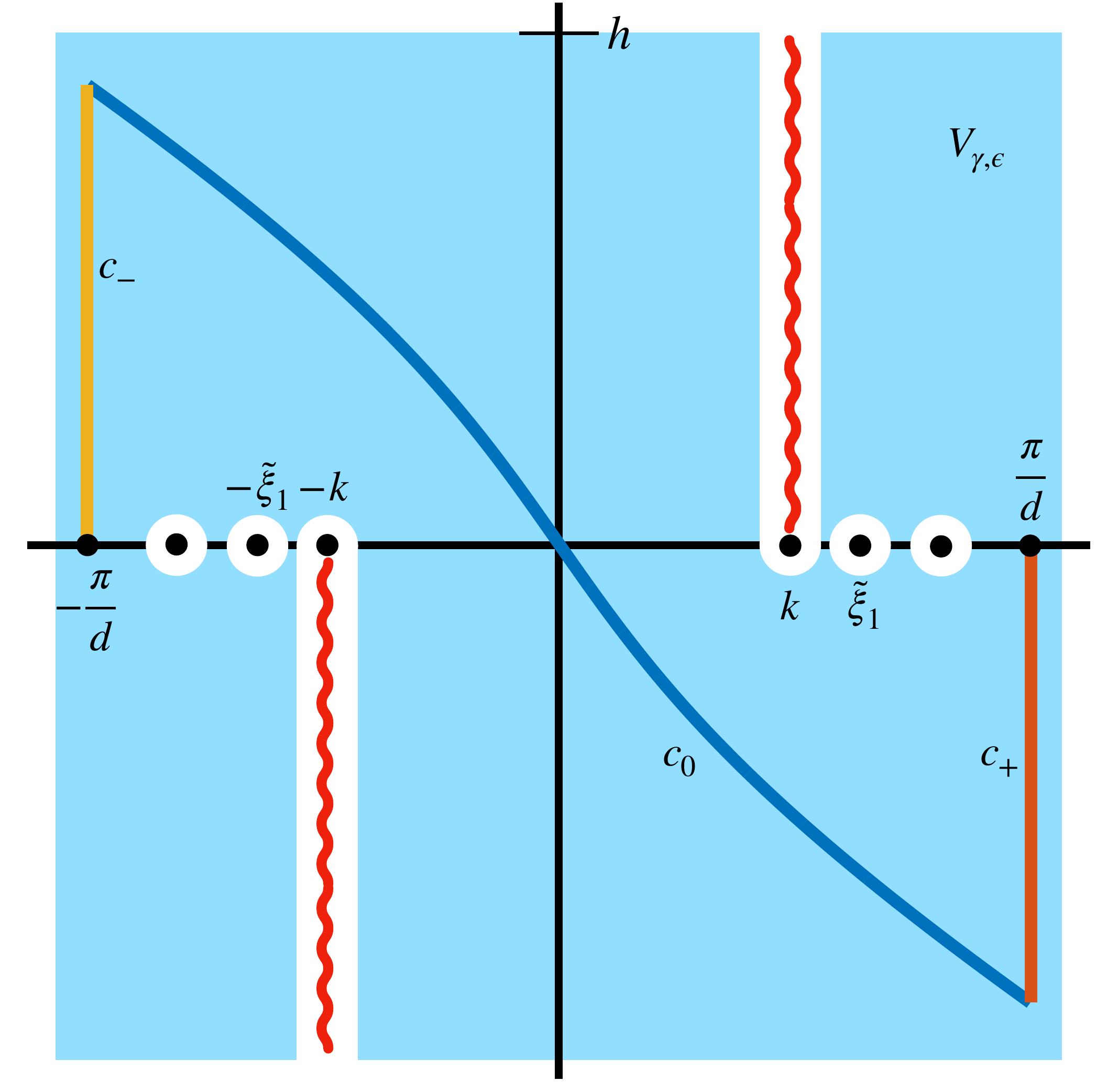}
    \caption{The contour~$c:=c_-\cup c_0\cup c_+$ used to derive our bounds on~$w_{nm}$.}
    \label{fig:Vgamma}
\end{figure}
In order to bound the~$w_{nm}$, we introduce the contours~$c_0,c_1,$ and~$c_+$ shown in \figref{fig:Vgamma}.
We choose~$c_0$ to be the contour parameterized by
\begin{equation}\label{eq:00_descent}
    \xi(t) = kt \sqrt{2} e^{-3i\pi/8} \sqrt{1+\frac{t^2}{2} e^{i\pi/4} } 
\end{equation}
truncated to live in the region~$|\Re \xi|\leq \frac{\pi}{d}$. This choice gives
\begin{equation}\label{eq:00_descent_val}
    \alpha(\xi(t)) = ik - kt^2e^{-i\pi/4},
\end{equation}
which ensures that for any~$z$ in the first quadrant
\begin{equation}
    \left|e^{\alpha(\xi(t)) z}\right| = \left|e^{\lp ik - kt^2e^{-i\pi/4}\rp z}\right| = e^{-k\Im z - k \Re\lp e^{-i\pi/4}z\rp}
\end{equation}
is exponentially decaying along~$c_0$. We choose the contours~$c_\pm$ to be the vertical lines connecting the ends of~$c_0$ to the points~$\xi=\pm \frac\pi d$. For the remainder of this appendix, we will let~$c:=c_-\cup c_0\cup c_+$. This results in contour that connects~$\pm\frac\pi d$, passes on the correct sides of the poles and branch cuts, and lives in~$V_{\gamma,\epsilon}$.

 \begin{lemma}\label{lem:alpha_bd}
    Let
    \begin{equation}
        \thetaeta := \frac12 \opnm{Arg}\left[\lp \frac{\pi}d +ih\rp^2-k^2\right] + \arctan \Kslope
    \end{equation}
    and assume~$\Kslope$ is small enough so that~$\thetaeta<\pi/2$. If~$z\in \Gamma_U$ and
    \begin{equation}
        \eta_n := \cos\thetaeta \alpha\lp \frac{(2|n|-1)\pi}d\rp,
    \end{equation}
    then
    \begin{equation}
        \Re \lp\alpha(\xi_n) z\rp  \leq \eta_n\Re z
    \end{equation}
    for all~$n\neq 0$ and~$\xi\in V_{\gamma,\epsilon}$. Further
    \begin{equation}
        \Re \lp\alpha(\xi)z\rp \leq \eta_1\Re z
    \end{equation}
    for all~$\xi\in c_-\cup c_+$.
 \end{lemma}
 \begin{proof}
We first note that since $0\leq \Im z\leq \Kslope\, \Re z$, we have
\begin{equation}
     \Im \alpha\lp\xi_n\rp  \Im z \leq \max\lp 0, -\Kslope \Im \alpha\lp\xi_n\rp \rp \Re z.
\end{equation}
Which term in the maximum is larger will depend on the sign of~$\Im \alpha(\xi_n)$.

We can therefore bound
\begin{equation}\label{eq:real_alpha_z}
    \Re \lp\alpha\lp\xi_n\rp z \rp= \Re \alpha\lp\xi_n\rp  \Re z - \Im \alpha\lp\xi_n\rp  \Im z\leq \max_{\xi\in V_{\gamma,\epsilon}} \left[\Re \alpha\lp\xi_n\rp +\max\lp 0, -\Kslope \Im \alpha\lp\xi_n\rp \rp\right]  \Re z.
\end{equation}

We can therefore prove the desired result by bounding the maximum of
\begin{equation}
    \eta_{n,1}:=\max_{\xi\in V_{\gamma,\epsilon}} \left[\Re \alpha\lp\xi_n\rp  -\Kslope \Im \alpha\lp\xi_n\rp\right]\quad\text{and}\quad  \eta_{n,2}:=\max_{\xi\in V_{\gamma,\epsilon}} \Re \alpha\lp\xi_n\rp.
\end{equation}

We begin by studying~$\eta_{n,+}$, which we bound by controlling its modulus and argument. By the choice of branch cut, we have that
\begin{align}\label{eq:re_alpha_bd}
    \frac{\Re \alpha(\xi_n)}{|\alpha(\xi_n)|}
    &=\cos\lp \pi-\frac12\opnm{Arg}(\xi_n^2-k^2) \rp = - \cos\lp \frac12\opnm{Arg}(\xi_n^2-k^2)\rp,
\end{align}
where~$\opnm{Arg}$ is the principle argument, defined to live in~$(-\pi,\pi]$.
We thus need to bound~$|\alpha(\xi_n)|$ and~$\opnm{Arg}(\xi_n^2-k^2)$. Since
\begin{equation}
    |\alpha(\xi_n)| = \sqrt{|\xi_n-k||\xi_n+k|},
\end{equation}
it is clear that that~$|\alpha(\xi_n)|$ is minimized when~$\xi$ is as close as possible to~$\pm k$, i.e.
\begin{equation}
    \min_{\xi\in V_{\gamma,\epsilon}}|\alpha(\xi_n)| = |\alpha((2
    |n|-1)\pi/d)| = -\alpha((2|n|-1)\pi/d).
\end{equation}
To understand~$\opnm{Arg}(\xi_n^2-k^2)$, we note that~$\{\xi_n\,|\,\xi\in V_{\gamma,\epsilon}\}$ is a subset of rectangle in the right half plane. Thus~$|\opnm{Arg} \xi_n|$ is maximized in the corners where~$\xi_n=\frac{(2n- \opnm{sign} n )\pi}d \pm ih$. This property is preserved by squaring and subtracting~$k^2$ (see \figref{fig:xi_m_angle}), and so
\begin{equation}\label{eq:xim2_arg}
    \max_{\xi\in V_{\gamma,\epsilon}}\left|\opnm{Arg}(\xi_n^2-k^2)\right| = \left|\opnm{Arg}\left[\lp \frac{(2
    |n|-1)\pi}d \pm ih\rp^2-k^2\right] \right|\\\leq \opnm{Arg}\left[\lp \frac{\pi}d +ih\rp^2-k^2\right].
\end{equation}
\begin{figure}
    \centering
    \includegraphics[width=0.5\linewidth]{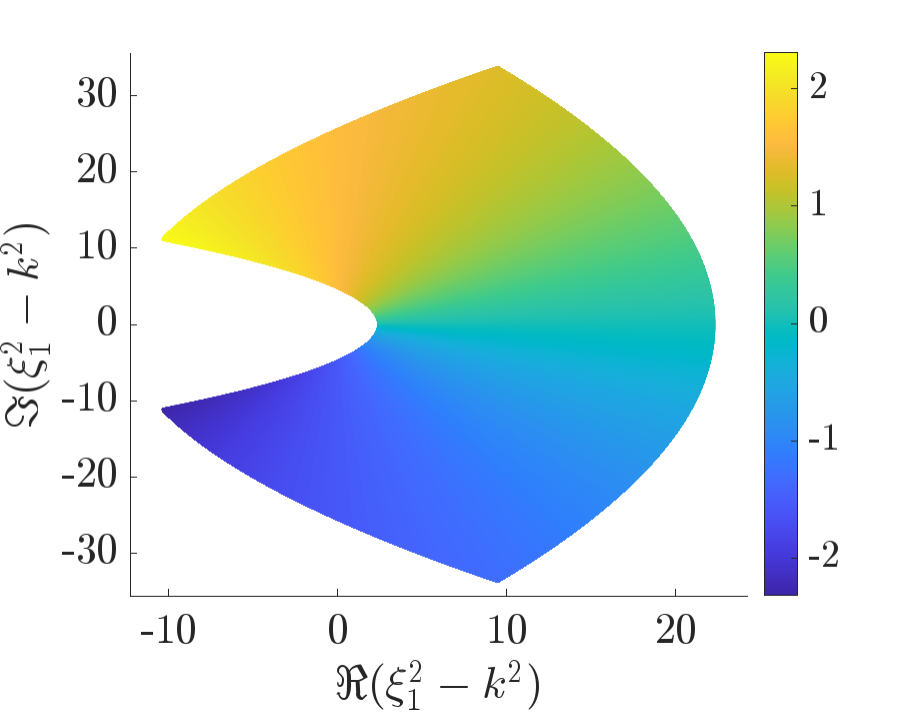}
    \caption{The argument of~$\xi_1^2-k^2$ for~$\xi$ in the bounding rectangle of~$V_{\gamma,\epsilon}$. The parameters~$k$ and~$d$ are set to 1 and 2 respectively.}
    \label{fig:xi_m_angle}
\end{figure}
By definition, this gives that
\begin{equation}
    \max_{\xi\in V_{\gamma,\epsilon}}\left|\frac12\opnm{Arg}(\xi_n^2-k^2)\right|\leq \thetaeta
\end{equation}
and so~\eqref{eq:re_alpha_bd} implies
\begin{equation}
   \eta_{n,1} =  \min_{\xi\in V_{\gamma,\epsilon}}\Re \alpha(\xi_n)\leq  \alpha((2|n|-1)\pi/d) \cos\thetaeta.
\end{equation}

To bound~$\eta_{n,2}$, we note
\begin{align}
    \frac{\Re \alpha(\xi_n) - \Kslope \Im \alpha(\xi_n)}{|\alpha(\xi_n)|\sqrt{1+\Kslope^2}}&=\frac{\Re \left[\alpha(\xi_n)(1+\Kslope i)\right]}{|\alpha(\xi_n)|\sqrt{1+\Kslope^2}}\\
    &=\cos\lp \pi-\frac12\opnm{Arg}(\xi_n^2-k^2) + \arctan \Kslope\rp \\
    &= - \cos\lp \frac12\opnm{Arg}(\xi_n^2-k^2) - \arctan \Kslope\rp.
\end{align}
Since we have already bounded~$|\alpha(\xi_n)|$, we now just have to bound the argument. Equation~\eqref{eq:xim2_arg} also implies that
\begin{equation}
    \opnm{Arg}\left[\lp \frac{\pi}d -ih\rp^2-k^2\right]\leq\frac12\opnm{Arg}(\xi_n^2-k^2) \leq  \opnm{Arg}\left[\lp \frac{\pi}d +ih\rp^2-k^2\right],
\end{equation}
for all~$\xi\in V_{\gamma,\epsilon}$. Thus
\begin{multline}
     \left|\frac12\opnm{Arg}(\xi_n^2-k^2) - \arctan \Kslope \right| \\
     \leq \min \left[\left|\frac12\opnm{Arg}\lp\lp \frac{\pi}d +ih\rp^2-k^2\rp - \arctan \Kslope \right|, \left|\frac12\opnm{Arg}\lp\lp \frac{\pi}d -ih\rp^2-k^2\rp + \arctan \Kslope \right|\right]\\
     =\left|\frac12\opnm{Arg}\lp\lp \frac{\pi}d -ih\rp^2-k^2\rp + \arctan \Kslope \right|\\
     = \frac12\opnm{Arg}\lp\lp \frac{\pi}d +ih\rp^2-k^2\rp - \arctan \Kslope=\thetaeta.
\end{multline}
We therefore have that
\begin{equation}
    \eta_{n,2} \leq  \min_{\xi\in V_{\gamma,\epsilon}}\left[\Re \alpha(\xi_n)-\Kslope \Im\alpha(\xi_n)\right]\leq  \alpha((2|n|-1)\pi/d) \cos(\thetaeta).
\end{equation}
Plugging both these estimates into \eqref{eq:real_alpha_z} gives
\begin{equation}
    \Re \lp \alpha\lp\xi_n\rp z\rp \leq \min(\eta_{n,1},\eta_{n,2}) \Re z \leq \alpha((2|n|-1)\pi/d) \cos(\thetaeta)
\end{equation}
for all~$\xi\in V_{\gamma,\epsilon}$, which is the desired result.

To bound~$\alpha(\xi)$ on~$c_{+}$, we note that if~$\xi\in c_+$ then~$\xi$ coincides with~$\tilde\xi_1$ for some~$\tilde\xi\in V_{\gamma,\epsilon}$. The previous argument thus gives that
\begin{equation}
    \max_{\xi\in c_+}\Re \lp \alpha(\xi) z\rp \leq \max_{\tilde\xi\in V_{\gamma,\epsilon}} \Re \lp \alpha(\tilde\xi_1) z\rp\leq \eta_1 \Re z.
\end{equation}
The symmetry of~$\alpha$ implies the same result for~$c_-$.
 \end{proof}

We now work to bound each of the~$w_{nm}$'s. Since~$\alpha(\xi_n)$ has a stationary point in~$V_{\gamma,\epsilon}$ for~$n=0$ and does not for~$n\neq 0$, we shall treat the cases where~$n$ or~$m$ is zero separately.

\subsection{Decay in source and target}

We begin with the easiest case, where both~$n$ and~$m$ are non-zero.
\begin{lemma}\label{lem:nonzero_nm}
    There is a constant~$C_>0$
    \begin{equation}
    \left|\sum_{n,m\neq 0}  w_{nm}(\bx,\by)\right| \leq C_ e^{\eta_1 \Re (x_2+y_2)+h|x_1-y_1|}
    \end{equation}
    whenever~$x_2,y_2\in \Gamma_U$ and~$x_1,y_1\in \bbR$. Further, an identical result holds for any~$x_1$ derivative of~$w_{Re}$ with a different constant~$C$.
\end{lemma}
\begin{proof} 
By Lemma~\ref{lem:alpha_bd}, we have
\begin{equation}
    \left|e^{\alpha(\xi_m)y_2} \right|\leq e^{\eta_m \Re y_2} \quad\text{and}\quad \left|  e^{\alpha(\xi_n)x_2} \right|\leq e^{\eta_n\Re y_2}
\end{equation}
for all~$\xi\in V_{\gamma,\epsilon}$.
We can combine these and integrate over~$c$ to see
\begin{equation}
    |w_{nm}(\bx,\by)| = \left|\int_c e^{i\xi_mx_1-i\xi_my_1}e^{\alpha(\xi_n)x_2+\alpha(\xi_m) y_2}f_{nm} \,\opd \xi \right| \leq |c|F e^{h|x_1-y_1|}  e^{\eta_n \Re x_2+\eta_m \Re x_2} .\label{eq:nm_nonzero_int}
\end{equation}
By symmetry, the same estimate holds for negative~$m$. Summing over~$n$ and~$m$ then gives
\begin{equation}
    \left|\sum_{n,m\neq 0}w_{nm}(\bx,\by)\right| \leq |c|F e^{h|x_1-y_1|} \lp\sum_{n\neq 0} e^{\eta_n\Re x_2}\rp \lp \sum_{m\neq 0} e^{\eta_m\Re y_2} \rp \label{eq:nm_non_sum}
\end{equation}
To bound the remaining sums, we must control the growth of~$\eta_n$. To do this, we note that
\begin{equation}
    \partial_\xi^2 \alpha\lp\frac{(2|n|-1)\pi}d\rp = -\frac{k^2}{\alpha\lp \frac{(2|n|-1)\pi}d\rp^3} >0,
\end{equation}
which implies that~$\partial_\xi \alpha(\xi) \geq \partial_\xi \alpha(\pi/d)$. The mean value theorem thus tells us that
\begin{multline}
    \eta_n = \cos(\thetaeta)\alpha\lp\frac{(2|n|-1)\pi}d\rp \leq  \cos(\thetaeta)\left[ \alpha\lp\frac\pi d\rp+ \frac{\pi}d (2|n|-2) \partial_\xi\alpha\lp\frac\pi d\rp \right]\\
    = \eta_1 + \partial_\xi\alpha\lp\frac{\pi} d\rp\frac{2\pi}d (|n|-1) \cos(\thetaeta) .\label{eq:etan_def}
\end{multline}
Plugging this into~\eqref{eq:nm_non_sum} gives
\begin{equation}
\begin{split}
    \left|\sum_{n,m\neq 0}w_{nm}(\bx,\by)\right| &\leq 4|c|F e^{\eta_1 \Re(x_2+y_2)+h|x_1-y_1|} \lp\sum_{n= 0}^\infty e^{\partial_\xi\alpha\lp\frac\pi d\rp\frac{2\pi}d\cos(\thetaeta) n\Re x_2}\rp \lp \sum_{n= 0}^\infty e^{\partial_\xi\alpha\lp\frac\pi d\rp\frac{2\pi}d\cos(\thetaeta) m\Re y_2} \rp \\
    &=4|c|F e^{\eta_1 \Re(x_2+y_2)+h|x_1-y_1|} \frac{1}{1-e^{\partial_\xi\alpha\lp\frac\pi d\rp\frac{2\pi}d \cos(\thetaeta) \Re x_2}}\frac{1}{1-e^{\partial_\xi\alpha\lp\frac\pi d\rp\frac{2\pi}d\cos(\thetaeta) \Re y_2}}
    \end{split}\label{eq:wnm_non_sum}
\end{equation}
Since~$\Re x_2,\Re y_2\geq d/2$, we have the desired result. Taking~$x_1$ or~$y_1$ derivatives just pulls down powers of~$\xi$ in~\eqref{eq:nm_nonzero_int}, whose modulus can be bounded by~$\sqrt{\pi^2/d^2+h^2}$ and so the derivatives of~$w_{Re}$ can be similarly bounded.
\end{proof}

\subsection{Oscillatory in the source and target}
We now consider the case that both~$n$ and~$m$ are zero.
\begin{lemma}\label{lem:00_bound}
For~$l\geq 0$, there are constants~$A_{l}, C_{l}$, and~$C_{l}$ such that
    \begin{multline}
    \left|\partial_{x_1}^lw_{00}(\bx,\by)- \frac{A_l e^{ik(x_2+y_2)}}{(x_2+y_2)^{\opnm{ceil}(l/2)+1/2}}\right| \leq \frac{C_{l}e^{-k\Im(x_2+y_2)}}{|x_2+y_2|^{\opnm{ceil}(l/2)+1}}e^{h|x_1-y_1|} + C_l e^{\eta_1\Re(x_2+y_2)}e^{h|x_1-y_1|}
    \end{multline}
    for all~$x_2,y_2\in\Gamma_U$ and all~$x_1,y_1\in \bbR$.
\end{lemma}
\begin{proof}
    We split the integral
\begin{equation}\label{eq:w00_int}
    \partial_{x_1}^lw_{00}(\bx,\by)=\int_c (i\xi)^le^{i\xi(x_1-y_1)+\alpha(\xi)(x_2+y_2)}f_{00}(\xi) \, \opd \xi=I_0 + I_- + I_+,
\end{equation}
where~$I_0,I_\pm$ are integrals over contours~$c_0,c_\pm$. 

In~\cite{epstein2025complex}, the authors used Laplace's method to prove the integral over~$c_0$ will be
\begin{equation}
    I_{0} =  \frac{e^{ik (x_2+y_2)}}{(x_2+y_2)^{\opnm{ceil}(l/2)+1/2}} A_l + e^{ik(x_2+y_2)}O\lp |x_2+y_2|^{-(l+3)/2}e^{h|x_1-y_1|}\rp.
\end{equation}

To bound~$I_+$, we note
\begin{equation}
    I_{+} = \int_{c_+} e^{\alpha\lp \xi\rp (x_2+y_2)+\xi (x_1-y_1)}f_{00}(\xi) \,\opd \xi.
\end{equation}
By definition, $|\Im\xi|$ will be less than~$h$. We can thus bound~$I_{+}$ by
\begin{equation}
    |I_{+}| \leq 2|c_+| Fe^{h|x_1-y_1|}\max_{\xi\in c_+}  e^{\Re\left[\alpha\lp \xi\rp (x_2+y_2)\right]} \leq 2|c_+| Fe^{h|x_1-y_1|+\eta_1 \Re (x_2+y_2)},
\end{equation}
where we used Lemma~\ref{lem:alpha_bd} in the second inequality.
The integral $I_{-}$ can be bounded similarly. Adding the bounds on~$I_0$,~$I_-$ and~$+I_+$ gives the desired result.
\end{proof}

\subsection{Oscillatory in the target and decay in the source }

We now consider the case that~$n=0$ and~$m\neq 0$. We start with a lemma. 

\begin{lemma}\label{lem:real_decay_bd}
Let~$U$ be any closed subset in the interior of~$V_{\gamma,\epsilon}$.
    For each~$l\geq 0$, there is a~$C_l$ such that
    \begin{equation}
        \left|\partial_\xi^l e^{\alpha(\xi_n) y_2}\right| \leq C_le^{(\eta_n+\epsilon) y_2}
    \end{equation}
    for all~$n\neq 0$,~$\xi\in U$, and~$y_2\in \Gamma_U$.
\end{lemma}
\begin{proof}
We begin by bounding the derivatives of~$\alpha(\xi_n)$. We recall
    \begin{equation}
        \partial_\xi\alpha(\xi_n) = \frac{\xi_n}{\alpha(\xi_n)},
    \end{equation}
    which is analytic for~$\xi \in V_{\gamma,\epsilon}$. We also have that
    \begin{equation}
        \max_{\xi\in V_{\gamma,\epsilon}}|\partial_\xi\alpha(\xi_n)| \leq \frac{(2|n|+1)\pi/d}{|\alpha((2|n|-1)\pi/d)|}.
    \end{equation}
    This is a continuous function of~$n$ and
    \begin{equation}
        \lim_{n\to \infty}\frac{(2|n|+1)\pi/d}{|\alpha((2|n|-1)\pi/d)|} = 1,
    \end{equation}
    so there is a~$C$ such that
    \begin{equation}
        \max_{\xi\in V_{\gamma,\epsilon},n\neq 0}|\partial_\xi\alpha(\xi_n)| \leq C.
    \end{equation}
    Cauchy's integral formula thus implies that
    \begin{equation}\label{eq:bd_all_der_alph}
        \max_{\xi\in U,n\neq 0}|\partial_\xi^l\alpha(\xi_n)| \leq \frac{l! |\partial V_{\gamma,\epsilon}|}{2\pi(\opnm{dist}( \partial V_{\gamma,\epsilon},U))^{l+1}} C
    \end{equation}
    for all~$n$.

    The product rule implies that
    \begin{equation}
        \partial_\xi^l e^{\alpha(\xi_n) y_2} = e^{\alpha(\xi_n) y_2} q_l(y_2;\xi_n),
    \end{equation}
    where~$q_l$ is a degree~$l$ polynomial in~$y_2$ whose coefficients are all derivatives of~$\alpha(\xi_n)$. Lemma~\ref{lem:alpha_bd} and the estimate~\eqref{eq:bd_all_der_alph} thus give that
    \begin{equation}
         \left|\partial_\xi^l e^{\alpha(\xi_n) y_2}\right| \leq C_le^{\eta_n \Re y_2}(1+|y_2|)^l.
    \end{equation}
    Since for any epsilon there is a constant such that
    \begin{equation}
        (1+|y_2|)^l \leq D_l e^{\frac{\epsilon}{\sqrt{1+\Kslope^2}} |y_2|} \leq D_l e^{\epsilon\Re y_2},
    \end{equation}
    we have proved the result.
\end{proof}

\begin{proposition}\label{prop:0n_close}
Let
\begin{equation}
    w_{0n,\opnm{close}}(\bx,\by) := \int_{c_0} e^{i(\xi x_1 -\xi_n y_1) + \alpha(\xi) x_2+\alpha(\xi_n) y_2} f_{0n}(\xi) \,\opd \xi.
\end{equation}
    If~$l\geq 0$, then there exist a function~$a_{n,l}(x_1,y_1,y_2)$ that is analytic in~$y_2$ such that
    \begin{equation} \label{eq:w0_nclose}
        \left| \partial_{x_2}^lw_{0n,\opnm{close}}(\bx,\by) - \frac{a_{n,l}(x_1,y_1,y_2) e^{ikx_2}}{x_2^{\opnm{ceil}(l/2)+1/2}}\right| \leq C_{l} \frac{e^{-k \Im x_2+(\eta_n+\epsilon)\Re y_2}e^{h|x_1-y_1|}}{|x_2|^{(l+3)/2}} 
    \end{equation}
    for all~$x_2,y_2\in \Gamma_U$, where~$C_l$ is independent of~$n$, and~$\eta_n$ is given by~\eqref{eq:etan_def}. Further, the functions $a_{n,l}$ satisfy
    \begin{equation}
        |a_{n,l}(x_1,y_1,y_2)| \leq D_l e^{\lp\cos \thetaeta \alpha\lp \frac{2|n|\pi}{d}\rp + \epsilon\rp \Re y_2}e^{h|x_1-y_1|}
    \end{equation}
    for some $D_l$ independent of~$n$.
\end{proposition}
\begin{proof}
Let~$\xi(t)$ be the parameterization from~\eqref{eq:00_descent} and~$\xi(t_0)$ be the end of~$c_0$ with positive real parameter. Using this parameterization, we have
\begin{multline}
    \partial_{x_2}^lw_{0n,\opnm{close}}(\bx,\by) \\
    = e^{ikx_2} i^l\int_{-t_0}^{t_0} e^{-ke^{-i\pi/4} t^2 x_2}e^{\alpha(\xi(t)+\beta_n) y_2}e^{i\xi(t)(x_1-y_1)-i\beta_n y_1} f_{0n}(\xi(t))\xi^l(t)\partial_t \xi(t) \,\opd t,\label{eq:w_0n_expand}
\end{multline}
where $\beta_n = 2\pi n/d$.
By definition, we can write~$\xi(t) = t v(t)$ for a smooth~$v(t)$.
We let
\begin{equation}
    \tilde g(t;x_1,y_1)= e^{i\xi(t)(x_1-y_1)}\partial_t \xi(t) v^l(t)
\end{equation}
and
\begin{equation}
    \tilde f_n(t;x_1,y_1,y_2)= e^{\alpha(\xi(t)+\beta_n) y_2}\tilde g(t;x_1,y_1) f_{0n}(\xi(t))e^{-i\beta_n y_1}.
\end{equation}
With these definitions, the integral in \eqref{eq:w_0n_expand} becomes
\begin{equation}
        I =\int_{-t_0}^{t_0}t^l e^{-ke^{-i\pi/4} t^2 x_2}\tilde f_n(t;x_1,y_1,y_2) \,\opd t.
    \end{equation} 
To find our asymptotic estimate, we shall Taylor expand~$\tilde f_n$ about~$t=0$. To bound the terms of the Taylor series, we bound each piece of~$\tilde f_n$, beginning with~$\tilde g$. Since~$v(t)$ is smooth, it is easy to see that 
\begin{equation}
    \max_{|t|\leq t_0}|\partial_t^j\tilde g(t;x_1,y_1)| \leq C_j e^{\max_{|t|\leq t_0}|\Im \xi(t)||x_1-y_1|}(1+|x_1-y_1|)^j \leq  D_j e^{h|x_1-y_1|}
\end{equation}
for all~$j\geq0$. We bound~$\partial_\xi^j f_{0n}(\xi(t))$ using Lemma~\ref{lem:fnm_bdd}. Applying Cauchy's integral formula on~$\partial V_{\gamma,\epsilon}$ gives
\begin{equation}
    \max_{|t|\leq \epsilon} |\partial_\xi^j f_{0n}(\xi(t))| \leq \frac{j! |\partial V_{\gamma,\epsilon}|}{2\pi \opnm{dist}(\partial V_{\gamma,\epsilon},c_0)^{j+1}} F,
\end{equation}
which implies that~$\max_{|t|\leq \epsilon} |\partial_t^j f_{0n}(\xi(t))| \leq \tilde C_j$.
Finally, we can bound the~$y_2$-dependent term:
\begin{equation}
    \left|\partial_t^j e^{\alpha(\xi(t)+\beta_n) y_2}\right| \leq \tilde D_j e^{(\eta_n+\epsilon)\Re y_2}
\end{equation}
using Lemma~\ref{lem:real_decay_bd} and the smoothness of~$\xi(t)$. At~$t=0$ we have the tighter bounder
\begin{equation}
    \left.\left|\partial_t^j e^{\alpha(\xi(t)+\beta_n) y_2}\right|\right|_{t=0} \leq \tilde D_j e^{(\cos(\thetaeta)\alpha(\beta_n)+\epsilon)\Re y_2}.
\end{equation}
Combining these terms allows us to write the Taylor series of~$\tilde f_n$:
\begin{equation}
    \tilde f_n(t;x_1,y_1,y_2) = t^l g_{n,l,0}(x_1,y_1,y_2)+t^{l+1} g_{n,l,1}(x_1,y_1,y_2)+ t^{l+2} G_{n,l}(t;x_1,y_1,y_2).\label{eq:fn_Taylor}
\end{equation}
where
\begin{equation}
    |g_{n,l,j}(x_1,y_1,y_2)|\leq C_{l,j} e^{(\cos(\thetaeta)\alpha(\beta_n)+\epsilon)\Re y_2}
\end{equation}
and
\begin{equation}
     \max_{|t|\leq t_0}\left|G_{n,l}(t;x_1,y_1,y_2)\right|\leq D_{l} e^{(\eta_n+\epsilon)\Re y_2}e^{h|x_1-y_1|}.
\end{equation}
We also have that~$g_{n,l,j}$ is analytic in~$y_2$ because~$\tilde f_n$ is.
We can thus write
\begin{equation}
     \partial_{x_2}^lw_{0n,\opnm{close}}(\bx,\by) = e^{ikx_2}i^l\lp I_{l,0}+I_{l,1}+ J_{l}  \rp, \label{eq:partial_w0n_expand}
\end{equation}
where
\begin{equation}
    I_{l,j} = \int_{-t_0}^{t_0}e^{-ke^{-i\pi/4} t^2 x_2} t^{l+j} g_{n,l,j}(x_1,y_1,y_2)\, \opd t
\end{equation}
and
\begin{equation}
    J_{l} = \int_{-t_0}^{t_0}e^{-ke^{-i\pi/4} t^2 x_2} t^{l+1}G_{n,l}(t;x_1,y_1,y_2)\, \opd t.
\end{equation}
It is clear that~$I_{l,j} =0$ if~$l+j$ is odd. To bound the other terms, we compute:
\begin{multline}
    I_{l,j}=  g_{n,l}(x_1,y_1,y_2)\int_{-t_0}^{t_0}e^{-kt^2e^{-i\pi/4} x_2} t^{l+j}\, \opd t =  g_{n,l,j}(x_1,y_1,y_2)\lp\int_{-\infty}^\infty e^{-kt^2e^{-i\pi/4} x_2} t^{j+l}\, \opd t +\mathcal{E}_{n,l,j}\rp\\
    = g_{n,l,j}(x_1,y_1,y_2)\lp \tilde C_{j+l} \frac{1}{(e^{-i\pi/4} x_2)^{(l +j+1)/2}}+\mathcal{E}_{n,l,j}\rp,
\end{multline}
where the fractional power is defined to be the principal root, which exists because~$e^{-i\pi/4} x_2$ is in the right half plane, and $|\mathcal{E}_{n,l,j}(x_2)|\leq Ce^{-k \Re[e^{-i\pi/4} x_2]t_0^2} $. Since~$x_2$ is in the first quadrant, we can rewrite this as
\begin{equation}
 I_{l,j}=  \begin{cases}
     \frac{\tilde C_{j+l} g_{n,l,j}(x_1,y_1,y_2)}{e^{-i\pi (l +j+1)/8}x_2^{(l +j+1)/2}} + g_{n,l,j}(x_1,y_1,y_2)\mathcal{E}_{n,l}(x_2) & l+j \text{ even}\\
     0 & l+j \text{ odd}
 \end{cases}.\label{eq:Ilj_form}
\end{equation}
We can bound the remainder term:
\begin{multline}
    |J_{l}| \leq  D_{l} e^{(\eta_n+\epsilon)\Re y_2}e^{h|x_1-y_1|} \int_{-\infty}^{\infty}e^{-kt^2 \Re[e^{-i\pi/4} x_2]} |t|^{l+2}\, \opd t \\
    =\tilde D_{l} e^{(\eta_n+\epsilon)\Re y_2}e^{h|x_1-y_1|} \frac{1}{|x_2|^{(l+3)/2}},
\end{multline}
where we have used the fact that~$x_2\in \Gamma_U$ to bound~$\Re[e^{-i\pi/4} x_2]$ by~$|x_2|$.

If $l$ is even, then adding up the pieces of~\eqref{eq:partial_w0n_expand} gives
\begin{equation}
    \left|\partial_{x_2}^lw_{0n,\opnm{close}}(\bx,\by) - e^{ikx_2}i^l I_{l,0} \right|\leq  \tilde D_{l} e^{-k\Im x_2+(\eta_n+\epsilon)\Re y_2}e^{h|x_1-y_1|} \frac{1}{|x_2|^{(l+3)/2}},
\end{equation}
which gives the desired result with $a_{n,l} = \frac{\tilde C_{l}}{e^{-i\pi(l+1)/8}}g_{n,l,0}$ by \eqref{eq:Ilj_form}. If~$l$ is odd, we find
\begin{equation}
    \left|\partial_{x_2}^lw_{0n,\opnm{close}}(\bx,\by) - e^{ikx_2}i^l I_{l,1} \right|\leq  \tilde D_{l} e^{-k\Im x_2+(\eta_n+\epsilon)\Re y_2}e^{h|x_1-y_1|} \frac{1}{|x_2|^{(l+3)/2}},
\end{equation}
which gives the desired result with $a_{n,l} = \frac{\tilde C_{l+1}}{e^{-i\pi(l+2)/8}} g_{n,l,1}$.
\end{proof}
\begin{remark}
We could add more terms to the asymptotic expansion in \eqref{eq:w0_nclose} by taking more terms in the Taylor series~\eqref{eq:fn_Taylor}. It is also important to note that the asymptotic term in \eqref{eq:w0_nclose} is only an asymptotic form as $x_2\to \infty$ for fixed~$x_1,y_1,y_2$. As~$y_2\to \infty$ the asymptotic form decays faster than the remainder does and so it isn't a true asymptotic.
\end{remark}

To finish off our estimate of~$w_{0n}$ we have to integrate over the vertical strips connecting~$\xi = \pm\frac\pi d$ to the descent contours.
\begin{lemma} \label{lem:0n_connect}
    Let
    \begin{equation}
        w_{0n,\pm}(\bx,\by):=\int_{c_\pm} e^{\alpha(\xi)x_2+\alpha(\xi_n )y_2+i\xi(x_1-y_1)-i\beta_n y_1} f_{0n}\lp\xi\rp \, \opd \xi,
    \end{equation}
    where~$c_{\pm}$ is a vertical strip contained in~$\{ \xi \in \bbC,|\, \Re \xi = \pm \frac{\pi}{d},\, \pm\Im \xi \geq 0\}$ of length less than or equal to~$h$. If~$x_2,y_2\in \Gamma_U$, then
        \begin{equation}
        \left|\partial_{x_1}^lw_{0n,\pm}(\bx,\by)\right|\leq C_l Fe^{\eta_1\Re x_2 +\eta_n\Re y_2}e^{h|x_1-y_1|}
    \end{equation}
    for all~$x_2,y_2\in \Gamma_U$ and~$l\geq 0$.
    Similar results hold for the~$x_1$ and~$y_1$ derivatives of~$w_{0n,\pm}$.
\end{lemma}
\begin{proof}
We parameterize $c_{\pm}$ in the same way as Lemma~\ref{lem:00_bound}.
In the proofs of Lemma~\ref{lem:nonzero_nm} and~\ref{lem:00_bound}, we observed that
\begin{equation}
     \left|e^{\alpha((2n-1)\frac\pi d+ it)y_2}\right| \leq e^{\eta_n \Re y_2} \quad\text{and}\quad\left|e^{\alpha(\frac\pi d+ it)x_2}\right|\leq e^{\eta_1 \Re x_2}.
\end{equation}
These bounds can be combined using similar arguments to those proofs to give the desired result.
\end{proof}

\begin{proposition}
\label{lem:0n_estimate}
For all~$l\geq 0$ there is a function~$a_{l}(x_1,y_1,y_2)$ that is analytic in~$y_2$ and constant~$C_l$ such that 
\begin{multline}\label{eq:w0far}
    \left| \partial_{x_1}^l\sum_{n\neq 0} w_{0n}(\bx,\by)  - \frac{a_{l}(x_1,y_1,y_2)e^{ikx_2}}{x_2^{\opnm{ceil}(l/2)+1/2}}\right| \leq C_{l} \frac{e^{-k \Im x_2+(\eta_1+\epsilon)\Re y_2}e^{h|x_1-y_1|}}{|x_2|^{(l+3)/2}}\\
    +C_le^{\eta_1\Re (x_2+ y_2)}e^{h|x_1-y_1|}
\end{multline} 
 if~$x_2,y_2\in \Gamma_U$.
Further, the functions~$a_{l}$ satisfy
\begin{equation}
    |a_{l}(x_1,y_1,y_2)| \leq C_{l} e^{h|x_1-y_1|}   e^{(\cos\theta \alpha\lp\frac{2\pi}{d}\rp +\epsilon)\Re y_2}.\label{eq:a_ljbound}
\end{equation}
\end{proposition}
\begin{proof}
By the choice of contour, we have
\begin{equation}
    w_{0n}(\bx,\by) = w_{0n,\opnm{close}}(\bx,\by) + w_{0n,+}(\bx,\by)+w_{0n,-}(\bx,\by). \label{eq:w0n_decomp}
\end{equation}
As before, we study these pieces separately. Using the estimate in Lemma \ref{lem:0n_connect} we can repeat the proof of Lemma~\ref{lem:nonzero_nm} to bound the sum over~$n$:
\begin{equation}
    \left|\sum_{n\neq 0} \partial_{x_1}^l  w_{0n,+}(\bx,\by)+\partial_{x_1}^l w_{0n,-}(\bx,\by)\right| \leq C_le^{\eta_1 \Re (x_2+y_2)}.
\end{equation}

For the remaining piece, we let
\begin{equation}
    a_{l}(x_1,y_1,x_2) = \sum_{n\neq 0} a_{n,l}(x_1,y_1,x_2).
\end{equation}
The bound on~$a_{n,l}$ in Proposition~\ref{prop:0n_close} gives that this sum converges uniformly for~$y_2\in \Gamma_U$ and satisfies \eqref{eq:a_ljbound}. It is also analytic because each~$a_{n,l}$ is. To bound the remainder, we note that
\begin{multline}
    \left|\left[\sum_{n\neq 0}  \partial_{x_1}^l  w_{0n,\opnm{close}}(\bx,\by)\right] - \frac{a_l(x_1,y_1,y_2) e^{ikx_2}}{x_2^{\opnm{ceil}(l/2)+1/2}}\right| \\
    \leq \sum_{n\neq 0} \left|\partial_{x_1}^lw_{0n,\opnm{close}}(\bx,\by)-\frac{a_{n,l}(x_1,y_1,y_2)e^{ikx_2}}{x_2^{\opnm{ceil}(l/2)+1/2}}\right|
    \leq  \sum_{n\neq 0} C_{l} \frac{e^{-k \Im x_2+(\eta_n+\epsilon)\Re y_2}e^{h|x_1-y_1|}}{|x_2|^{(l+3)/2}}.
\end{multline}
Summing over~$n$ gives
\begin{equation}
        \left|\left[\sum_{n\neq 0}  \partial_{x_1}^l  w_{0n,\opnm{close}}(\bx,\by)\right] - \frac{a_l(x_1,y_1,y_2) e^{ikx_2}}{x_2^{\opnm{ceil}(l/2)+1/2}}\right|\leq D_{l} \frac{e^{-k \Im x_2+(\eta_1+\epsilon)\Re y_2}e^{h|x_1-y_1|}}{|x_2|^{(l+3)/2}}
\end{equation}
As all of the sums converge uniformly, we are free to swap the sums and derivatives. The summing the right hand side of \eqref{eq:w0n_decomp} then gives the result.
\end{proof}
The symmetry of~$w_{nm}$ in~$n$ and~$m$ implies that this bound also holds for the sum over~$n$.
\begin{lemma}\label{lem:n0_estimate}
    The results of Proposition~\ref{lem:0n_estimate} hold for~$\sum_{n\neq0} w_{0n}$ with~$x_2$ and~$y_2$ swapped.
\end{lemma}

\section{Far sources and near targets}\label{app:far2near_as}
In this appendix, we study the behavior of~$w_\gamma(\bx,\by)$ for~$\bx$ close to~$\gamma$ and~$\by\in \bbR\times \Gamma_U$. By equation~\eqref{eq:rho_exp}, we can write
\begin{equation}
    w(\bx,\by) =\sum_n \int_c S_{\gamma,\xi}\left[  e^{\alpha(\xi_n)y_2-i\xi_n y_1} \rho_{n,\xi}\right]\,\opd \xi = \sum_n w_n(\bx,\by),
\end{equation}
where
\begin{equation}\label{eq:f2n_exp}
    w_n(\bx,\by) := \int_c e^{\alpha(\xi_n)y_2-i\xi_n y_1}S_{\gamma,\xi}\left[ \rho_{n,\xi}\right] \,\opd \xi.
\end{equation}

As before, we we bound the case~$n=0$ and~$n\neq 0$ separately, beginning with the latter.
\begin{lemma}\label{lem:near2far_non}
    Let~$\Omega_H = \{\bx\in\Omega \,|\, x_2<H\}$ for some~$H>0$. For any~$\delta>0$, let~$\Omega_{H,\delta}$ be the set of points in~$\Omega_H$ that are at least a distance~$\delta$ from any corners of~$\gamma$.
    For each~$l\geq 0$ there is a constant~$C_{l,\delta}$ and such that
    \begin{equation}
        \left| \partial_{y_1}^l\sum_{n\neq 0}w_n(\bx,\by)\right| \leq C_{l,H,\delta}e^{\eta_1 \Re y_2+h|y_1|} \label{eq:wnearfar_n}
    \end{equation}
    for~$\bx\in \Omega_{H,\delta}$ and~$\by\in \Omega_\bbC\setminus \Omega_{d/2}$.
\end{lemma}
\begin{proof}
By~\eqref{eq:bessel_sum}, it is clear that~$S_{\gamma,\xi}$ and~$S'_{\gamma,\xi}$ have the same singularity as the respective free-space Helmholtz layer potentials. In particular, the kernel of $S'_{\gamma,\xi}$ will be smooth wherever~$\gamma$ is smooth. A simple bootstrapping argument can thus be used to show usual arguments show that~$\rho_{\xi,n}$ is smooth away from the corners of~$\gamma$. Further, if~$\gamma_\delta$ is the subset of~$\gamma$ at least a distance~$\delta$ away from the corners of~$\gamma$, then there is a~$K_{\xi,\delta}$ such that~$\|\rho_{\xi,n}\|_{L^\infty (\gamma_\delta)} \leq K_{\xi,\delta} \|\rho_{\xi,n}\|_{L^2 (\gamma)} + K_{\xi,\delta}\|h_{\xi,n}\|_{L^2(\gamma_\delta)} \leq  K_{\xi,\delta}(1+\|\cK_{\xi}^{-1} \|_{L^2(\gamma)})\|h_{\xi,n}\|_{L^2(\gamma)}$. The analyticity of~$S_{\gamma,\xi}$ will also imply that~$K_{\xi,\delta}$ is analytic in~$V_{\gamma,\epsilon}$. 

The logarithmic singularity of~$S_{\gamma,\xi}$ thus implies that there is~$\tilde K_{\xi,H,\delta}$ such that \begin{equation}
    \|S_{\gamma,\xi}[\rho_{\xi,n}]\|_{C(\bar \Omega_{H,\delta})}\leq \tilde K_{\xi,H,\delta} \|h_{\xi,n}\|_{L^2(\gamma)}.
\end{equation} The function~$\tilde K_{\xi,H,\delta}$ is similarly analytic and so can be bounded uniformly on~$V_{\gamma,\epsilon}$. We can also bound the~$h_{\xi,n}$'s by noting that
\begin{multline}
    |h_{\xi,n}(\bz)| = |(i\xi_n,-\alpha(\xi_n) \cdot \bn(\bz) |\frac{\left|e^{i\xi_nz_1-\alpha(\xi_n)z_2}\right|}{2|\alpha(\xi_n)|}\\
    \leq \sqrt{|\xi_n|^2+|\alpha(\xi_n)|^2}\frac{e^{-\Im \xi z_1}}{2|\alpha(\xi_n)|}\leq \frac{|\xi_n|}{2|\alpha(\xi_n)|}e^{-\Im \xi z_1}.
\end{multline}
The asymptotics of~$\alpha$ in Lemma~\ref{lem:alphaprop} thus tell us that~$\|h_{n,\xi}\|_{L^2(\gamma)}$ can be bounded independent of~$n$ and~$\xi\in V_{\gamma,\epsilon}$. We thus have that there is a~$D_{H,\delta}>0$ such that~$|S_{\gamma,\xi}[\rho_{n,\xi}](\bx)|\leq D_{H,\delta}$ for all~$\xi\in V_{\gamma,\epsilon}$ and~$\bx\in \Omega_{H,\delta}$.

Plugging this estimate into~\eqref{eq:f2n_exp} gives
\begin{equation}
    |\partial_{y_1}^lw_n(\bx,\by)| \leq D_{H,\delta} \max_c|\xi|^l\int_c |e^{\alpha(\xi_n)y_2-i\xi_n y_1}|\,\opd \xi \leq D_{l,H,\delta}|c|e^{h|y_1|}\max_{\xi\in V_{\gamma,\epsilon}} e^{\Re (\alpha(\xi_n) y_2)}
\end{equation}
for some constants~$ D_{l,H,\delta}$
We observed in Lemma~\ref{lem:nonzero_nm} that~$\Re (\alpha(\xi_n) y_2)\leq \eta_n \Re y_2$. Thus
    \begin{equation}
        |\partial_{y_1}^lw_n(\bx,\by)| \leq \tilde D_{l,H,\delta}e^{\eta_n \Re y_2+h|y_1|} 
    \end{equation}
    for some constants~$\tilde D_{l,H,\delta}$.
    The bound~\eqref{eq:wnearfar_n} then follows from the same argument that was used to derive~\eqref{eq:wnm_non_sum}.
\end{proof}

\begin{lemma}
    For~$l\geq 0$ there is a continuous function~$b_{l}(\bx,y_1)$ and constant~$C_{l,H,\delta}$, we have
    \begin{equation}
        \left|\partial_{y_1}^lw_0(\bx,\by) - \frac{b_{l}(\bx,y_1)e^{iky_2}}{  y_2^{\opnm{ceil}(l/2)+1/2}}\right| \leq C_{l,H,\delta}\frac{1}{|y_2|^{(l+3)/2}}e^{-k\Im y_2+h|y_1|}+ C_{l,H,\delta} e^{\eta_1\Re y_2} \label{eq:wnearfar_0}
    \end{equation}
    for~$\bx\in \Omega_{H,\delta}$ and~$\by\in \Omega_\bbC\setminus \Omega_{d/2}$.
\end{lemma}
\begin{proof}
    We split the contour~$c$ into the piece~$c_0,c_-$, and~$c_+$ from the proof of Lemma~\ref{lem:00_bound}. The integrals over~$c_\pm$ can be bounded in the same manner as the previous lemma. The integral over~$c_0$ is
    \begin{multline}
        \partial_{x_1}^lw_{0,\opnm{close}}(\bx,\by) = \int_{c_0} (i\xi)^le^{\alpha(\xi)y_2-i\xi y_1}S_{\gamma,\xi}\left[ \rho_{0,\xi}\right](\bx) \,\opd \xi\\
        = i^l\int_{-t_0}^{t_0} \xi^le^{-ke^{-i\pi/4} t^2 \Re y_2} e^{-i\xi_0(t) y_1} S_{\gamma,\xi_0(t)}\left[ \rho_{0,\xi_0(t)}\right](\bx) \partial_t \xi_0(t)\, \opd t.
    \end{multline}
    Repeating the arguments from Lemma~\ref{lem:near2far_non} and using that~$S_{\gamma\xi}[\rho_{0,\xi}](\bx)$ is bounded and smooth will give that
    \begin{equation}
        \left|\partial_{x_1}^lw_{0\opnm{close}}(\bx,\by)-\frac{b_{l}(\bx,y_1)e^{iky_2}}{  y_2^{\opnm{ceil}(l/2)+1/2}}\right| \leq C_{l,\delta} \frac{e^{-k\Im y_2}}{|y_2|^{\opnm{ceil}(l/2)+1/2}}.
    \end{equation}

    The integrals over~$c_\pm$ can be bounded in the same way as they were in Lemma~\ref{lem:00_bound}. Adding the bounds gives the result.
\end{proof}
\begin{remark}
    The function~$b_l$ is related to the function~$a_l$ from Proposition~\ref{lem:0n_estimate} and the constant $A_L$ from Lemma~\ref{lem:00_bound}. It shall turn out, however, that we don't need to explicitly characterize this relationship.
\end{remark}

Adding the above bounds gives the following result.
\begin{lemma}\label{lem:far2near}
For~$l\geq 0$ there is a constant~$C_{l,\delta}$ such that
    \begin{equation}
        \left|\partial_{y_1}^l w(\bx,\by)-\frac{b_{l}(\bx,y_1)e^{iky_2}}{  y_2^{\opnm{ceil}(l/2)+1/2}}\right| =C_{l,\delta}\frac{1}{|y_2|^{(l+3)/2}}e^{-k\Im y_2+h|y_1|}+ C_{l,\delta} e^{\eta_1\Re y_2}
    \end{equation}
    for~$\bx\in \Omega_{d,\delta}$ and~$\by\in \Omega_\bbC\setminus \Omega_d$.
\end{lemma}
By symmetry, the same result clearly holds with~$\bx$ and~$\by$ swapped.

\section[Proof of Proposition 2]{Proof of Proposition~\ref{prop:asym_rho}}\label{app:som_ang}
In this appendix, we derive the asymptotic formula~\eqref{eq:v_sommerfeld}. We focus on proving it for~$v_L$ as the proof for~$v_R$ is identical. As we did in Appendix~\ref{app:far2near_as}, we do this by using the expansion~\eqref{eq:dualsum} and work one term at a time. Specifically, we let~$v_{\xi} =S_{\xi,\gamma_L}[\tilde \rho_\xi]$. We then write 
\begin{equation}
        v_{\xi }(\bx) = \sum_{n=-\infty}^\infty v_{\xi,n}(\bx),
    \end{equation}
    where
    \begin{equation}
        v_{\xi,n}(\bx) = e^{\alpha(\xi_n)x_2 +i\xi_n x_1}\int_{\gamma_L} \frac{e^{-\alpha(\xi_n)z_2 -i\xi_n z_1}}{-2\alpha(\xi_n)}\tilde \rho_\xi(\bz)\, \opd \bz = f_n(\xi) e^{\alpha(\xi_n)x_2 +i\xi_n x_1}.
    \end{equation}
    We can easily bound the~$f_n$'s  by noting that
    \begin{equation}
        |f_n(\xi)| \leq \|\tilde \rho_\xi\|_{L^2(\gamma_L)} \left\|  \frac{e^{-\alpha(\xi_n)z_2 -i\xi_n z_1}}{-2\alpha(\xi_n)}\right\|_{L^2(\gamma_L)}\leq \frac{1}{2|\alpha(\xi_n)|}\|\tilde \rho_\xi\|_{L^2(\gamma_L)} \left\| e^{ -i\xi z_1}\right\|_{L^2(\gamma_L)}.
    \end{equation}
    The analyticity of the right hand side then tells us that the functions $f_n(\xi)$ will be uniformly bounded by some~$F>0$ on~$V_{\gamma,\epsilon}$. 
    
    We can similarly expand~$v$ as a sum of the functions
    \begin{equation}
        v_n(\bx) = \int_c v_{\xi,n}(\bx)\, \opd \xi= \int_c f_n(\xi) e^{\alpha(\xi_n)x_2 +i\xi_n x_1}\, \opd \xi.
    \end{equation}
    Along a ray~$\bx = r\hat\theta$ we have
    \begin{equation}
        v_n(r\hat\theta) = \int_c f_n(\xi) e^{(\alpha(\xi_n)\sin\theta +i\xi_n \cos\theta)r}\, \opd \xi=\int_c f_n(\xi) e^{g(\xi_n,\theta)r}\, \opd \xi,
    \end{equation}
    where
        \begin{equation}
         g(\xi,\theta) = \alpha(\xi) \sin\theta  + i\xi \cos\theta.
    \end{equation}
    
    In order to verify the Sommerfeld radiation condition, we study
    \begin{equation}
         (\partial_r-ik)v_n(r\hat\theta) =\int_c f_n(\xi) (g(\xi_n,\theta)-ik)e^{g(\xi_n,\theta)r}\, \opd \xi
    \end{equation}
We can thus bound the~$n\neq0$ terms by
    \begin{multline}
        |(\partial_r-ik)v_n(r\hat\theta)| \leq |F|\int_c |g(\xi_n,\theta)-ik|e^{\Re g(\xi_n,\theta)r} \, \opd \xi \\
        \leq C(1+2\pi |n|/d)\int_c e^{\Re (\alpha(\xi_n)\sin\theta +i\xi \cos\theta)r} \, \opd \xi.
    \end{multline}

    \begin{lemma}\label{lem:vn_asym}
        Under the assumptions of Proposition~\ref{prop:asym_rho}, there is a function~$c_{L}(\theta)$ such that
        \begin{equation}
            \left|(\partial_r-ik)\sum_{n\neq 0}v_n(r\hat\theta)\right| \leq c_{L}(\theta) e^{-\eta r\sin\theta}.
        \end{equation}
        Further the function is continuous on~$(0,\pi)$.
    \end{lemma}
\begin{proof}
    
    If we choose~$c$ to be a contour in~$V_{\gamma_L,\epsilon}$ with~$|\Im \xi|\leq 2\epsilon$, then along any ray we can bound the integral by
    \begin{equation}
        |(\partial_r-ik)v_n(r\hat\theta)| \leq C|c|(1+2\pi |n|/d) e^{\epsilon|\cos\theta| r}\max_{\xi\in c} e^{r \sin\theta\Re \alpha(\xi_n)}.
    \end{equation}
    Using the same ideas as in the proof of Lemma~\ref{lem:nonzero_nm}, we can see that
    \begin{multline}
        \max_{\xi\in c} \Re \alpha(\xi_n) \leq  \alpha\lp(2|n|-1)\frac{\pi}d \rp \cos\lp\frac12\opnm{Arg}(\xi_n^2-k^2) \rp \\
        \leq   \alpha\lp(2|n|-1)\frac{\pi}d \rp \cos\lp \frac12 \arctan \lp \left.\epsilon  \middle/ \frac{(2|n|-1)\pi}d \right.\rp\rp \leq \alpha\lp(2|n|-1)\frac{\pi}d \rp(1- K \epsilon^2).
    \end{multline}
    Using the monotonicity of~$\alpha$, we can replace the~$K\epsilon^2$ term to get the simpler bound
    \begin{equation}
        \max_{\xi\in c} \Re \alpha(\xi_n) \leq  \alpha\lp(2|n|-1)\frac{\pi}d\rp + \tilde K \epsilon.
    \end{equation}
    If we sum over~$n$, we get
\begin{equation}
    \left|(\partial_r-ik)\sum_{n\neq 0}v_n(r\hat\theta)\right| \leq C|c| e^{\epsilon|\cos \theta|  r} e^{r\sin\theta \tilde K\epsilon}\sum_{n\neq 0} (1+2\pi |n|/d) e^{r\sin \theta  \alpha\lp(2|n|-1)\frac{\pi}d \rp}.
\end{equation}
Repeating the argument from Lemma~\ref{lem:nonzero_nm} gives
\begin{equation}
    \left|(\partial_r-ik)\sum_{n\neq 0}v_n(r\hat\theta)\right| \leq \frac{D e^{\epsilon|\cos \theta|  r} e^{r\sin\theta (\alpha(\pi/d)+ \tilde K\epsilon)}e^{\partial_\xi \alpha\lp \frac\pi d\rp \frac{2\pi}d r\sin \theta}}{\lp 1-e^{\partial_\xi \alpha\lp \frac\pi d\rp \frac{2\pi}d r\sin \theta}\rp^2}.
\end{equation}
    If~$\theta$ is far enough from horizontal that
    \begin{equation}
        \epsilon |\cos \theta| + \cos\lp \frac12 \arctan \lp \left.\epsilon  \middle/ \frac{(2|n|-1)\pi}d \right.\rp\rp \leq \alpha\lp(2|n|-1)\frac{\pi}d \rp(1- K \epsilon^2)\sin\theta\leq 0
    \end{equation}
    then the numerator will decay exponentially along the ray. Similarly, since~$r>1$ the denominator can be bounded away from zero. For $\theta$ closer to horizontal, we can repeat the argument with smaller~$\epsilon$, though the constants will blow up as~$\theta \to 0,\pi$.

\end{proof}

The previous lemma bounds the behavior of~$v_n$ for all~$n\neq 0$. What remains is~$(\partial_r -ik)v_0$, which is given as the integral
    \begin{equation}\label{eq:v_0_somm_start}
        (\partial_r -ik)v_0(\bx) = \int_c (g(\xi,\theta)-ik)e^{r g(\xi,\theta)}f_0(\xi) \,\opd \xi.
    \end{equation}
    To bound $v_0$, we begin by looking for its steepest descent contour. For a fixed~$\theta$, it was observed in~\cite{epstein2023solvingb} that the stationary point of the integral (i.e. the point where $\partial_\xi g(\xi_*(\theta),\theta) = 0$) will be
    \begin{equation}
        \xi_*(\theta) = k\cos\theta.
    \end{equation}
    At the stationary point, we have~$g(\xi_*(\theta),\theta)=ik$, and so the integrand in \eqref{eq:v_0_somm_start} vanishes there. To bound the integral, we look for a descent contour $\xi(t;\theta)$ such that
    \begin{equation}
        g(\xi(t;\theta),\theta) = g(\xi_*(\theta),\theta) -kt^2 =k(i-t^2).\label{eq:ang_steep}
    \end{equation}
    To find the contour, we let~$\xi(t;\theta) = k\sin \phi(t;\theta)$ and note that~$\alpha(k\sin\phi(t;\theta)) = ik\cos \phi(t;\theta)$, at least for~$t$ small enough to avoid any branch cuts. After this substitution, \eqref{eq:ang_steep} becomes
    \begin{equation}
        ik \lp \cos\phi(t;\theta) \sin \theta + \sin \phi(t;\theta) \cos \theta \rp =  k(i - t^2).
    \end{equation}
    An addition of angle formula gives
    \begin{equation}
          \sin(\theta +\phi(t;\theta)) = 1-\frac{t^2}{i},
    \end{equation}
    and so
    \begin{equation}
        \phi(t;\theta) = \arcsin\lp1+it^2\rp - \theta.
    \end{equation}
    After some simplification, this gives that
    \begin{equation}\label{eq:xi_ang_steep}
        \xi(t;\theta)
         = k\lp 1+it^2\rp\cos \theta + kt\sqrt{ t^2 -2i}\sin \theta .
    \end{equation}
    Since this formula is analytic for real~$t$, this~$\xi(t;\theta)$ will satisfy \eqref{eq:ang_steep}, as long as it avoids the branch cuts of~$\alpha$. 
    
    To see that it does so, we note that and $
        \opnm{sign} \Im \xi(t;\pi/2) = -\opnm{sign} t.$ Thus~$\xi(t;\pi/2)$ avoids the branch cuts.
    It is also clear from \eqref{eq:ang_steep} and the definition of~$g$ that~$\xi(t;\theta)$ only crosses the real axis when~$\xi(t;\theta) = k\cos\theta$ or~$k/\cos\theta$. Thus the continuity of~$\xi(t;\theta)$ implies that it avoids the branch cuts of~$\alpha$, and so \eqref{eq:xi_ang_steep} applies for all~$t$ and all~$\theta\in(0,\pi)$.
\begin{lemma}
    Let~$c_0(\theta)\subset V_{\gamma_L,\epsilon}$ be a portion of the steepest descent contour parameterized by~\eqref{eq:xi_ang_steep} with a truncation that is a continuous function of angle. Under the assumptions of Proposition~\ref{prop:asym_rho}, there is a continuous function~$b_{L}$ such that 
\begin{equation}
    v_{00}(r,\theta):=\int_{c_0(\theta)} e^{rg(\xi,\theta)} f_0(\xi)\, \opd \xi
\end{equation}
    satisfies
        \begin{equation}\label{eq:v00_sommerfeld}
        \left|(\partial_r-ik)v_{00}(r\hat\theta) \right| \leq \frac{b_{L}(\theta)}{r^{3/2}}.
    \end{equation}
\end{lemma}
\begin{proof}
     Along this contour, we have
    \begin{equation}
        (\partial_r-ik)v_{00}(r,\theta)= \int_{t_{0-}}^{t_{0+}}(- kt^2)e ^{ikr - rkt^2} f_0(\xi(t;\theta))\partial_t \xi(t;\theta) \, \opd t,
    \end{equation}
    where~$t_{0\pm}$ are the endpoints of~$c_0(\theta)$.

    Since the stationary point and descent contour are in~$V_{\gamma_L,\epsilon}$, this integral can be analyzed using Laplace's method (see Proposition~\ref{prop:0n_close}). The resulting formula will show that
    \begin{equation}
        \left|(\partial_r-ik)v_{00}\right| \leq \frac{b(\theta)}{r^{3/2}}.
    \end{equation}
    By taking smaller and smaller $\epsilon$, we can find this~$b_{L}(\theta)$ for all $\theta\in (0,\pi),$ though it may blow up as~$\theta\to0,\pi$. As the contour is a smooth function of~$\theta$, the function $b_L$ will also be smooth.
\end{proof}

\begin{lemma}\label{lem:v0_asym}
Under the assumptions of Proposition~\ref{prop:asym_rho}, there are functions~$b_{L}$ and~$c_{L}$ such that 
        \begin{equation}\label{eq:v0_sommerfeld}
        \left|(\partial_r-ik)v_{0}(r\hat\theta) \right| \leq \frac{b_{L}(\theta)}{r^{3/2}}+c_{L}(\theta)e^{-\sin(\theta)\tilde\eta_L(\theta) r},
    \end{equation}
    where~$\tilde \eta_L(\theta)$ is a positive function that is equal to $|\alpha(\pi/d)|$ for $|\cos\theta| \geq k\pi/d$ and goes to zero as~$\sin\theta\to 0$. Further $a_L,b_L,$ and~$\tilde\eta_L$ are continuous on~$(0,\pi)$.
\end{lemma}
\begin{proof}
    If~$\theta\approx\pi/2$, we shall split the integral defining~$v_0$ into an integral over the piece of the steepest descent contour~$c_0(\theta)$ with $|\Re \xi(t,\theta)|<\frac{\pi}{d}$ and the two vertical line segments~$c_{\pm}(\theta)$ that connect~$c_0(\theta)$ to the points~$\xi =\pm \pi/d$. If~$|\cos \theta|<k\pi/d,$ then $c_0(\theta)$ may go on the wrong side of the poles~$\pm\tilde \xi_j$. To avoid this, we truncate~$c_0$ at the point that~$|\Im \xi(t;\theta)| = \epsilon$ but~$|\Re \xi(t;\theta)|>k$. We then replace the contour~$c_+(\theta)$ by a contour~$\tilde c(\theta)\subset V_{\gamma_L,\epsilon}$ that connects this point to~$\xi=(\opnm{sign}\cos\theta) \pi/d$ with imaginary part within~$2\epsilon$ of zero.

    The previous lemma bounds the integral over~$c_0(\theta)$.
    To understand the behavior of~$\alpha$ on~$c_-$, we look at the curve $\alpha\lp - \frac \pi d + it\rp$ in the complex plane for~$t\in\bbR$. Indeed the square of this curve is given by
\begin{equation}
   \lp \alpha\lp - \frac \pi d + it\rp \rp^2 =  \lp - \frac \pi d + it \rp^2 - k^2 = \frac{\pi^2}{d^2} -t^2 - k^2 - 2i \frac\pi d,
\end{equation}
which is to the right of the curve
\begin{equation}
     \lp \alpha\lp -\frac\pi d\rp + it\rp^2 =  \frac{\pi^2}{d^2} -t^2 - k^2 - 2it \sqrt{\lp \frac\pi d \rp^2 - k^2}.
\end{equation}
Our choice of branch cut and the mapping properties of the square root therefore give that
\begin{equation}
    \Re \alpha\lp - \frac \pi d + it\rp  \leq \Re \alpha\lp -\frac\pi d\rp.
\end{equation}
We can repeat the same argument for~$\alpha(\pi/d+it)$ to see that
\begin{equation}
    \Re \alpha(\xi) \leq \Re \alpha(\pm \pi/d)= \Re \alpha( \pi/d)
\end{equation}
for all~$\xi\in c_\pm$.

    For simplicity, we shall assume that~$\cos\theta \geq 0$. In this case, it is easy to bound the integral over~$c_-(\theta)$ by using the fact that
    \begin{equation}
        \Re g(\xi,\theta) \leq \Re \alpha(\xi) \sin\theta \leq \alpha(\pi/d) \sin\theta . 
    \end{equation}
    If we let
    \begin{equation}
        v_{0-}(r,\theta):= \int_{c_-(\theta)} e^{rg(\xi,\theta)}f_0(\xi) \, \opd \xi,
    \end{equation}
    then this observation allows us to bound
    \begin{equation}
        |(\partial_r - ik) v_{0-}(r,\theta)| \leq |c_-(\theta)| \max|f_0(\xi)| e^{r\alpha(\pi/d) \sin\theta}.
    \end{equation}

The integral over~$c_+$ is more subtle because~$\Re \xi \cos\theta>0$.  In order to bound it, we must control the extent of~$c_+$. To do this, we note 
\begin{equation}
    \Im t\sqrt{t^2-2i} > -1,
\end{equation}
so~$\Im \xi(t;\theta) > -k \sin\theta$, since~$\cos\theta\geq0$. To bound the integral over~$c_+$, we note that it is only required for $\cos\theta < \frac{kd}{\pi}$.  Under these conditions, we have that
\begin{multline}
    \max_{\xi\in c_+(\theta)}\Re g(\xi,\theta) \leq \frac{kd}{\pi}\max_{\xi\in c_+(\theta)} \Im\xi - \sin\theta\min_{\xi\in c_+(\theta)}  \Re \sqrt{\xi^2-k^2}\\
    = \sin\theta\lp \frac{k^2d}{\pi} - \sqrt{\lp\frac{\pi}d\rp^2-k^2}\rp
    = \sin\theta \frac{\pi}d\lp \frac{k^2d^2}{\pi^2} - \sqrt{1-\lp \frac{kd}{\pi}\rp^2}\rp\leq  -\sin\theta\tilde\eta,
\end{multline}
where
\begin{equation}
    \tilde\eta = \min \left[ -\frac{\pi}d\lp \frac{k^2d^2}{\pi^2} - \sqrt{1-\lp \frac{kd}{\pi}\rp^2}\rp,\-\alpha\lp \frac\pi d\rp \right].
\end{equation}
This constant is positive provided
\begin{equation}
    \frac{k^2d^2}{\pi^2} < \frac12 (\sqrt{5}-1) ,
\end{equation}
which is equivalent to the requirement
\begin{equation}
    k < \sqrt{\frac12 (\sqrt{5}-1) }\frac{\pi} d \approx 0.79 \frac{\pi}d.
\end{equation}
Using these estimate, it is easy to see that 
    \begin{equation}
        |(\partial_r - ik) v_{0+}(r,\theta)| \leq |c_+(\theta)| \max|f_0(\xi)| e^{r\tilde\eta \sin\theta}.
    \end{equation}
    where
        \begin{equation}
        v_{0+}(r,\theta):= \int_{c_+(\theta)} e^{rg(\xi,\theta)}f_0(\xi) \, \opd \xi.
    \end{equation}
The integral over~$\tilde c_+(\theta)$ can be bounded in the same manner as the integral in Lemma~\ref{lem:vn_asym}, except that now~$\Re\alpha(\xi)$ is bounded by~$\tilde\eta_L(\theta)=\Re\alpha(k/\cos\theta -\epsilon)$, which will in general be smaller than~$\eta$. The properties of~$\tilde\eta_L$ follow from the properties of~$\alpha$.
\end{proof}
\begin{remark}
    In general, the behavior of~$v_0(r\hat\theta)$ for~$\theta \approx 0,\pi$ will be better behaved than the previous lemma would indicate, since the above analysis picks contours that avoids dealing with with poles ($\pm\tilde\xi_j$) and branch cuts ($\pm k)$ of~$f_0$. If the residues of~$f_0$ were known, then we could push the steepest descent contour out to the boundary of the Brillouin zone and replace $\tilde\eta_L(\theta)$ by~$\alpha(\pi/d)$ or~$\alpha(\tilde\xi_j)$, depending on the angle. A better understanding of $f_0$ at the branch cuts would allow us to control $b_L(\theta)$ as $\theta\to0,\pi$.
\end{remark}
Proposition~\ref{prop:asym_rho} follows directly from Lemmas~\ref{lem:vn_asym} and \ref{lem:v0_asym}.

\end{document}